\documentclass{amsart}
\usepackage[utf8]{inputenc}
\usepackage{amssymb}
\usepackage{amsmath,amsthm}
\usepackage{color,fullpage}
\usepackage{hyperref}
\usepackage{comment}
\usepackage{cleveref}
\usepackage{enumitem}
\usepackage{ulem}

\newtheorem{theorem}{Theorem}[section]
\newtheorem{proposition}[theorem]{Proposition}

\newtheorem{remark}[theorem]{Remark}
\newtheorem{lemma}[theorem]{Lemma}
\newtheorem{corollary}[theorem]{Corollary}

\newcommand\md{\textrm{\upshape mod }}
\numberwithin{equation}{section}

\newcommand{\floor}[1]{\left\lfloor #1 \right\rfloor}

\title{Explicit estimates for the Goldbach summatory functions}
\author[Gautami Bhowmik]{Gautami Bhowmik}
\address{Laboratoire Paul Painlev\'e LABEX-C2EMPI , Universit\'e de Lille, Batiment M2, 59655 Villeneuve-d'Ascq, cedex France}
\email{gautami.bhowmik@univ-lille.fr}
\author[A.-M. Ernvall-Hyt\"onen]{Anne-Maria Ernvall-Hyt\"onen}
\address{Mathematics and Statistics, P. O 68, 00014 University of Helsinki, Finland}
\email{anne-maria.ernvall-hytonen@helsinki.fi}
\author[Neea Paloj\"arvi]{Neea Paloj\"arvi}
\address{School of Science, The University of New South Wales (Canberra), ACT, Australia}
\email{n.palojarvi@unsw.edu.au}

\begin{document}

\begin{abstract} 
In order to study the analytic properties of the Goldbach generating function  we 
consider a smooth version, similar to the Chebyshev function for the Prime Number 
Theorem. In this paper, we obtain 
explicit numerical estimates for the average order of
its summatory function
both in the classical case and in arithmetic progressions. In addition, we derive new
explicit estimates for sums over zeros and for the function $\psi(u,\chi)$. Our results 
 are  general and describe how the explicit bounds depend on other known explicit 
 estimates.
 These support the known asymptotic results under the (Generalised) Riemann 
 Hypothesis involving error terms.
\end{abstract}

\maketitle

\section{Introduction}
A  question posed in 1742,  known  in its contemporary form as the Goldbach conjecture, asks if every even integer greater than $2$ can be expressed 
as the sum of two prime numbers. Almost three centuries later, the hypothesis remains unproved though it 
is known to be statistically true and is empirically supported by calculations for all numbers  
up to \(4\cdot 10^{18}\) \cite{OHP2014}.  

For analytic study, rather than showing that the Goldbach function
\begin{equation*}
    g(n)=\sum_{\substack{p_1+p_2=n \\ p_i \text{ prime}}} 1
\end{equation*}
is positive for all even $n$, 
it is easier to handle the smoother form over powers of prime
\begin{equation}
\label{def:Gn}
G(n)=\sum_{\ell+m=n}\Lambda(\ell)\Lambda(m),
\end{equation}
where \(\Lambda\) is the von Mangoldt function defined to be \(\log p\)  for a positive power of a prime number $p$ and defined to vanish elsewhere.
Hence \(g(n)\) can be recovered easily from $G(n)$ by partial summation and if \(G(n)\)  could be shown to be sufficiently large (see e.g. \cite[Section 2]{Granville2008}), the Goldbach conjecture would be true. However this seems to not yet be within our reach 
though we do know explicit upper bounds for $g(n)$. \cite{DGNP1993}.

Further, Hardy and Littlewood \cite{HL23} conjectured that
\begin{equation}\label{conj:HL}
g(n)\sim \frac{2n}{(\log n)^2}C_2\prod_{\substack{p\mid n\\ p>2}}\frac{p-1}{p-2},
\end{equation}

where \(C_2=2\prod_{p>2}(1-\frac{1}{(p-1)^2})\) is the twin-prime constant.

It is natural in analytic number theory to obtain information on erratic arithmetic functions via the better-behaved problem on a partial sum, 
up to some cut-off value. Consider thus  
the \textit{Goldbach summatory function}
\[ 
S(x)=\sum_{n\le x}G(n).
\]
 An asymptotic of the form
\[
\sum_{n\le x}g(n)\sim \frac{x^{2}}{2\log^{2}x}
\]
was  known since Landau \cite{L1900} though more precise work on \(S(x)\), begun by Fujii \cite{F1991_1} almost a century later, 
required information on zeros of the Riemann zeta function due to Gallagher \cite{Gallagher1989}. 
The congruence Goldbach function, defined for \( a, b \)  positive integers coprime to \(q\),  as
\[
G(n;q,a,b)=\sum_{\substack{\ell+m=n\\\ell\equiv a,\,m\equiv b\,(\md q)}}\Lambda(\ell)\Lambda(m)
\]
together with its summatory function \(S(x; q,a,b)\), has been studied more recently (see \cite{BHMS2019}).

 We notice that this situation is analogous to the very classical question counting primes. Let \(\pi (x) :=\sum_{p\le x} 1\) denote the prime counting function and let \(\psi \) be the Chebyshev function, indeed
\begin{equation}
\label{eq:defpsi}
{\displaystyle \psi (x):=\!\!\!\!\sum _{\stackrel {p^{k}\leq x,}{p{\text{ is prime}}}}\!\!\!\!\log p=\sum_{n\le x}\Lambda(n)\;}.
\end{equation}
Now, we note that $g(2N) \leq 2\pi(2N)-\pi(nN)-\pi(N-1)$. Similarly, the arithmetic progression situation is connected to 
 the primes in arithmetic progression, whose counting function, for $a$ coprime to $q$, is defined as 
\[\pi(x; q, a):=\sum_{\substack{p\le x\\p\equiv a(\md q)}}1,\]
their legendary asymptotics being 
\[
\pi(x)\sim \frac{x}{\log x} ,\  \psi(x)\sim x\  \text {and}\  \pi(x; q, a)\sim \frac{1}{\varphi(q)}\frac{x}{\log(x)}
\] respectively.

To obtain precise analytic information on functions involving prime numbers, we encounter the  Dirichlet $L$-functions, including the Riemann zeta function.
For the summatory functions it is possible to obtain an expression involving a dominant term, which is computed knowing the order of specific zeros at central points of the associated 
\(L\)-functions and to separate 
a term containing the trivial zeros  from an oscillatory term 
involving the infinitely many non-trivial zeros which can not be studied without assuming some conjectures. 
And since  many 
plausible hypotheses on zeros of $L$-functions  are out of reach at the
moment,  it is difficult to treat average orders of the Goldbach
functions unconditionally.

Perhaps  the
most widely known of such conjectures is the  \textit {Generalized Riemann Hypothesis} (GRH) and the weaker one on
\textit {Siegel zeros}. If $\chi$ is a Dirichlet character modulo $q$  and $L(s, \chi)$ 
the associated Dirichlet $L$-function, the real part of its  non-trivial 
zeros are expected, according to the GRH, to lie on the line $\Re (s)=\frac{1}{2} $. 
The existence of a Landau-Siegel zero  would serve as  a counterexample to
 the GRH. Such an eventual zero, let us denote it by 
$\beta_1$, would be a real one associated to a unique primitive 
quadratic Dirichlet character of conductor $q_\chi$ with $\beta_1=1- \frac{1}{c(\epsilon)\log q}$ for all $\epsilon >0$.
Siegel's theorem of 1930 asserts that $c(\epsilon) \ll_{\epsilon} q^{\epsilon}$.
Unfortunately $c(\epsilon)$ cannot be computed effectively
for any \( \epsilon < 1/2 \). There is abundant literature on this topic, for example \cite{MV2007} could be an entry point.

Then there are the folkloric
assumptions on \textit {Linear Independence of zeros} which states that the imaginary parts
of the non-trivial zeros of \(L\)-functions are linearly independent over the rationals and the \textit {Non-Coincidence of zeros} (see \cite[p. 353]{Conrey2003}) which
expects that different primitive Dirichlet characters with the same modulus do not 
have a common non-trivial zero with the same multiplicity. If we relax the 
condition on multiplicity we get the 
\textit {Distinct Zero conjecture} which predicts that for any given conductor any two
 distinct Dirichlet \(L\)-functions associated to it do not have a common 
 non-trivial zero, except for a possible multiple zero at \( s=1/2\) (see \cite{BHMS2019}). 
 
 As an example of a conditional Goldbach result, we note that  assuming the \textit {Riemann Hypothesis} (RH),  the average
order
of \(G(n)\)
 can  be written as 
\[
 S(x)= \frac{x^2}{2}-H(x)+O(x\log^3 x)
\]
where the oscillating term
 \(H(x)=\sum_{\rho}\frac{x^{\rho+1}}{\rho(\rho+1)}\)
is a sum over \(\rho\), the
non-trivial zeros of the Riemann zeta function \cite{GY2017, LZ2012},
whereas an unconditional form
\[
S(x)= \frac{x^2}{2}+O( x^{3/2+\epsilon} )
\]
for all positive \( \epsilon\) would already imply the RH \cite{Granville2008}.
See also a more detailed formula for the function $S(x)$ obtained
by Goldston and Suriajaya \cite{GS2023}. 

Similarly for the congruence Goldbach function, for \(a,b\) coprime to $q$, we state the following simplified summatory expression without 
the oscillating term
\begin{equation}
\label{eq:SBBound}
S(x;q,a,b)= \frac{x^2}{2\varphi(q)^2} +O(x^{1+B_q})
\end{equation}
where $B_q$ depends on the non-trivial zeros of the associated Dirichlet $L$-functions. In her very recent pre-print \cite{N2024}, Nguyen generalized the type \eqref{eq:SBBound} of result to the case where each of the summands is congruent to a different modulus.

On the other hand,
for $a$ co-prime to an odd $q$ , if the  \textit {Distinct Zero Conjecture} is assumed be true,
 the asymptotic formula
\begin{equation*}
S(x;q,a,a)=\frac{x^2}{2\varphi(q)^2}+O_{q,\varepsilon}(x^{3/2+\varepsilon})
\end{equation*}
for any $\varepsilon>0$
is equivalent to the GRH for the functions $L(s,\chi)$ with any character $\chi$ mod $q$ \cite{BHMS2019}. As far as we know there are no  explicit estimates for $S(x)$ or $S(x;q,a,b)$ in  literature.

In this paper we study explicit versions of the above asymptotics. Our results 
are unconditional in the sense that we do not assume GRH, DZC or any other similar 
conjecture. The details are described  in Section \ref{sec:mainresults}.
The estimates are divided in two different categories:  completely explicit results
and general results that can be numerically improved if certain other explicit
assumptions are improved.  In addition, our general results indicate
the terms that  affect our error terms most (see Remark 
\ref{rmk:dependence}.) In Section \ref{sec:preliminarylemmas}, we derive the main 
lemmas to prove the general results on the lines of the papers \cite{BHMS2019, 
Granville2008}. The general theorems are proved in Sections \ref{sec:FirstGeneral} 
and \ref{sec:SecondGeneral}. The explicit versions of the most of the assumptions that we used
are described in Section \ref{sec:ExplicitAssumptions} containing estimates for
different sums over zeros. In Section \ref{sec:explicitA5}
we provide new explicit estimates for the function $\psi(u,\chi)$. 
Finally, the main explicit results are proved in Section \ref{sec:ProofsMain}.

\section {Related  explicit results}
We notice that many of the results above have implied constants, some of these can be determined, and are called \textit{effective} others 
which can not be determined and called 
\textit{ineffective}. Siegel's Theorem, cited above, is a classical 
example that involves the second type of constant and Goldbach sums that 
depend on
this theorem also give rise 
to such constants. The numbers behind the effective constants are often 
not specified and when they are, we obtain \textit{explicit} numeric 
values, which answer questions like 
which cut-off points can be used.  As a prototype, note that the 
original form of the Prime Number Theorem \[ \pi(x)\sim \frac{x}{\log x} \] of 1896.
Schoenfeld \cite{SchoenfeldSharperRH} showed that, numerically, 
under the RH
\[ |\pi(x)- \text{li}(x)|  \le \frac{1}{\sqrt{8\pi}}\sqrt x \log x,\]
where $\text{li}(x):=\lim_{\epsilon\to 0+}\int_0^{1-\epsilon}\frac{dt}{\log{t}}+\int_{1+\epsilon}^x\frac{dt}{\log{t}}$, for all \( x\geq 2657\).

Though explicit results around Goldbach functions  have been  
little investigated,
one such consideration  is that
of  Fujii \cite{F1991_3} who proved that under the linear 
independence conjecture on 
the first 70 zeros on the critical line of the Riemann zeta function the
oscillating term 
$$
R(x)=\Re\sum_{\gamma_{\chi}>0}\frac{x^{i\gamma_{\chi}}}{(1/2+i\gamma_{\chi})(3/2+i\gamma_{\chi})}
$$
with the sum over ordinates of the non-trivial zeros of $\zeta(s)$, the inequalities
\[R(x) > 0.012, \quad R(x)< -0.012\] hold for an unbounded sequence of positive real numbers $x$.
More recently \cite{MT2022} Mossinghoff and Trudgian improved the above bounds to 
\[ R(x)< -0.022978,\quad R(x)>0.021030\]
 under the RH
and that these oscillations are almost optimal.
Though the authors assume linear independence of imaginary parts of zeros 
of $\zeta(s)$, in fact a weaker form called \textit{ N-independence} is
enough to establish large oscillations for sums of 
arithmetic functions. A similar classical example is that the 
Mertens' conjecture 
$M(x)=\sum_{n\leq x}\mu (n)\le x^{1/2}$ implies infinitely many 
linear dependences over the rationals among the ordinates of the zeros of the zeta function on the critical line in the upper half plane.

Moreover, Martin \cite{Ma2022} has recently provided numerical data on the exceptional sets of the Goldbach function in arithmetic progressions.

To study explicit results on the Goldbach functions, we could build upon existing explicit results on zeros of
$L$-functions. Here we quote some of these.

Languasco \cite{L2023}, recently obtained numerical constants for 
the region of possible existence of the Siegel zero for primes 
$4.10^5<q\le 10^7$ and for even quadratic Dirichlet characters.
Morrill and Trudgian \cite{MT2020} made explicit a method of Pintz to obtain
 $1-\frac{0.933}{\log q}<\beta_1<1$. In another paper \cite{BMOR2021} 
it is proved that $\beta_1\le 1-\frac{40}{\sqrt q \log^2 q },\ \  q\ge 3$.
At about the same time, Bordignon \cite{B2019,B2020} improved
such bounds using  explicit estimates for $ L'(\sigma,\chi) $ for
$\sigma$ close to unity.

Earlier computational work of Platt \cite {P2016}, who verified the GRH 
for primitive characters to height up to $\max(108/q, A.107/q)$  with $A$ 
depending on the parity of the conductor,  shows that there is no 
exceptional zero for primes $q\leq 4\cdot10^5$. 
These are in line with the bounds 
 $\beta_1>1-\frac{1}{1-R\log (\max (q, q|t|, 10)}$ with 
 $R = 5.60$. \cite{K2018}, or $R=9.46$ \cite{M1984}. 
 
 The Vinogradov-Korobov type
estimates of the zero-free regions of the Riemann zeta function
\[ 
\sigma>1- \frac{c}{\log|t|^{2/3}(\log \log |t|)^{1/3}}, 
\]  with a positive constant $c$, have 
been studied for a very long time.
It is possible to take $c = 1/53.989, |t| \ge 3 $, see \cite{Bellotti2024}. 
Very recently Khale \cite{K2024} established, for the first time 
a similar result for $L$-functions,  with a dependence on the
conductor, for  $|\Im(s)| \geq 10$, i.e. 
with explicit $c$ and $c_{q,1}$ in 
\begin{equation}
\label{eq:VinogradovZeroFree}
    \sigma>1- \frac{1}{c_{q,1}\log q + c\log|t|^{2/3}(\log \log |t|)^{1/3}}.
\end{equation} 

Another question relevant to the study of the explicit Goldbach problem
is that of recent explicit results obtained  on primes in arithmetic 
progressions obtained unconditionally that for all \( q\geq 4 \) and for \(x\) larger than
\(  \exp(8\sqrt q\log ^3 q)\) 
\[ \left|\pi(x; q,a)- \frac{\text{Li}(x)}{\varphi(q)}\right|  \le \frac{1}{160}\frac{x}{\log^2 x},\]
where $\text{Li}(x):=\int_2^x (\log{t})^{-1}\, dt$, in \cite{BMOR2018} and under the GRH
in \cite{GM2019}. 

In a  more general context, we can  describe the splitting of primes in a 
given Galois extension of a field. Thus, for example, the classical 
Chebotarev density 
theorem predicts such a density by computing the set of conjugacy 
classes of the associated group. 
Let $L/K$ a finite Galois extension with
Galois group $G$, let $C$ be a union of conjugacy classes of $G$
and let the function $\psi_C$ count the number of non-ramified prime ideals
of norm up to $x$ with a von Mangoldt type weight (compare to the definition \eqref{eq:defpsi}).
Now the Chebotarev density theorem states that
$\psi_C(x)\sim \frac{|C|}{|G|}x$ as $x$ tends to infinity.

Grenié and Molteni \cite{GM2019} have recently obtained an explicit 
conditional version for the above asymptotic form. They
show that, under the GRH, the density result can be expressed with the 
dependence on
the degree of extension $n_L$ and discriminant $\Delta_L$. Their results
applied to the \(q^{th}\)  cyclotomic Dedekind zeta function can give estimates for
primes in arithmetic progressions. Their auxiliary results 
with $\Delta_L \leq q^{\varphi(q)}$ could provide, conditional 
to the GRH, appropriate estimates in the Goldbach setting.

 Pintz' recent  paper on Goldbach problems  \cite{P2023} has a title that
 could be misleading  in this context since it does not contain {\it explicit} in the 
 sense of {\it numerical} results.

\section{Main results}
\label{sec:mainresults}
We provide two types of results. 
The estimates in Section \ref{sec:resultsExplicit} are 
completely explicit. In Section \ref{sec:resultsGeneral}, we provide 
effective results that can be explicitly computed when numerical values for
certain formulas are known. 
The former  are ready for use in an explicit setting whereas the effective ones
appear with a possibility for improvement whenever better estimates for related
functions become available. 

Before describing the results we introduce a bit of notation.
For $\rho_{\chi}$ being the non-trivial zeros of an $L$-function associated to
a Dirichlet character $\chi$ modulo $q$, let us define 
$B_{\chi}:=\sup \{\Re(\rho_\chi)\}$, $B_q:=\sup\{B_\chi : \chi \pmod q\}$ and
\begin{equation}
\label{eq:defBq}
    B_q^*(t):=\min\{B_q,1-\eta_q(t)\}, \text{ where } \eta_q(t):=\frac{1}{c_1\log{q}+c_2(\log{t})^{2/3}(\log{\log{t}})^{1/3}}.
\end{equation}
Here $c_1$ and $c_2$ are  as in assumption \ref{eq:ZeroFree}. 
Let  $W_{-1}(x)$ denote the Lambert W-function with the branch where 
$-1/e \leq x<0$. 
Further, let  $\varphi_1^*(q)$ denote the number of primitive and principal 
characters $\pmod q$. 

\subsection{Numerical results}
\label{sec:resultsExplicit}
First we present the results for $S(x)$.

\begin{theorem}
\label{thm:main1}
Assume that
\begin{equation*}
    \log{x} \geq \exp\left(-4W_{-1}\left(-\left(4\cdot 5^{3/4}\cdot 53.989^{3/4}\right)^{-1}\right)\right) \approx 1.7\cdot10^{13}.
\end{equation*}
Then
$\left|S(x)- \frac{x^2}{2}\right| \leq 6.794x^{B_1^*(x)+1}$.
\end{theorem}

The following theorem gives a little bit more precise description of the term $S(x)$.

\begin{theorem}
\label{thm:main2}
     Assume $\log{x} \geq 10^6$. Then
    \begin{equation*}
         \left|S(x)-\frac{x^2}{2}+2\sum_{|\rho_\zeta|<x}\frac{x^{\rho_\zeta+1}}{\rho_\zeta(\rho_\zeta+1)}\right|\leq  13.149(\log{x})^5x^{2B_1^*(x)},
    \end{equation*}
    where the sum runs over the non-trivial zeros of the Riemann zeta function.
\end{theorem}

Next we present similar numerical results in the case of arithmetic progressions.

\begin{theorem}
\label{thm:main3}
Assume that $x \geq q >4\cdot 10^5$, $q \not\equiv 2\pmod{4}$, $a$ and $b$ are integers with $(ab,q)=1$, 
 \begin{equation}
\label{eq:logqLowerNum}
    \log{x} \geq \exp\left(-4W_{-1}\left(-\left(4\cdot 5^{3/4}\left(36.75\left(\frac{1}{e}\right)^{2/3}+61.5\right)^{3/4}\right)^{-1}\right)\right) \approx 6.7\cdot10^{13}.
\end{equation}
 Then
    \begin{equation*}
    \left|S(x;q,a,b)- \frac{x^2}{2\varphi_1^*(q)^2}\right| \leq 5.805x^{B_q^*(x)+1}.
\end{equation*}
\end{theorem}

\begin{theorem}
\label{thm:main4}
    Assume that $x \geq q >4\cdot 10^5$, $\log{x}\geq 10^8$, $q \not\equiv 2\pmod{4}$, $a$ and $b$ are integers with $(ab,q)=1$.  Then
     \begin{multline*}
         \left|S(x;q,a,b)-\frac{x^2}{2\varphi_1^*(q)^2}+\frac{1}{\varphi_1^*(q)^2}\sum_{\chi \text{ primitive or principal } \pmod q}\left(\overline{\chi}(a)+\overline{\chi}(b)\right)\sum_{|\rho_\chi|<x}\frac{x^{\rho_\chi+1}}{\rho_\chi(\rho_\chi+1)}\right|\\
         \leq 7.246(\log{x})^5x^{2B_q^*(x)},
    \end{multline*}
    where the sum over zeros runs over non-trivial zeros of the function $L(s,\chi)$.
\end{theorem}

\subsection{General results}
\label{sec:resultsGeneral}
In this section, we provide general explicit results. We formulate the 
results,
under the assumptions described in Appendix \ref{appendix:assumptions} 
and 
use the functions $f_j$ described in Appendix \ref{appendix:f}. 

The first main results describe the sizes of the functions $S(x)$ and $S(x;q,a,b)$.

\begin{proposition}
\label{prop:firstMainS(x)}
 Assume that \ref{assumptionProductExceptional}, \ref{eq:rhoSecondUpper}, \ref{eq:defIntervalZeroGeneral} and \ref{eq:defzerosGeneral} hold with $T\geq T_1$ and  \ref{assumption:psiuchi}, \ref{eq:assumption1OverRhoSum}, \ref{eq:assumptionrhorho1} and \ref{eq:assumptionLambda} hold with $T \geq T_0$ for the character $\chi_0 \pmod 1$.

   For $x \geq \max\{83,x_0,T_0\}$, $x_0 \geq \max\{e^e, T_1\}$ and
\begin{equation}
\label{eq:logqLowerSx}
    \log{x} \geq \exp\left(-4W_{-1}\left(-\left(4\cdot 5^{3/4}c_2^{3/4}\right)^{-1}\right)\right),
\end{equation}
we have
\begin{equation*}
    \left|S(x)- \frac{x^2}{2}\right| \leq \left(f_1(1,x_0)+2d_{8}(1) +\frac{f_{7,\zeta}(1, T_0,x_0)}{\log{x}}\right)x^{B_1^*(x)+1}.
\end{equation*}
\end{proposition}

The next one is a similar result for the function $S(x;q,a,b)$.

\begin{proposition}
\label{prop:firstMain}
Let $q\geq 83$, $a$ and $b$ be integers with $(ab,q)=1$. Assume that \ref{assumptionProductExceptional} holds for $|\Im(s)| \geq T_1$ and \ref{eq:rhoSecondUpper}, \ref{eq:defIntervalZeroGeneral} and \ref{eq:defzerosGeneral} hold for $T\geq T_1$ and for all $\chi$, and assumptions \ref{assumption:psiuchi}, \ref{eq:assumption1OverRhoSum}, \ref{eq:assumptionrhorho1} and \ref{eq:assumptionLambda} hold with some $T_0$ for all characters $\chi \pmod q$. Assume also that for all primitive characters $\chi^*$ that induce a character $\chi \pmod q$ and $T \geq T_1$, we have \ref{eq:assumptionLogDer}. 

For $x \geq \max\{x_0,4\cdot10^{15}, q, T_0\}$, $x_0 \geq \max\{e^e, T_1\}$ and
\begin{equation}
\label{eq:logqLower}
    \log{x} \geq \exp\left(-4W_{-1}\left(-\left(4\cdot 5^{3/4}\left(3.5c_1\left(\frac{\log{\log{x_0}}}{\log{x_0}}\right)^{2/3}+c_2\right)^{3/4}\right)^{-1}\right)\right),
\end{equation}
we have
\begin{multline}
\label{eq:SxMain}
    \left|S(x;q,a,b)- \frac{x^2}{2\varphi(q)^2}\right| \leq \max\left\{f_1(q,x_0)+\frac{2d_{8}(q)}{\varphi(q)} +\frac{f_7(q, T_0,x_0)}{\log{x}},1\right\}x^{B_q^*(x)+1} \\
    +2e^{2C_0}{x^{1+B_q^*(x)}}\left(\frac{\log\log{x}}{\log{x}}\right)^2+2.1x^{\frac{1+B_q^*(x)}{2}}\log{x}.
\end{multline}
\end{proposition}

\begin{remark}
    If $q \leq \log{x}$, the estimate \eqref{eq:SxMain} can be written as
    \begin{equation*}
        \left|S(x;q,a,b)- \frac{x^2}{2\varphi(q)^2}\right| \leq \left(f_1(q,x_0)+\frac{2d_{8}(q)}{\varphi(q)} +\frac{f_7(q, T_0,x_0)}{\log{x}}\right) x^{B_q^*(x)+1}. 
    \end{equation*}
\end{remark}

\begin{remark}
\label{rmk:dependence}
    From Propositions \ref{prop:firstMain} and \ref{prop:firstMainS(x)} 
    we see that the  estimates essential 
    to sharpen the result in Proposition \ref{prop:firstMain} are related 
    to the assumptions \ref{eq:assumption1OverRhoSum} and 
    \ref{eq:assumptionrhorho1}. These could improve the assumptions and 
    lower the term $x_0$. In particular, when $d_5(q)$ and $d_8(q)$ are
    small, the  crucial contribution comes from the term $x_0$. Indeed, 
    in those cases the coefficient of the error term does not depend that
    much on the explicit result itself but more on the lower bound where 
    some of them start to matter.
    
    However, it should be noted that better bounds for other results in 
    Appendix \ref{appendix:assumptions} may also improve these two 
    assumptions. For example,  sharper estimates for the number of zeros
    (see also the assumption \ref{eq:defzerosGeneral}) may improve those 
    estimates. In addition, it should be noted that the lower bound for
    $\log{x}$ increases when $x_0$ decreases.
\end{remark}

Since many explicit results are proved only for primitive or principal 
characters, it would be useful to provide a version of Proposition 
\ref{prop:firstMain} that does not require knowledge of the behavior of 
imprimitive or non-principal characters. The next remark describes one
result of this  kind.
\begin{remark}
    \label{rmk:primitiveS}
    Assume the hypotheses \ref{assumptionProductExceptional}, \ref{assumption:psiuchi}---\ref{eq:assumptionLambda} for primitive and principal characters. Then we have  a version of the estimate \eqref{eq:SxMain} where the terms $\varphi(q)$ are replaced with the terms $\varphi_1^*(q)$ and the second last term in \eqref{eq:SxMain} is
    \begin{equation*}
        82.366x^{1+B_q^*(x)}\left(\frac{(\log\log{x})^2}{\log{x}}\right)^2.
    \end{equation*}
\end{remark}

The second main result gives another estimate for the functions $S(x)$ and $S(x;q,a,b)$. They are used to prove Theorems \ref{prop:firstMainS(x)} and \ref{prop:firstMain}. First, we describe the result for the term $S(x)$ and then for the term $S(x;q,a,b)$.

\begin{proposition}
    \label{prop:SxSecond}
      Under the same assumptions of Appendix \ref{appendix:assumptions} as in Proposition \ref{prop:firstMainS(x)},
     for $x\geq \{83,x_0, T_0\}$, $x_0 \geq \max\{e^e, T_1\}$, we have
    \begin{equation*}
         \left|S(x)-\frac{x^2}{2}+2\sum_{|\rho_\zeta|<x}\frac{x^{\rho_\zeta+1}}{\rho_\zeta(\rho_\zeta+1)}\right|\leq  f_1(1,x_0)(\log{x})^5x^{2B_1^*(x)}+f_{7,\zeta}(1, T_0,x_0)(\log{x})^4x^{2B_1^*(x)},
    \end{equation*}
    where the sum runs over the non-trivial zeros of the Riemann zeta function.
\end{proposition}

\begin{proposition}    
\label{prop:Sxqab}
   Let $q\geq 83$, and let $a$ and $b$ be integers with $(ab,q)=1$. With the same assumptions of Appendix \ref{appendix:assumptions} as in Proposition \ref{prop:firstMain}, 
     for $x \geq \{q, T_0\}$, $x_0 \geq e^e$ we have
    \begin{multline*}
         \left|S(x;q,a,b)-\frac{x^2}{2\varphi(q)^2}+\frac{1}{\varphi(q)^2}\sum_{\chi \pmod q}\left(\overline{\chi}(a)+\overline{\chi}(b)\right)\sum_{|\rho_\chi|<x}\frac{x^{\rho_\chi+1}}{\rho_\chi(\rho_\chi+1)}\right| \\
         \leq  f_1(q,x_0)(\log{x})^5x^{2B_q^*(x)}+f_7(q, T_1,x_0)(\log{x})^4x^{2B_q^*(x)},
    \end{multline*}
    where the sum over zeros runs over the non-trivial zeros of the function $L(s,\chi)$.
\end{proposition}

\begin{remark}
\label{rmk:primitiveprincipalMainLemma}
    If we want to consider the sum over primitive and principal characters modulo $q$ instead of the sum over all characters modulo $q$, then we only need to change $\varphi(q)$ to $\varphi_1^*(q)$ on the right-hand side of the result. 
\end{remark}

Since the conjecture of Hardy and Littlewood \eqref{conj:HL} is expected to 
approximate $G(n)$, it is natural to consider the average order of the error term $G(n)-J(n)$ on the lines of (\cite{Granville2008} and \cite{BHMS2019}), where $G(n)$ is as in \eqref{def:Gn} and $J(n)$ is defined as
\begin{equation}
\label{def:Jn}
J(n):=nC_2\prod_{\substack{p\mid n\\p>2}}\frac{p-1}{p-2}.
\end{equation}
We will now give an explicit version of Theorem 3 of \cite{BHMS2019}.

\begin{proposition}\label{prop:JG}
Let $q \geq 83$, let $G(n)$ and $J(n)$ be as in \eqref{def:Gn} and \eqref{def:Jn}, respectively, and let $d(n)$ denote the number of positive divisors of $n$.
     Assume that all characters $\chi \pmod q$ and all primitive characters $\chi^* \pmod q$ that induce any $\chi \pmod q$ satisfy the following conditions: \ref{assumptionProductExceptional} holds for $|\Im(s)|\geq T_1$, \ref{eq:assumptionLogDer} holds for $\chi^*$ and $T\geq T_1$,  \ref{eq:rhoSecondUpper}, \ref{eq:defIntervalZeroGeneral} and \ref{eq:defzerosGeneral} hold for $T \geq T_1$ and $\chi$ and \ref{assumption:psiuchi}, \ref{eq:assumption1OverRhoSum} and \ref{eq:assumptionrhorho1} hold for $T \geq T_0$ and $\chi$. Assume also that \ref{eq:assumptionLambda} holds for $T_0$, and let $x_0 \geq \max\{e^e,T_1\}$. 
     
     Then, for $x\geq \max\{q,T_0, x_0, 572\}$ and for any positive integer $c$ we 
     have 
     \begin{equation*}
    \sum_{\substack{n\leq x\\n\equiv c\pmod{q}}}(G(n)-J(n))\leq \frac{2d_8(q)d(q)}{\varphi(q)}x^{B_q^{*}+1}+f_8(q,T_1,x_0,c)(\log{x})^5x^{2B_q^*(x)}.
    \end{equation*}
\end{proposition}

 We also note that in the previous proposition we can estimate  $\log{d(n)} \leq \frac{1.5379\log{n} \log{2}}{\log\log{n}}$ for $n \geq 2$ by \cite{NR1983}.

\section{Preliminary lemmas for the proofs of Propositions \ref{prop:SxSecond} and \ref{prop:Sxqab}}
\label{sec:preliminarylemmas}

In this section, we prove preliminary lemmas for Propositions \ref{prop:SxSecond} and \ref{prop:Sxqab}. Essentially, we derive estimates for the more general case, $S(x;q,a,b)$ since the estimates for the case $S(x)$ follow similarly. We follow closely the structure and ideas of the proofs from \cite[Section 5]{BHMS2019}. 

 We introduce some shortened notation. Let $\chi$ be a character modulo $q$. Now we
set $\delta_0(\chi)=1$ if $\chi$ is a principal character $\chi_0$, and $0$ otherwise.
We let $\delta_1(\chi)=1$ if there exists an exceptional non-trivial zero of 
$L(s,\chi)$ in the set \eqref{set:zerofree}, and $\delta_1(\chi)=0$ otherwise. Further,
\begin{equation}
\label{eq:defpisuchiC}
    \psi(x,\chi):=\sum_{n \leq x} \chi(n)\Lambda(n), \quad C\left(\chi\right):=\begin{cases}
     \frac{L'}{L}(1,\overline{\chi})+\log{\frac{q}{2\pi}}-C_0 & \text{if } \chi\neq \chi_0 \\
     -\log{(2\pi)} &\text{if } \chi= \chi_0 
    \end{cases},
\end{equation}
where $C_0$ is the is the Euler–Mascheroni constant. In addition, we introduce the notation
\begin{equation}
\label{def:Gnchi2Schi2}
    G(n; \chi_1,\chi_2):=\sum_{l+m=n}\chi_1(l)\Lambda(l)\chi_2(m)\Lambda(m), \quad S(x;\chi_1,\chi_2):=\sum_{n \leq x} G(n;\chi_1,\chi_2),
\end{equation}
\begin{equation}
\label{def:TSalphachiW}
    T(\alpha):=\sum_{n\leq x} e\left(n\alpha\right), \quad S(\alpha,\chi):=\sum_{n \leq x} \chi(n)\Lambda(n)e\left(n\alpha\right), \quad W(\alpha,\chi):=S(\alpha,\chi)-\delta_0(\chi)T(\alpha),
\end{equation}
\begin{equation}
\label{eq:defHR}
        H(x,\chi)=\sum_{|\rho_\chi| <x} \frac{x^{\rho_\chi+1}}{\rho_\chi(\rho_\chi+1)}, \quad R(x; \chi_1, \chi_2)=\int_{0}^1 W(\alpha, \chi_1)W(\alpha,\chi_2)T(-\alpha)\, d\alpha
    \end{equation}
\begin{equation}
\label{def:J}
    \text{and}\quad K(\chi):=\int_{-1/2}^{1/2} \left|W(\alpha,\chi)\right|^2\left|T(-\alpha)\right|\, d\alpha.
\end{equation}
Let also $N(T, \chi)$ denote the number of zeros $\varrho$ of of $L(s, \chi)$ with 
$0 < \Re(\varrho) < 1,|\Im (\varrho)| < T$. Finally recall  the notation from the 
beginning of Section \ref{sec:mainresults}.

The idea to prove Proposition \ref{prop:Sxqab} is to write the function 
$S(x;q,a,b)$ 
as a sum of the functions $S(x;\chi_1,\chi_2)$. Hence, we estimate the 
terms 
$S(x;\chi_1,\chi_2)$ first.

\begin{lemma}
\label{lemma:Schi12}
    Let $q \geq 83$, $\chi_1, \chi_2 \pmod q$ and let $\chi_1^*$ and 
    $\chi_2^*$ be primitive characters that induce $\chi_1$ and $\chi_2$,
    respectively. In the case of principal character let 
    $\chi_i=\chi_i^*$. Assume \ref{assumptionProductExceptional} for 
    $|\Im(s)| \geq T_1$. Let also $T_0$ be a positive real number such 
    that for $\chi\in \{\chi_1,\chi_2\}$, the assumptions 
    \ref{assumption:psiuchi} and \ref{eq:assumptionrhorho1} hold. Let now 
    $x \geq \max\{q,T_0,x_0\}$, where $x_0 \geq e^e$. Assume also 
    \ref{eq:assumptionLogDer} for $\chi_1^*, \chi_2^*$ and $T \geq T_1$. 
   
    Then, we have   
\begin{multline*}
        \left|S(x;\chi_1,\chi_2)-\frac{\delta_0(\chi_1)\delta_0(\chi_2)}{2}x^2+\delta_0(\chi_2)H(x,\chi_1)+\delta_0(\chi_1)H(x,\chi_2)-R(x; \chi_1, \chi_2)\right|\\
        \leq x(\log{x})(\log\log{x})\left(\frac{c_{q,3}}{(\log{x_0})(\log\log{x_0})}\left(\delta_0(\chi_1)\delta_1(\chi_2^*)+\delta_0(\chi_2)\delta_1(\chi_1^*)\right) \right. \\
        \left.+\frac{3}{2(\log{x_0})(\log\log{x_0})}\delta_0(\chi_1)\delta_0(\chi_2)+\left(\delta_0(\chi_1)+\delta_0(\chi_2)\right)f_6(q,x_0) \right),
    \end{multline*}
    where $S(x;\chi_1,\chi_2)$, $H(x,\chi)$ and $R(x; \chi_1, \chi_2)$ 
    are defined as in \eqref{def:Gnchi2Schi2} and \eqref{eq:defHR}, 
    respectively.
\end{lemma}

\begin{proof}
From the first paragraph of the proof of \cite[Lemma 11]{BHMS2019}, we have
\begin{equation*}
    S(x,\chi_1,\chi_2)=\delta_0(\chi_2)I(\chi_1)+\delta_0(\chi_1)I(\chi_2)-\delta_0(\chi_1)\delta_0(\chi_2)I+R(x;\chi_1,\chi_2),
\end{equation*}
where
\begin{equation*}
    I:=\int_0^1 T(\alpha)^2T(-\alpha) \, d\alpha \quad\text{and}\quad I(\chi):=\int_0^1 S(\alpha,\chi)T(\alpha)T(-\alpha)\, d\alpha.
\end{equation*}
Hence, we want to estimate the terms
\begin{equation*}
    I-\frac{x^2}{2} \quad\text{and}\quad I(\chi)-\frac{\delta_0(\chi)x^2}{2}+H(x,\chi).
\end{equation*}

Let us start with the integral $I$. By orthogonality, the only terms that are left are the terms $e(0)$ and hence we have
\begin{equation*}
    \left|I-\frac{x^2}{2}\right|=\left|\sum_{m+n\leq x} 1-\frac{x^2}{2}\right|=\left|\sum_{n\leq x} (n-1)-\frac{x^2}{2}\right|\leq \frac{3x}{2}.
\end{equation*}
Let us estimate the integral $I(\chi)$ next.

Similar to the previous case, we can deduce that
\begin{equation*}
    I(\chi)=\sum_{m+n\leq x} \chi(n)\Lambda(n)=\sum_{n \leq x} \left(x-n\right)\chi(n)\Lambda(n)-\sum_{n \leq x} \{x\}\chi(n)\Lambda(n).
\end{equation*}
Due to assumption \ref{eq:assumptionLambda}, the absolute value of the last term in the right-hand side is at most $d_{11}x$. Changing the order of the integration and the summation, the first term in the right-hand side can be written as
\begin{equation*}
    =\sum_{2\leq n \leq x} \chi(n)\Lambda(n)\int_{n}^x 1 \, du=\int_2^x \psi(u,\chi) \, du.
\end{equation*}
Because of the assumption \ref{assumption:psiuchi} with $T=x$, the right-hand side can be estimated as 
\begin{multline*}
    \left|\int_2^x \psi(u,\chi) \, du-\frac{\delta_0(\chi) x^2}{2}+\sum_{|\rho_\chi|<x} \frac{x^{\rho_\chi+1}}{(\rho_\chi+1)\rho_\chi}\right| \leq  d_1(q)(x-2)\log{x}\log\log{x}+\frac{d_2(q)x(2\log{x}-1)}{4}\\
    +d_3(q)\left(x\log{x}-x-2\log{2}+2\right)+d_4(q)(x-2)\log{x}+2\delta_0(\chi)+\left|C(\chi^*)\right|(x-2)+\left|\sum_{|\Im(\rho_\chi)|<x} \frac{2^{\rho_\chi+1}}{(\rho_\chi+1)\rho_\chi}\right|.
\end{multline*}
Using assumptions \ref{eq:assumptionLogDer} and \ref{eq:assumptionrhorho1}, the last three terms are
\begin{equation*}
    \leq x(c_{q,4}+1)\log{q}+\frac{x\delta_1(\chi^*)}{1-\beta_1}+4d_8(q)
\end{equation*}
if $\chi \neq \chi_0$. Note that since $c_{q,4}$ is non-negative and $q \geq 83$, the estimate holds also when $\chi$ is a principal character. Combining the estimates for $I$ and $I(\chi)$, simplifying them, using assumption \ref{assumptionProductExceptional} and keeping in mind $x \geq q \geq 83$, we obtain the result.
\end{proof}

\begin{remark}
\label{rmk:SFormula}
    In the case of $S(x)$, we have $q=1$, $\chi_1=\chi_2=\chi_0$ and 
    hence $S(x; \chi_1,\chi_2)=S(x)$ in the previous lemma. Similarly as
    before, for $x\geq \{x_0,T_0\}$ and $x_0 \geq 83$ we have
    \begin{equation*}
        \left|S(x)-\frac{x^2}{2}+2H(x,\chi_0)-R(x;\chi_0,\chi_0)\right| \\
        \leq x(\log{x})(\log\log{x})f_{6, \zeta}(1,x_0).
    \end{equation*}
\end{remark}

Now the essential part is to estimate the term $R(x;\chi_1,\chi_2)$.
The next five  estimates provide the desired bound.

\begin{lemma}
\label{lemma:gammachi}
    Let $q \geq 1$ and assume that \ref{eq:defIntervalZeroGeneral} holds for $T_1$. Let $T \geq T_1$, and $y$, $|y| \leq T$, be real numbers. 
    
    Then we have
    \begin{equation*}
        \sum_{\substack{\rho_\chi \\ |\gamma_\chi| \leq T}} \frac{1}{1+\left|\gamma_\chi-y\right|} \leq N(T_1,\chi)+d_{9}\log{\left(q(3T+1)\right)}\left(\log{(2T)}+1+C_0+\frac{1}{4T}\right).
    \end{equation*}
\end{lemma}

\begin{proof} 
    Using 
    \ref{eq:defIntervalZeroGeneral} and \cite{Young} we can deduce that
    \begin{multline*}
        \sum_{\substack{\rho_\chi \\ |\gamma_\chi| \leq T}} \frac{1}{1+\left|\gamma_\chi-y\right|} \leq \sum_{\substack{\rho_\chi \\ |\gamma_\chi|< T_1}} 1+ \sum_{\substack{\rho_\chi \\ |\gamma_\chi-y|\leq 1 \\ |\gamma_\chi|\geq T_1}} 1+\sum_{\substack{1 \leq n \leq 2T}} \frac{1}{n} \sum_{\substack{\rho_\chi \\ n<|\gamma_\chi-y| \leq n+1 \\ |\gamma_\chi|\geq T_1}} 1 \\
        \leq N(T_1,\chi)+d_{9}\log{\left(q(3T+1)\right)}\left(1+\log{(2T)}+C_0+\frac{1}{4T}\right).
    \end{multline*}
\end{proof}

\begin{lemma}
\label{lemma:psilambdadelta}
    Let $q\geq 83$ and consider $\chi \pmod q$. Assume that \ref{assumption:psiuchi} holds for $T_0$ and \ref{eq:rhoSecondUpper}, \ref{eq:defIntervalZeroGeneral} and \ref{eq:defzerosGeneral} hold for $T_0$. In addition, let us also assume that $h, x$ are real numbers with $1 \leq h \leq x/x_0$, $x \geq \max\{T_0, q\}$ and $x_0 \geq \max\{e^e, T_1\}$.

Now we have
\begin{equation}
\label{eq:chiLambdaInterval}
    \int_x^{2x} \left|\sum_{t<n\leq t+h} \chi(n)\Lambda(n)-\delta_0(\chi)h\right|^2 \, dt
    \leq 3f_3(q,T_1,x_0)h x^{2B_q^*(x)}\left(\log{x}\right)^3+\frac{3f_2(q, x_0) }{\pi}\left(\log{x}\right)^2(\log\log{x})^2x.
\end{equation}
\end{lemma}

\begin{proof}
    With the assumption \ref{assumption:psiuchi} with $T=x$, we can write
    \begin{multline}
    \label{eq:zerosOther}
       \left| \sum_{t<n\leq t+h} \chi(n)\Lambda(n)-\delta_0(\chi)h\right|\leq \left|\sum_{|\Im(\rho_\chi)|<x} \frac{(t+h)^{\rho_\chi}-t^{{\rho_\chi}}}{{\rho_\chi}}\right| \\
       +d_1(q)\frac{(t+h)(\log{(t+h)})(\log{\log{(t+h)}})+t(\log{t})(\log\log{t})}{x} \\
      +d_2(q)\frac{(t+h)\log{(t+h)}+t\log{t}}{x}+d_3(q)\left(\log{(t+h)}+\log{t}\right)+2d_4(q)\log{x}. 
    \end{multline}
Since $x \geq x_0$, $x \geq q\geq 83$, $h \leq x/x_0$ and $\log{(2x)}<\log{\left(x\left(2+\frac{1}{x_0}\right)\right)} <1.2\log{x}$, for $t \leq 2x$ the last four terms are
    \begin{multline*}
    \leq \left(\log{x}\right)(\log\log{x})\cdot\left(1.2\cdot\left(4+\frac{1}{x_0}\right)\left(1+\frac{\log{1.2}}{\log\log{x_0}}\right)d_1(q)+\frac{1.2d_2(q)}{\log\log{x_0}}\left(4+\frac{1}{x_0}\right) \right. \\
    \left.+\frac{d_3(q)}{\log\log{x_0}}\left(2+\frac{\log{(2(2+x_0^{-1}))}}{\log{x_0}}\right)+\frac{2d_4(q)}{\log\log{x_0}}\right).
    \end{multline*}
    Now, when we integrate the second power of the right-hand side from $t=x$ to $t=2x$, we get
    \begin{equation}
    \label{eq:easyfinal}
        =\frac{1}{\pi}\left(\log{x}\right)^2(\log{\log{x}})^2xf_2(q,x_0).
    \end{equation}
    Hence, we have estimated the last four terms coming from estimate \eqref{eq:zerosOther}.
    Let us next consider the sum over non-trivial zeros.

    We have
    \begin{equation*}
        \sum_{|\Im({\rho_\chi})|<x} \frac{(t+h)^{\rho_\chi}-t^{{\rho_\chi}}}{{\rho_\chi}}=\sum_{\substack{{\rho_\chi} \\ |\Im({\rho_\chi})| \leq x/h}} \int_{t}^{t+h} u^{{\rho_\chi}-1} \, du+\sum_{x/h<|\Im({\rho_\chi})|<x} \frac{(t+h)^{\rho_\chi}-t^{{\rho_\chi}}}{{\rho_\chi}}=: \psi_1(t)+\psi_2(t).
    \end{equation*}
   In order to estimate the left-hand side of \eqref{eq:chiLambdaInterval}, we want to estimate the integrals
    \begin{equation*}
        \int_{x}^{2x} \left|\psi_1(t)\right|^2 \, dt \quad\text{and}\quad  \int_{x}^{2x} \left|\psi_2(t)\right|^2 \, dt.
    \end{equation*}

    Let us first consider the integral $\int_{x}^{2x} \left|\psi_1(t)\right|^2 \, dt$. By the Cauchy-Schwarz inequality, we have
    \begin{equation*}
        \int_{x}^{2x} \left|\sum_{\substack{{\rho_\chi} \\ |\Im({\rho_\chi})| \leq x/h}} \int_{t}^{t+h} u^{{\rho_\chi}-1} \, du \right|^2 \, dt \leq h  \int_{x}^{2x} \int_{t}^{t+h} \left|\sum_{\substack{{\rho_\chi} \\ |\Im({\rho_\chi})| \leq x/h}}  u^{{\rho_\chi}-1}\right|^2 \, du \, dt.
    \end{equation*}
    When we change the order of the integration and note $h \leq x/x_0$, the right-hand side is
    \begin{equation*}
        = h\int_{x}^{2x+h} \left|\sum_{\substack{{\rho_\chi} \\ |\Im({\rho_\chi})| \leq x/h}}  u^{{\rho_\chi}-1}\right|^2 \left(\int_{\max\{u-h,x\}}^{\min\{u,2x\}} 1 \, dt\right)\, du \leq \left(\frac{h}{x}\right)^2 \int_{x}^{x(2+1/x_0)} \left|\sum_{\substack{{\rho_\chi} \\ |\Im({\rho_\chi})| \leq x/h}}  u^{{\rho_\chi}}\right|^2 \, du.
    \end{equation*}
    Let us denote sums over zeros of $L(s,\chi)$ by $\rho_1=\beta_1+\gamma_1i$ and $\rho_2=\beta_2+\gamma_2i$. Integrating first, then using the arithmetic-geometric inequality and noting that $\max\{\beta_1,\beta_2\} \leq B_q^*(x)$, the right-hand side is
    \begin{multline*}
        =\left(\frac{h}{x}\right)^2 \int_{x}^{2+1/x_0} \sum_{\substack{\rho_1 \\ |\gamma_1| \leq x/h}}\sum_{\substack{\rho_2 \\ |\gamma_2| \leq x/h}} \frac{u^{\rho_1+\overline{\rho_2}}+u^{\overline{\rho_1}+\rho_2}}{2} \, du \\
        \leq \left(\frac{h}{x}\right)^2\sum_{\substack{\rho_1 \\ |\gamma_1| \leq x/h}}\sum_{\substack{\rho_2 \\ |\gamma_2| \leq x/h}} \frac{\sqrt{2}\left(\left(2+\frac{1}{x_0}\right)^{\beta_1+\beta_2+1}-1\right)x^{\beta_1+\beta_2+1}}{1+\left|\gamma_1-\gamma_2\right|} \\
        \leq \sqrt{2}\left(\left(2+\frac{1}{x_0}\right)^{2B_q^*(x)+1}-1\right)h^2 x^{2B_q^*(x)-1}\sum_{\substack{\rho_1 \\ |\gamma_1| \leq x/h}}\sum_{\substack{\rho_2 \\ |\gamma_2| \leq x/h}} \frac{1}{1+\left|\gamma_1-\gamma_2\right|}.
    \end{multline*}

    Noting that $B_q^*(x) \leq 1$ and applying Lemma \ref{lemma:gammachi} and assumption \ref{eq:defzerosGeneral}, the right-hand side is
    \begin{multline*}
        \leq \sqrt{2}\left(\left(2+\frac{1}{x_0}\right)^{3}-1\right)h x^{2B_q^*(x)}\left(N(T_1,\chi) \vphantom{\log{\left(\frac{4qx}{h}\right)}} \right. \\
        \left.+d_{9}\log{\left(\frac{4qx}{h}\right)}\left(\log{\frac{x}{h}}+\log{2}+1+C_0+\frac{h}{4x}\right)\right)\cdot d_{10}(q) \log{\left(\frac{qx}{2\pi e h}\right)}.
    \end{multline*}

   When we combine the terms and remember that $x/h \geq x_0$, we get
    \begin{multline}
      \label{eq:psi1final}
        \leq  \sqrt{2}\left(\left(2+\frac{1}{x_0}\right)^{3}-1\right)d_{10}(q)h x^{2B_q^*(x)}\left(\log{x}\right)^3\left(\left(4+\frac{11.9}{\log{x_0}}+\frac{6.3}{(\log{x_0})^2}+\frac{\log{x_0}+\log{2}}{x_0(\log{x_0})^2}\right)d_{9} \right. \\
        \left.+\frac{2N(T_1,\chi)}{(\log{x_0})^2}\right).
    \end{multline}

    Let us now consider the integral $\int_{x}^{2x} \left|\psi_2(t)\right|^2 \, dt$. The upper bound follows similarly as in the case of $\psi_1(t)$. Noting that $x \leq t \leq t+h\leq 2+ 1/x_0$ and ignoring the difference, the integral can be estimated as
    \begin{multline*}
        \int_{x}^{2x} \left|\psi_2(t)\right|^2 \, dt
        \leq 2\int_{x}^{2x+h}\left|\sum_{x/h<|\Im({\rho_\chi})|<x} \frac{t^{{\rho_\chi}}}{{\rho_\chi}}\right|^2 \, dt \\
        =2 \sum_{x/h<|\Im(\rho_1)|<x}\sum_{x/h<|\Im(\rho_2)|<x}\int_{x}^{2x+h} \frac{t^{\beta_1+\beta_2}\left(t^{(\gamma_1-\gamma_2)i}\overline{\rho_1}\rho_2+t^{(\gamma_2-\gamma_1)i}\rho_1\overline{\rho_2}\right)}{2\left|\rho_1\rho_2\right|^2} \, dt \\
        \leq \sum_{x/h<|\gamma_1|<x}\sum_{x/h<|\gamma_2|<x} \frac{2\sqrt{2}\left(\left(2+\frac{1}{x_0}\right)^{\beta_1+\beta_2+1}-1\right)x^{\beta_1+\beta_2+1}}{\left|\rho_1\rho_2\right|\left(1+\left|\gamma_1-\gamma_2\right|\right)}.
    \end{multline*}
    Again, using arithmetic-geometric inequality and noting that $\max\{\beta_1,\beta_2\} \leq B_q^*(x)$, the right-hand side of the inequality is
    \begin{equation*}
        \leq 2\sqrt{2}\left(\left(2+\frac{1}{x_0}\right)^{2B_q^*(x)+1}-1\right)x^{2B_q^*(x)+1}\sum_{x/h<|\gamma_1|<x}\frac{1}{\left|\gamma_1\right|^2}\sum_{x/h<|\gamma_2|<x} \frac{1}{1+\left|\gamma_1-\gamma_2\right|}.
    \end{equation*}
    Noting that $B_q^*(x) \leq 1$, using Lemma \ref{lemma:gammachi}, assumption \ref{eq:rhoSecondUpper} and the fact $x\geq x_0$, the right-hand side is
    \begin{multline}
       \leq  2 \sqrt{2}\left(\left(2+\frac{1}{x_0}\right)^{3}-1\right) d_{7}(q)hx^{2B_q^*(x)}\log^3{x}\cdot \left(\left(4+\frac{11.9}{\log{x_0}}+\frac{6.3}{(\log{x_0})^2}+\frac{\log{x_0}+\log{2}}{x_0(\log{x_0})^2}\right)d_{9} \right. \\
        \left.+\frac{2N(T_1,\chi)}{(\log{x_0})^2}\right). \label{eq:psi2final}
    \end{multline}
    The result follows when we combine estimates \eqref{eq:easyfinal}, \eqref{eq:psi1final} and \eqref{eq:psi2final} and use the assumption \ref{eq:defzerosGeneral}.
\end{proof}

\begin{remark}
    If $q=1$, $x_0 \geq \{e^e, T_1\}$, $1 \leq h \leq x/x_0$, $x_0 \geq \max\{83, T_0\}$. Otherwise the same assumptions as in Lemma \ref{lemma:psilambdadelta} and then
    the estimate \eqref{eq:chiLambdaInterval} holds, too. 
\end{remark}

\begin{lemma}
\label{lemma:Gallagher}
    Let $\xi>0$, $\delta=(2\xi)^{-1}$ and
    \begin{equation*}
        S(t):=\sum c(v)\cdot e(vt)
    \end{equation*}
    be an absolutely convergent exponential sum where the frequencies $v$ 
    run over arbitrary sequence of real numbers and the coefficients 
    $c(v)$ are complex. Then
    \begin{equation*}
        \int_{-\xi}^\xi \left|S(t)\right|^2 \, dt \leq \left(\frac{\pi}{2}\right)^2 \int_{-\infty}^\infty \left|\delta^{-1}\sum_{x}^{x+\delta} c(v)\right|^2 \, dx.
    \end{equation*}
\end{lemma}

\begin{proof}
    This follows immediately from the proof of \cite[Lemma 1]{Gallagher1970} by noting that $\left|\sin(\pi\delta t)\right|/\left|\pi \delta t\right| \geq 2/\pi$ if $|t| \leq \xi$.
\end{proof}

\begin{lemma}
\label{lemma:EstW}
    Let $q \geq 83$, $\chi \pmod q$ and let $\chi^*$ be a primitive character that induces $\chi$.
    Assume \ref{eq:ZeroFree3}, \ref{eq:assumptionLogDer} for $\chi^*$ and $T\geq T_1$,  \ref{eq:rhoSecondUpper}, \ref{eq:defIntervalZeroGeneral} and \ref{eq:defzerosGeneral} for $T \geq T_1$ and $\chi$ and \ref{assumption:psiuchi} and \ref{eq:assumption1OverRhoSum} for $T \geq T_0$ and $\chi$. Finally, let $x$ and $\xi$ be a real numbers such that $x \geq \max\{q,T_0\}$, $x_0 \geq \max\{e^e,T_1\}$ and $x_0x^{-1}/2 \leq \xi \leq 1/2$.
    
    Then, we have
    \begin{multline*}
        \int_{-\xi}^\xi \left|W(\alpha,\chi)\right|^2 d\alpha \leq  \xi f_1(q,x_0)\log{2}\left(\log{x}\right)^4x^{2B_q^*(x)}+\left(f_4(q,T_1,x_0)+\delta_0(\chi)f_5(q,T_1,x_0)\right)\xi\left(\log{x}\right)^3x^{2B_q^*(x)},
    \end{multline*}
    where $W(\alpha,\chi)$ is defined as in \eqref{def:TSalphachiW}.
\end{lemma}

\begin{proof}
Notice that
\begin{equation*}
    W(\alpha,\chi)=\sum_{0<n\leq x} \left(\chi(n)\Lambda(n)-\delta_0(\chi)\right)e(n\alpha).
\end{equation*}
Hence, using Lemma \ref{lemma:Gallagher}, we have
\begin{multline*}
    \int_{-\xi}^\xi \left|W(\alpha,\chi)\right|^2 d\alpha =\int_{-\xi}^\xi \left|\sum_{0<n\leq x} \left(\chi(n)\Lambda(n)-\delta_0(\chi)\right)e(n\alpha)\right|^2 \, d\alpha \\
    \leq \pi^2\xi^2\int_{-(2\xi)^{-1}}^x \left|\sum_{a(t)<n\leq b(t)}\left(\chi(n)\Lambda(n)-\delta_0(\chi)\right)\right|^2 \, dt,
\end{multline*}
where $a(t):=\max\{0,t\}$ and $b(t):=\min\{x,t+(2\xi)^{-1}\}$. The integral can be divided to the following three parts:
\begin{equation}
\label{eq:div3Integrals}
    \pi^2\xi^2 \int_{-(2\xi)^{-1}}^{x_0(2\xi)^{-1}}+\pi^2\xi^2 \int_{x_0(2\xi)^{-1}}^{x-x_0(2\xi)^{-1}}+\pi^2\xi^2 \int_{x-x_0(2\xi)^{-1}}^x=: \pi^2\xi^2 I_{-}+\pi^2\xi^2 I+\pi^2\xi^2 I_{+}. 
\end{equation}
We estimate each of the integrals. Let us start with integrals $I_{\pm}$.

By assumptions \ref{eq:ZeroFree3} and \ref{assumption:psiuchi}, we have
\begin{multline}
\label{eq:Lambdadelta0I-+}
    \left|\sum_{n \leq t} \chi(n)\Lambda(n)-\delta_0(\chi)t\right| \leq x^{B_q^*(x)}\sum_{\substack{|\Im({\rho_\chi})|<x \\ {\rho_\chi} \neq 1-\beta_1}}\frac{1}{\left|{\rho_\chi}\right|}+\left|C(\chi^*)-\delta_1(\chi^*)\frac{t^{1-\beta_1}}{1-\beta_1}\right| \\
    +d_1(q)(\log{x})(\log\log{x})+\left(d_2(q)+d_3(q)+d_4(q)\right)\log{x}
\end{multline}
for $2 \leq t \leq x$. Under assumption \ref{eq:assumptionLogDer}, the second term in the right-hand side of the inequality 
\eqref{eq:Lambdadelta0I-+} is
\begin{multline}
\label{eq:83}
    \leq \delta_1(\chi)\frac{t^{1-\beta_1}-1}{1-\beta_1}+(c_{q,4}+1)\log{q}-\log{2\pi}-C_0+\leq \delta_1(\chi)\frac{\log{t}}{1-\beta_1}\int_0^{1-\beta_1} t^\sigma \, d\sigma+(c_{q,4}+1)\log{q}
\end{multline}
if $\chi \neq \chi_0$. Clearly, the estimate holds also in the case $\chi=\chi_0$. Using this, the assumption \ref{eq:assumption1OverRhoSum} with $x=T$ and keeping in mind $B_q^*(x)\geq 1/2$, the right-hand side of \eqref{eq:Lambdadelta0I-+} is
\begin{equation*}
 \leq x^{B_q^*(x)}\log{x}\cdot\left(d_{5}(q)\log{x}+d_{6}(q)+\delta_1(\chi)+d_1(q)\frac{\log\log{x}}{\sqrt{x}}+\frac{d_2(q)+d_3(q)+d_4(q)}{\sqrt{x}}+\frac{c_{q,4}+1}{\sqrt{x}}\right)
\end{equation*}
for $2 \leq t \leq x$.
If $0 \leq t <2$, then 
\begin{equation*}
    \left|\sum_{n \leq t} \chi(n)\Lambda(n)-\delta_0(\chi)t\right| \leq 2\delta_0(\chi).
\end{equation*}
Hence, we can estimate
\begin{multline*}
    \left|\sum_{a(t)<n\leq b(t)}\left(\chi(n)\Lambda(n)-\delta_0(\chi)\right)\right|^2\\
    \leq 4\delta_0(\chi) +4\delta_0(\chi)x^{B_q^*(x)}\log{x}\cdot\left(d_{5}(q)\log{x}+d_{6}(q)+\delta_1(\chi)+d_1(q)\frac{\log\log{x}}{\sqrt{x}}+\frac{d_2(q)+d_3(q)+d_4(q)}{\sqrt{x}}+\frac{c_{q,4}+1}{\sqrt{x}}\right)\\
    +x^{2B_q^*(x)}(\log{x})^2\left(d_{5}(q)\log{x}+d_{6}(q)+\delta_1(\chi)+d_1(q)\frac{\log\log{x}}{\sqrt{x}}+\frac{d_2(q)+d_3(q)+d_4(q)}{\sqrt{x}}+\frac{c_{q,4}+1}{\sqrt{x}}\right)^2.
\end{multline*}
We obtain the estimates for $I_{-}$ and $I_{+}$ when we multiply by
\begin{equation*}
    \frac{x_0+1}{2\xi}\quad\text{and}\quad \frac{x_0}{2\xi},
\end{equation*}
respectively.

Let us now estimate the integral $I$ in \eqref{eq:div3Integrals}. Clearly, by adding non-negative terms and dividing the integral to different parts, we have
\begin{multline*}
    I=\int_{x_0(2\xi)^{-1}}^{x-x_0(2\xi)^{-1}} \left|\sum_{t<n\leq t+(2\xi)^{-1}} \left(\chi(n)\Lambda(n)-\delta_0(\chi)\right)\right|^2 \, dt \\
    \leq \int_{x_0(2\xi)^{-1}}^{x} \left|\sum_{t<n\leq t+(2\xi)^{-1}} \chi(\chi)\Lambda(n)-\delta_0(\chi)\left(2\xi\right)^{-1}\right|^2 \, dt \\
    +2\delta_0\int_{x_0(2\xi)^{-1}}^{x} \left|\sum_{t<n\leq t+(2\xi)^{-1}} \chi(n)\Lambda(n)-\delta_0(\chi)\left(2\xi\right)^{-1}\right| \, dt+\delta_0x.
\end{multline*}
Using dyadic sums and the Cauchy-Schwarz inequality, the right-hand side is
\begin{multline*}
     \leq \sum_{k=0}^{\log_2{\left(\frac{2\xi x}{x_0}\right)}-1}\left( \int_{x/2^{k+1}}^{x/2^k} \left|\sum_{t<n\leq t+(2\xi)^{-1}} \chi(n)\Lambda(n)-\delta_0(\chi)\left(2\xi\right)^{-1}\right|^2 \, dt\right. \\
     \left.+2\delta_0\sqrt{\frac{x}{2^{k+1}}\int_{x/2^{k+1}}^{x/2^k} \left|\sum_{t<n\leq t+(2\xi)^{-1}} \chi(n)\Lambda(n)-\delta_0(\chi)\left(2\xi\right)^{-1}\right|^2 \, dt}\right)+\delta_0x.
\end{multline*}
By Lemma \ref{lemma:psilambdadelta}, the right-hand side is
\begin{multline*}
    \leq 6\frac{f_2(q,x_0)}{\pi}\left(\log{\frac{x}{2}}\right)^2\left(\log\log{\frac{x}{2}}\right)^2x+4\delta_0\sqrt{\frac{3f_2(q, x_0)}{\pi}}\left(\log{\frac{x}{2}}\right)\left(\log\log{\frac{x}{2}}\right)x  \\
    +6\frac{f_3(q,T_1,x_0)}{2\xi}\left(\log{\frac{x}{2}}\right)^3\cdot x^{2B_q^*\left(\frac{x}{2}\right)}+4\delta_0\sqrt{\frac{3f_3(q,T_1,x_0)}{2\xi}}\left(\log{\frac{x}{2}}\right)^{1.5}\cdot x^{B_q^*\left(\frac{x}{2}\right)+1/2}+\delta_0x.
\end{multline*}
The result follows when we combine the previous estimates and keep in mind that $x \geq q \geq 83$ and $0 \leq \delta_1(\chi)\leq 1$.
\end{proof}

\begin{remark}
    If needed, the assumption $q \geq 83$ can be easily changed to the assumption $q \geq 1$ in Lemma \ref{lemma:EstW}. However, then we would need the term $\log{q}+\log{2\pi}+C_0$ in the right-hand side of \eqref{eq:83} and the sum $\sum_{a(t)<n\leq b(t)}$ would not be estimated as easily as above. It should be noted that due to this fact,  most of the lemmas in this section have the assumption $q \geq 83$.
\end{remark}

\begin{remark}
    Under the assumptions $q=1$, $x \geq \max\{83,T_0\}$, $x_0 \geq \max\{e^e,T_1\}$ and $x_0x^{-1}/2 \leq \xi \leq \frac{1}{2}$, Lemma \ref{lemma:EstW} gets the form
    \begin{multline*}
        \int_{-\xi}^\xi \left|W(\alpha,\chi_0)\right|^2 d\alpha \leq  \xi f_1(q,x_0)\log{2}\left(\log{x}\right)^4x^{2B_q^*(x)} +\left(f_{4,\zeta}(q,T_1,x_0)+f_{5,\zeta}(q,T_1,x_0)\right)\xi\left(\log{x}\right)^3x^{2B_q^*(x)}.
    \end{multline*}
\end{remark}

Now we are ready to state the main lemma that gives us an estimate for the term $R(x;\chi_1,\chi_2)$.

\begin{lemma}
\label{lemma:EstJ}
  Let $q, \chi, T_0$, $x_0$ and $x$ be defined as in Lemma \ref{lemma:EstW} and assume that the same hypothesis hold for them as in that lemma. Moreover, let $K(\chi)$ be defined as in \eqref{def:J}. Then
    \begin{multline*}
        \left|K(\chi)\right| \leq f_1(q,x_0)(\log{x})^5x^{2B_q^*(x)}\\
        +\left(0.35x_0f_1(q,x_0)+\left(1.443+\frac{0.5x_0}{\log{x_0}}\right)\left(f_4(q,T_1,x_0)+\delta_0(\chi)f_5(q,T_1,x_0)\right)\right)(\log{x})^4x^{2B_q^*(x)}.
    \end{multline*}
\end{lemma}

\begin{proof}
Let us first bound the term $\left|T(-\alpha)\right|$. First we note that we clearly have $\left|T(-\alpha)\right| \leq x$. Even more, we also have
\begin{equation*}
    \left|T(-\alpha)\right| \leq \left|\frac{2}{e\left(\alpha\right)-1}\right|\leq 
    \begin{cases}
    \frac{2}{|\alpha|\cdot\left|\frac{\sin(-2\alpha \pi)}{\alpha}\right|}, &\text{if } |\alpha|\leq 1/4 \\
    \frac{2}{|\alpha|\cdot\left|\frac{1-\cos(-2\alpha \pi)}{\alpha}\right|} &\text{if } 1/4<|\alpha|\leq 1/2
    \end{cases}
    \leq \frac{1}{2|\alpha|}.
\end{equation*}
Hence, $\left|T(-\alpha)\right|\leq \min\{x,1/\left(2|\alpha|\right)\}$.

Let us now consider $K(\chi)$ with the help of the result above. We have
\begin{equation*}
    \left|K(\chi)\right| \leq x\int_{|\alpha|\leq \frac{x_0}{2x}} \left|W(\alpha,\chi)\right|^2 \, d\alpha+\sum_{k=1}^{\log_2{\left(\frac{2x}{x_0}\right)}-1} 2^{k}\int_{\frac{1}{2^{k+1}}<|\alpha|\leq \frac{1}{2^k}} \left|W(\alpha,\chi)\right|^2 \, d\alpha.
\end{equation*}
By Lemma \ref{lemma:EstW}, the right-hand side is
\begin{multline*}
    \leq \left(\log_2{\left(\frac{2x}{x_0}\right)}+\frac{x_0}{2}\right)\left(f_1(q,x_0)(\log{2})(\log{x})^4x^{2B_q^*(x)}+\left(f_4(q,T_0,x_0)+\delta_0f_5(q,T_0,x_0)\right)\left(\log{x}\right)^3x^{2B_q^*(x)}\right),
\end{multline*}
and the result follows.
\end{proof}

\begin{remark}
\label{rmk:J}
    If we again have $q=1$, $x \geq \max\{83, T_0\}$ and otherwise the same assumptions as in Lemma \ref{lemma:EstW}, then  as in the proof of the previous lemma, we obtain
       \begin{multline*}
        \left|K(\chi_0)\right| \leq f_1(1,x_0)(\log{x})^5x^{2B_1^*(x)}\\
        +\left(0.35x_0f_1(1,x_0)+\left(1.433+\frac{0.5x_0}{\log{x_0}}\right)\left(f_{4,\zeta}(1,T_0,x_0)+f_{5,\zeta}(1,T_0,x_0)\right)\right)(\log{x})^4x^{2B_1^*(x)}.
    \end{multline*}
\end{remark}

\section{Proofs of Propositions \ref{prop:SxSecond} and \ref{prop:Sxqab}}
\label{sec:FirstGeneral}
First we prove Proposition \ref{prop:Sxqab}.

\begin{proof}[Proof of Proposition \ref{prop:Sxqab}]
    By the orthogonality of characters, we have
    \begin{equation*}
        S(x;q,a,b)=\frac{1}{\varphi(q)^2}\sum_{\chi_1,\chi_2 \pmod q} \overline{\chi_1(a)\chi_2(b)} S(x; \chi_1,\chi_2).
    \end{equation*}
    Hence, by Lemma \ref{lemma:Schi12} and assumption \ref{assumptionProductExceptional}, the function $S(x;q,a,b)$ can be estimated as
    \begin{multline*}
        \left|S(x;q,a,b)-\frac{x^2}{2\varphi(q)^2}+\frac{1}{\varphi(q)^2}\sum_{\chi \pmod q}\left(\overline{\chi}(a)+\overline{\chi}(b)\right)\sum_{|\rho_\chi|<x}\frac{x^{\rho_\chi+1}}{\rho_\chi(\rho_\chi+1)}\right| \\
        \leq \frac{1}{\varphi(q)^2}\sum_{\chi_1,\chi_2 \pmod q} \left|R(x;\chi_1,\chi_2)\right|+\frac{x(\log{x})(\log\log{x})}{\varphi(q)}\left(\frac{2c_{q,3}+1.5}{\varphi(q)(\log{x_0})(\log\log{x_0})}+2f_6(q, x_0)\right).
    \end{multline*}
    Thus, we want to estimate the terms $R(x;\chi_1,\chi_2)$.

    By Cauchy-Schwarz inequality
    \begin{equation*}
       \frac{1}{\varphi(q)^2}\sum_{\chi_1,\chi_2 \pmod q} \left|R(x;\chi_1,\chi_2)\right|\leq \frac{1}{\varphi(q)^2}\sum_{\chi_1,\chi_2 \pmod q} K(\chi_1)^{1/2}K(\chi_2)^{1/2}.
    \end{equation*}
    By Lemma \ref{lemma:EstJ}, the right-hand side is
    \begin{multline*}
        \leq f_1(q,x_0)(\log{x})^5x^{2B_q^*(x)}\\
        +\left(0.35x_0f_1(q,x_0)+\left(1.443+\frac{0.5x_0}{\log{x_0}}\right)\left(f_4(q,T_1,x_0)+\frac{f_5(q,T_1,x_0)}{\varphi(q)}\right)\right)(\log{x})^4x^{2B_q^*(x)}.
    \end{multline*}
    Keeping in mind the lower bounds for $x$ and $B_q^*(x) \geq 1/2$, the result follows.
\end{proof}

The proof of Proposition \ref{prop:SxSecond} follows now easily from the previous proof.
\begin{proof}[Proof of Proposition \ref{prop:SxSecond}]
The claim follows directly from the fact that $\left|R(x;\chi_0,\chi_0)\right| \leq K(\chi_0)$, and Remarks \ref{rmk:SFormula} and \ref{rmk:J}. 
\end{proof}

\section{Proofs of Propositions \ref{prop:firstMainS(x)} and \ref{prop:firstMain} and Remark \ref{rmk:primitiveS}}
\label{sec:SecondGeneral}
\begin{proof}[Proof of Proposition \ref{prop:firstMain}]
    Let us first consider the case $q>x^{\frac{1-B_q^*(x)}{2}}\log{x}$. Using the trivial estimates that $\Lambda(m) \leq \log{x}$ if $m \leq x$ and there are at most $x/q+1$ such numbers $m$ that satisfy $m \equiv b \pmod q$, we obtain
    \begin{equation*}
         0 \leq S(x; q,a,b) \leq \left(\sum_{\substack{l \leq x \\ l \equiv a\pmod q}} \Lambda(l)\right)\left(\sum_{\substack{l \leq x \\ l \equiv a\pmod q}} \Lambda(l)\right) \leq \frac{(x\log{x})^2}{q^2}+\frac{2x\log^2{x}}{q}+\log^2{x}.
    \end{equation*}
    Since $q>x^{\frac{1-B_q^*(x)}{2}}\log{x}$, $B_q^*(x) \geq 1/2$ and $x \geq 4 \cdot 10^{15}$, the right-hand side is
    \begin{equation*}
    < x^{1+B_q^*(x)}+2.1x^{\frac{1+B_q^*(x)}{2}}\log{x}.
    \end{equation*}
    Moreover, by \cite[Theorem 15]{RS1962} and noting that $2.50637/\log\log{q} \leq e^{C_0}\log\log{q}$, we have
    \begin{equation}
    \label{eq:psiEstx}
        \frac{x^2}{2\varphi(q)^2} \leq 2e^{2C_0}{x^{1+B_q^*(x)}}\left(\frac{\log\log{x}}{\log{x}}\right)^2.
    \end{equation}

    This proves the case $q>x^{\frac{1-B_q^*(x)}{2}}\log{x}$.

    Let us now consider the case $q\leq x^{\frac{1-B_q^*(x)}{2}}\log{x}$. The idea is to estimate the terms in Proposition \ref{prop:Sxqab}. First we notice that by assumption the \ref{eq:assumptionrhorho1}, we have
    \begin{equation}
    \label{eq:estSumZeroschi}
        \frac{1}{\varphi(q)^2}\left|\sum_{\chi \pmod q}\left(\overline{\chi}(a)+\overline{\chi}(b)\right)\sum_{|\rho_\chi|<x}\frac{x^{\rho_\chi+1}}{\rho_\chi(\rho_\chi+1)}\right| \leq \frac{2d_{8}(q)}{\varphi(q)} x^{B_q^*(x)+1}.
    \end{equation}
    
    Let us consider the term $x^{2B_q^*(x)}(\log{x})^5$ next. As in the first paragraph of the proof of \cite[Theorem 1]{BHMS2019}, in the case $1-B_q^*(x)\geq 5\log{\log{x}}/\log{x}$, we have $x^{2B_q^*(x)}(\log{x})^5 \leq x^{1+B_q^*(x)}$. Hence, let us consider the case $1-B_q^*(x)<5\log{\log{x}}/\log{x}$. Now
    \begin{equation}
    \label{eq:upperq}
        \log{q} \leq \frac{1-B_q^*(x)}{2}\log{x}+\log\log{x}< 3.5\log{\log{x}}.
    \end{equation}
    Moreover, since $B_q^*(x) \leq 1-\eta_q(x)$, we can deduce
    \begin{equation}
    \label{eq:xBqlogmain}
        x^{2B_q^*(x)}(\log{x})^5=x^{1+B_q^*(x)}e^{-\eta_q(x)\log{x}+5\log\log{x}}.
    \end{equation}
    Combining estimates \eqref{eq:upperq} and \eqref{eq:xBqlogmain} with the definition \eqref{eq:defBq} of $\eta_q(x)$ and the assumptions for $x$, we get
     \begin{equation}
     \label{eq:5loglog}
        -\eta_q(x)\log{x} \leq -\frac{(\log{x})^{1/3}}{(\log\log{x})^{1/3} \left(3.5c_1\left(\frac{\log{\log{x_0}}}{\log{x_0}}\right)^{2/3}+c_2\right)} \leq -5\log\log{x}.
    \end{equation}
    Hence, also in this case the term $x^{2B_q^*(x)}(\log{x})^5$ is at most $x^{1+B_q^*(x)}$.

    From the previous paragraph, we have $x^{2B_q^*(x)}(\log{x})^4 \leq x^{1+B_q^*(x)}/\log{x}$. The result follows when we use Proposition \ref{prop:Sxqab}.  
\end{proof}

\begin{proof}[Proof of Proposition \ref{prop:firstMainS(x)}] Because of 
Proposition\ref{prop:SxSecond}, we would like to estimate the sum over zeros and the terms $x^{2B_q^*(x)}(\log{x})^5$ and $x^{2B_q^*(x)}(\log{x})^4$. As in \eqref{eq:estSumZeroschi}, the sum over zeros multiplied by two is $\leq 2d_8(q)x^{1+B_q^*(x)}$. The rest of the terms can be estimated as in the last paragraph of the proof of Proposition \ref{prop:firstMain}. 
\end{proof}

We  would now like to present a result where the sum is limited to run over principal and primitive characters. Hence, let us first present a remark that describes what we should change if our sums are over primitive and principal characters and then provide an estimate for the $\varphi_1^*(q)$ since it is needed to derive the result in Remark \ref{rmk:primitiveS}.

\begin{remark}
    We could also apply Remark \ref{rmk:primitiveprincipalMainLemma} instead of Proposition \ref{prop:Sxqab} and obtain the result in Remark \ref{rmk:primitiveS}. In the case $q>x^{\frac{1-B_q^*}{2}}\log{x}$, the only change we need to make is to replace the estimate \eqref{eq:psiEstx} with
    \begin{equation*}
        \frac{x^2}{2\varphi_1^*(q)^2}<82.366x^{1+B_q^*(x)}\left(\frac{(\log\log{x})^2}{\log{x}}\right)^2.
    \end{equation*}
    Here we used Lemma \ref{lemma:psi1Estimate} and the fact that for $q \geq 83$ we have $1+2.8/(\log\log{q})^2+1.96/(\log\log{q})^4<2.671$. In the case $q\leq x^{\frac{1-B_q^*(x)}{2}}\log{x}$, in addition of using Remark \ref{rmk:primitiveS}, we need the assumption \ref{eq:assumptionrhorho1} only for the some over primitive and principal characters. Hence, we change the term $\varphi(q)$ to $\varphi_1^*(q)$. 
\end{remark}

The next two results provide an estimates $q/\varphi_1^*(q)$. Let $\varphi^*(q)$ denote here the number of primitive characters modulo $q$.

\begin{lemma}
\label{lemma:prodpp2}
Assume that $x \geq 286$. Then 
    \begin{equation}
    \label{eq:p2Prod}
        \prod\limits_{2<p\leq x} \frac{p}{p-2} \leq \frac{e^{2C_0}}{4C_2}(\log{x})^2 \left(1+\frac{1}{2(\log{x})^2}\right)^2,
    \end{equation}
    where $C_2$ denotes the Twin Primes Constant.
\end{lemma}

\begin{proof}
    Let us first notice that 
    \begin{equation*}
        \prod\limits_{2<p\leq x} \frac{p}{p-2}=\left(\prod\limits_{2<p\leq x} \frac{p}{p-1}\right)^2\cdot \prod\limits_{2<p\leq x} \frac{(p-1)^2}{(p-1)^2-1} < \frac{1}{C_2} \left(\prod\limits_{2<p\leq x} \frac{p}{p-1}\right)^2.
    \end{equation*}
    By \cite[Estimate (3.29)]{RS1962}, the right-hand side is at most the right hand side of \eqref{eq:p2Prod}. We divide by $4$ since in the above sum we do not have $p=2$. 
\end{proof}

\begin{lemma}
\label{lemma:psi1Estimate}
    Let $q \not\equiv 2 \pmod{4}$ and $q \geq 3$ be an integer. Then
    \begin{equation}
    \label{eq:primitiveprincipal}
        \frac{q}{\varphi_1^*(q)} \leq \frac{q}{\varphi^*(q)}\leq \frac{e^{2C_0}}{C_2}\left((\log\log{q})^2+2.8+\frac{1.96}{(\log\log{q})^2}\right),
    \end{equation}
    where $C_2$ denotes the Twin Primes Constant.
\end{lemma}
\begin{proof}
     By \cite[Theorem 8]{S1969} $\varphi^*(p)=p-2$ and $\varphi^*(p^\alpha)=(p-1)^2p^{\alpha-2}$ when $p$ is a prime number and $\alpha \geq 2$ on integer. Notice that $\varphi^*(mn)=\varphi^*(m)\varphi^*(n)$ if $\gcd(m,n)=1$, $\varphi^*(q) \neq 0$ if $q \not\equiv 2 \pmod{4}$ and $\varphi_1^*(q)=\varphi^*(q)+1$. Hence, the first inequality in \eqref{eq:primitiveprincipal} is proved and we can also conclude that for all $q \not\equiv 2 \pmod{4}$ we have
     \begin{equation}
     \label{eq:prodStart}
         \frac{q}{\varphi^*(q)}=\frac{\prod_i^r q_i^{\alpha_i}}{\varphi^*(q)}\leq 4 \prod_{q_i \neq 2} \frac{q_i}{q_i-2},
         \end{equation}
         where $q=\prod_i^r q_i^{\alpha_i}$ is the prime factorization of $q$. Hence, we consider the product $\prod_{q_i \neq 2} \frac{q_i}{q_i-2}$.
         Since the function $x\mapsto 1+\frac{2}{x-2}$ is decreasing for $x>2$, the right-hand side of \eqref{eq:prodStart}  is at most
         \begin{equation}
         \label{eq:prodPrimes}
             \frac{q}{\varphi^*(q)} \leq 4\prod_{i=2}^{m} \frac{p_i}{p_i-2},
         \end{equation}
          where the numbers $p_2,\ p_2,\ \dots \ p_m$ denote the first $m-1$ odd primes.
         
         Let now $\theta(n):=\sum_{p \leq n} \log{p}$.
          Similarly as in the proof of Theorem 33 in \cite{RS1962}, if $q<\exp(\theta(x))$ for some real number $x \geq 5$, then the last product in \eqref{eq:prodPrimes} can be replaced with
         \begin{equation}
         \label{eq:estFinite}
             \frac{q}{\varphi^*(q)}\leq 4\prod\limits_{2<p\leq x-2} \frac{p}{p-2}.
         \end{equation}
         Using estimate \eqref{eq:estFinite} and Lemma \ref{lemma:prodpp2}, we can divide the proof to different cases based on the size of $q$.

         Similarly as in the proof of Lemma 15 in \cite{RS1962}, we obtain that if $y \in \mathbb{R}_+$ such that $288 \leq \log{q}+y$, $\log{q}<\theta(\log{q}+y)$ and
         \begin{equation*}
             0 \leq y-2 \leq \frac{0.9\log{q}}{\log{\log{q}}},
         \end{equation*}
         then
         \begin{equation*}
             \frac{e^{-2C_0}C_2q}{\varphi^*(q)}\leq 
             \left(\log\left(q+y-2\right)+\frac{0.5}{\log\left(\log{q}+y-2\right)}\right)^2
             \leq  (\log\log{q})^2+2.8+\frac{1.96}{(\log\log{q})^2}.
         \end{equation*}
         Here we used inequality \eqref{eq:estFinite} and Lemma \ref{lemma:prodpp2}, and the last inequality is the second power of the inequality in \cite[Lemma 15]{RS1962}. Now, as in the proof of Theorem 15 in \cite{RS1962}, we can choose $y=2+2\sqrt{1+\log{q}}$ if $286 \leq \log{q} \leq 1340$ and
         $y=0.9\log{q}/\log\log{q}$ if $\log{q} \geq 1340$. Hence, we have proved the cases where $\log{q} \geq 286$.

         We can check the cases $3 \leq q \leq 10^{7}$ by computing the exact values using Mathematica version 14.0.0.0. The desired estimate holds in this case. Hence, we can assume that the number $m$ in \eqref{eq:prodPrimes} is between $8$ and $64$. We go through each of these cases by using Mathematica again but now applying formula \eqref{eq:prodPrimes}. In the computations we have estimated $C_2>0.66$. The claim holds in these cases, too.
\end{proof}

\section{Proof of Proposition \ref{prop:JG}}

We follow the method of the proof of Theorem 3 \cite{BHMS2019}.

Again, $\chi^*$ denotes the inducing character and by $q^*$ the inducing modulus.
Moreover, denote by 
\begin{equation}
\label{def:Gq}
\mathfrak{S}_q(c,\chi)=\frac{\mu(q^{*})\chi^{*}(c)}{\varphi(q)\varphi(q^{*})}\prod_{\substack{p\mid q\\ p\nmid q^{*}c}}\frac{p-2}{p-1}
\end{equation}
and
\begin{equation}
\label{def:Gqc}
\mathfrak{S}_q(c)=\mathfrak{S}_q(c,\chi_0),
\end{equation}
 where $\mu$ is the M\"{o}bius function.

\begin{remark}
The reader is requested to note that there is a misprint in Theorem 4 of \cite{BHMS2019} where an extra 
$\mu(q^{*})$ occurs in the definition of the singular series. The singular series occurs in Lemma 12 of the same paper, whose correct formulation should be
\[
\sum_{\substack{
a=1\\
(a(c-a),q)=1
}}^{q}
\chi(a)
=
\mu(q^{\ast})\chi^{\ast}(c)\frac{\varphi(q)}{\varphi(q^{\ast})}
\prod_{\substack{
p\mid q\\
p\nmid q^{\ast}c
}}
\frac{p-2}{p-1}
=
\varphi(q)^{2}\mathfrak{S}_{q}(c,\chi),
\]
\end{remark}

We need the following version of Theorem 4 of that article:

\begin{theorem}\label{thm:versiothm4} 
Let $\mathfrak{S}_q(c,\chi)$ and $\mathfrak{S}_q(c)$ be as in \eqref{def:Gq} and \eqref{def:Gqc}, respectively. Let $q \geq 83$. Assume that all characters $\chi \pmod q$ and all primitive characters $\chi^* \pmod q$ that induce any $\chi \pmod q$ satisfy the following conditions: assume \ref{assumptionProductExceptional} for $|\Im(s)| \geq T_1$ and assume that \ref{eq:assumptionLogDer} holds for $\chi^*$ and $T\geq T_1$,  \ref{eq:rhoSecondUpper}, \ref{eq:defIntervalZeroGeneral} and \ref{eq:defzerosGeneral} hold for $T \geq T_1$ and $\chi$ and \ref{assumption:psiuchi}, \ref{eq:assumption1OverRhoSum} and \ref{eq:assumptionrhorho1} hold for $T \geq T_0$ and $\chi$. Assume also that \ref{eq:assumptionLambda} holds for $T_0$, and let $x_0 \geq \max\{e^e,T_1\}$.  Then, for $x\geq \max\{q,T_0, x_0\}$ and for any positive integer $c$ we have 
\begin{multline*}
\left|\sum_{\substack{n\leq x\\ n\equiv c\pmod{q}}}G(n)-\frac{\mathfrak{S}_q(c)}{2}x^2-2\sum_{\chi\pmod{q}} \mathfrak{S}_q(c,\overline{\chi})\sum_{|\rho_{\chi}|<x}\frac{x^{\rho_{\chi}+1}}{\rho_{\chi}(\rho_{\chi}+1)}\right|\leq \frac{d_{11}x(\log x)(\log q)}{\log 2}\\+\frac{(4c_{q,3}+3)x(\log x)(\log\log x)}{2(\log x_0)(\log\log x_0)\varphi(q)}+\frac{2x(\log x)(\log\log x)f_6(q,x_0)}{\varphi(q)}+f_1(q,x_0)(\log{x})^5x^{2B_q^*(x)}\\
        +\left(0.35x_0f_1(q,x_0)+\left(1.443+\frac{0.5x_0}{\log{x_0}}\right)\left(f_4(q,T_1,x_0)+\frac{f_5(q,T_1,x_0)}{\varphi(q)}\right)\right)(\log{x})^4x^{2B_q^*(x)}.
\end{multline*}
\end{theorem}

\begin{proof}
We start similarly as in \cite{BHMS2019}. We have
\[
\left|\sum_{\substack{n\leq x\\ n\equiv c\pmod{q}}}G(n)-\sum_{\substack{l+m\leq x\\ l+m\equiv c\pmod{q}\\ (lm,q)=1}}\Lambda(m)\Lambda(l)\right|\leq 2\sum_{\substack{l+m\leq x\\ l+m\equiv c\pmod{q}\\ (m,q)>1}}\Lambda(m)\Lambda(l).
\]
Now
\[
2\sum_{\substack{l+m\leq x\\ l+m\equiv c\pmod{q}\\ (m,q)>1}}\Lambda(m)\Lambda(l)\leq \sum_{l\leq x}\Lambda(l)\sum_{\substack{m\leq x\\(m,q)>1}}\Lambda(m)\leq d_{11} x\sum_{\substack{m\leq x\\(m,q)>1}}\Lambda(m)\leq \frac{d_{11}x(\log x)(\log q)}{\log 2},
\]
where the last bound is based on the estimate (2.3) in \cite{BHMS2019}, and $d_{11}$ corresponds to \ref{eq:assumptionLambda}. Next we write the other sum similarly as in the original article
\[
\sum_{\substack{l+m\leq x\\l+m\equiv c\pmod{q}\\(lm,q)>1}}\Lambda(l)\Lambda(m)=\frac{1}{\varphi(q)^2}\sum_{\substack{a=1\\(a(c-a),q)=1}}^q\sum_{\chi_1,\chi_2\pmod{q}} \overline{\chi_1(a)\chi_2(c-a)}S(x;\chi_1,\chi_2).
\] 
We may now replace the term $S(x;\chi_1,\chi_2)$ with the expression coming from Lemma \ref{lemma:Schi12}. The first three terms in the right hand side of that lemma yield
\[
\leq \frac{1}{\varphi(q)^2}\sum_{\substack{a=1\\(a(c-a),q)=1}}^q\frac{(4c_{q,3}+3)x(\log x)(\log\log x)}{2(\log x_0)(\log\log x_0)}\leq \frac{(4c_{q,3}+3)x(\log x)(\log\log x)}{2(\log x_0)(\log\log x_0)\varphi(q)}.
\]
The last term yields 
\begin{multline*}
    \frac{x(\log x)(\log\log x)f_6(q,x_0)}{\varphi(q)^2}\sum_{\substack{a=1\\(a(c-a),q)=1}}^q\sum_{\chi_1,\chi_2\pmod{q}} \overline{\chi_1(a)\chi_2(c-a)}(\delta_0(\chi_1)+\delta_0(\chi_2))\\
    =\frac{x(\log x)(\log\log x)f_6(q,x_0)}{\varphi(q)^2}\sum_{\substack{a=1\\(a(c-a),q)=1}}^q \left(\sum_{\chi_1}\sum_{\chi_2} \chi_1(a)\chi_2(c-a)\delta_0(\chi_1)+\sum_{\chi_1}\chi_1(a)\right)\\
=\frac{x(\log x)(\log\log x)f_6(q,x_0)}{\varphi(q)^2}\sum_{\substack{a=1\\(a(c-a),q)=1}}^q \left(\sum_{\chi_1}\sum_{\chi_2} \chi_1(a)\chi_2(c-a)\delta_0(\chi_1)+
\begin{cases}
0 &\text{if } a \not\equiv 1 \pmod q \\
\varphi(q) &\text{if } a \equiv 1 \pmod q
\end{cases}\right)\\
=\frac{2x(\log x)(\log\log x)f_6(q,x_0)}{\varphi(q)}
\end{multline*}
by (6.5) in \cite{BHMS2019}. We may now move to the main terms arising from Lemma \ref{lemma:Schi12}. We have
\[
\frac{1}{\varphi(q)^2}\sum_{\substack{a=1\\(a(c-a),q)=1}}^q\sum_{\chi_1,\chi_2\pmod{q}} \overline{\chi_1(a)\chi_2(c-a)}\frac{\delta_0(\chi_1)\delta_0(\chi_2)}{2}x^2=\frac{\mathfrak{S}_q(c)x^2}{2}
\]
We may now move to the $R$ term in the same lemma. We must estimate
\begin{multline*}
\frac{1}{\varphi(q)^2}\sum_{\substack{a=1\\(a(c-a),q)=1}}^q\sum_{\chi_1,\chi_2\pmod{q}} \overline{\chi_1(a)\chi_2(c-a)}R(x;\chi_1,\chi_2)\\
=\frac{1}{\varphi(q)^2}\sum_{\substack{a=1\\(a(c-a),q)=1}}^q\sum_{\chi_1,\chi_2\pmod{q}}\overline{\chi_1(a)\chi_2(c-a)}\int_{0}^1 W(\alpha,\chi_1)W(\alpha,\chi_2)T(-\alpha)d\alpha\\ \leq \frac{1}{\varphi(q)^2}\int_0^1|T(\alpha)|\left|\sum_{\substack{a=1\\(a(c-a),q)=1}}^q\sum_{\chi_1\pmod{q}}\overline{\chi_1(a)}W(\alpha,\chi_1)\sum_{\chi_2\pmod{q}}\overline{\chi_2(c-a)}W(\alpha,\chi_2)\right| d\alpha.
\end{multline*}
Using Cauchy-Schwarz inequality, we obtain
\begin{multline*}
\left|\sum_{\substack{a=1\\(a(c-a),q)=1}}^q\sum_{\chi_1\pmod{q}}\overline{\chi_1(a)}W(\alpha,\chi_1)\sum_{\chi_2\pmod{q}}\overline{\chi_2(c-a)}W(\alpha,\chi_2)\right| \\ \leq \sqrt{\sum_{\substack{a=1\\(a(c-a),q)=1}}^q\left|\sum_{\chi_1\pmod{q}}\overline{\chi_1(a)}W(\alpha,\chi_1)\right|^2}\sqrt{\sum_{\substack{a=1\\(a(c-a),q)=1}}^q\left|\sum_{\chi_2\pmod{q}}\overline{\chi_2(c-a)}W(\alpha,\chi_2)\right|^2}.
\end{multline*}
Since
\[
\sum_{\substack{a=1\\(a(c-a),q)=1}}^q\left|\sum_{\chi_1\pmod{q}}\overline{\chi_1(a)}W(\alpha,\chi_1)\right|^2\leq \sum_{\substack{a=1\\(a,q)=1}}^q\left|\sum_{\chi_1\pmod{q}}\overline{\chi_1(a)}W(\alpha,\chi_1)\right|^2,
\]
and
\[
\sum_{\substack{a=1\\(a,q)=1}}^q\left|\sum_{\chi_1\pmod{q}}\overline{\chi_1(a)}W(\alpha,\chi_1)\right|^2\leq \varphi(q)\sum_{\chi_1\pmod{q}}|W(\alpha,\chi_1)|^2
\]
and a similar bound holds for the other term, by Lemma \ref{lemma:EstJ} we obtain the bound
\begin{multline*}
\frac{1}{\varphi(q)}\int_0^1|T(\alpha)|\sum_{\chi\pmod{q}}|W(\alpha,\chi)|^2d\alpha=\frac{1}{\varphi(q)}\sum_{\chi\pmod{q}}K(\chi) \leq f_1(q,x_0)(\log{x})^5x^{2B_q^*(x)}\\
        +\left(0.35x_0f_1(q,x_0)+\left(1.443+\frac{0.5x_0}{\log{x_0}}\right)\left(f_4(q,T_1,x_0)+\frac{f_5(q,T_1,x_0)}{\varphi(q)}\right)\right)(\log{x})^4x^{2B_q^*(x)}.
\end{multline*}
Finally we observe
\begin{multline*}
\frac{1}{\varphi(q)^2}\sum_{\substack{a=1\\(a(c-a),q)=1}}^q\sum_{\chi_1,\chi_2\pmod{q}} \overline{\chi_1(a)\chi_2(c-a)}(\delta_0(\chi_2)H(x,\chi_1)+\delta_0(\chi_1)H(x,\chi_2))\\ =\frac{1}{\varphi(q)^2}\sum_{\chi\pmod{q}}H(x,\chi)\sum_{\substack{a=1\\(a(c-a),q)=1}}^q(\overline{\chi(a)}+\overline{\chi(c-a)})=2\sum_{\chi\pmod{q}}H(x,\chi)\mathfrak{S}_q(c,\chi)\\=2\mathfrak{S}_q(c,\chi)\sum_{\chi\pmod{q}}\sum_{|\rho_{\chi}|<x}\frac{x^{\rho_{\chi}+1}}{\rho_{\chi}(\rho_{\chi}+1)}.
\end{multline*}
The proof is complete.
\end{proof}

Now we also need the following explicit version of Lemma 13 \cite{BHMS2019}.

\begin{lemma}\label{lem:versionlemma13}
For $x\geq 572$ and positive integers $c$ ja $q$, we have
\[
\left|\sum_{\substack{n\leq x\\ n\equiv c\pmod{q}}}J(n)-\frac{\mathfrak{S}_q(c)}{2}x^2\right|\leq \frac{27}{4}e^{C_0}x\log x+\frac{9C_2e^{C_0}}{2}x(\log x+1-\log(2)).
\]
\end{lemma}

\begin{proof}
The proof is similar in the case with $q$ even and $c$ odd.

Similarly to the original proof, let us now consider the case with $\mathrm{gcd}(2,q)\mid c$, and start by writing
\[
\sum_{\substack{n\leq x\\ n\equiv c\pmod{q})}}J(n)=2C_2\sum_{\substack{d\leq x\\ d\, \textrm{odd}}}\frac{\mu(d)^2d}{\varphi_2(d)}\sum_{\substack{2dn\leq x\\ 2dn\equiv c\pmod{q}}} n=2C_2\sum_{\substack{d\leq x/2\\ d\, \textrm{odd}\\(d,q)\mid c}}\frac{\mu(d)^2d}{\varphi_2(d)}\sum_{\substack{n\leq x/2d\\ n\equiv c_1\pmod{q_1}}} n,
\]
where $c_1$ is a solution modulo $q_1=\frac{q}{(2d,q)}$ to the equation $2dn\equiv c$ if $(2d,q)\mid c$. Now we need a more precise bound for the inner sum as in the original paper. Write first $n=hq_1+c_1$. Now
\[
\sum_{\substack{n\leq x/2d\\ n\equiv c_1\pmod{q_1}}} n=\sum_{\substack{hq_1+c_1\leq x/2d}} (hq_1+c_1)=\frac{xc_1}{2dq_1}
+\frac{x^2}{8d^2q_1}+\frac{x}{4d}+\theta \frac{x}{2d},
\]
where $-1\leq \theta \leq 0$. Hence,
\begin{equation}\label{mu-sumsplit}
2C_2\sum_{\substack{d\leq x/2\\ d\, \textrm{odd}\\(d,q)\mid c}}\frac{\mu(d)^2d}{\varphi_2(d)}\sum_{\substack{n\leq x/2d\\ n\equiv c_1\pmod{q_1}}} n=\frac{C_2x^2}{4q}\sum_{\substack{d\leq x/2\\ d\, \textrm{odd}\\(d,q)\mid c}}\frac{\mu(d)^2(2d,q)}{d\varphi_2(d)}+2C_2\sum_{\substack{d\leq x/2\\ d\, \textrm{odd}\\(d,q)\mid c}}\frac{\mu(d)^2d}{\varphi_2(d)}\left(\theta \frac{x}{2d}+\frac{xc_1}{2dq_1}+\frac{x}{4d}\right).
\end{equation}
Similarly as in the proof of the original, and by \cite{RS1962} and Lemma \ref{lemma:prodpp2}, we have
\[
2C_2\sum_{\substack{d\leq x/2\\ d\, \textrm{odd}\\(d,q)\mid c}}\frac{\mu(d)^2}{\varphi_2(d)}\leq 2C_2\frac{2e^{2C_0}(\log x)^2\left(1+\frac{1}{2(\log x)^2}\right)^2}{4C_2e^{C_0}\log x\left(1-\frac{1}{2(\log x)^2}\right)^2}=e^{C_0}(\log x)\frac{\left(1+\frac{1}{2(\log x)^2}\right)^2}{\left(1-\frac{1}{2(\log x)^2}\right)^2}.
\]
 Thus, the absolute value of the latter term in \eqref{mu-sumsplit} is at most $\frac{27}{4}e^{C_0}x\log x$. Now we wish to extend the first sum in \eqref{mu-sumsplit}, so we estimate the remainder similarly as in the original proof:
\begin{multline*}
\frac{C_2x^2}{4q}\sum_{\substack{d>x/2\\ d\, \textrm{odd}\\(d,q)\mid c}}\frac{\mu(d)^2(2d,q)}{d\varphi_2(d)}<\frac{C_2x^2}{4}\sum_{\substack{d>x/2\\ d\, \textrm{odd}}}\frac{\mu(d)^2}{d\varphi_2(d)}\leq \frac{C_2x^2}{4}\int_{x/2}^{\infty}\left(\sum_{\substack{d\leq u\\ d\, \textrm{odd}}}\frac{\mu(d)^2}{\varphi_2(d)}\right)\frac{du}{u^2}\\ \leq \frac{C_2x^2}{4}\int_{x/2}^{\infty}9e^{C_0}\frac{\log u}{u^2}du= \frac{9C_2e^{C_0}}{2}x(\log x+1-\log(2)).
\end{multline*}
To get the main term to be $\frac{\mathfrak{S}_q(c)}{2}x^2$, one may use precisely the same manipulation as in the original proof. Hence, putting everything together, yields the lemma.
\end{proof}

Let us now move to the proof of Proposition \ref{prop:JG}. 
\begin{proof}[Proof of Proposition \ref{prop:JG}]
As in the proof of \cite[Theorem 3]{BHMS2019}, we start with bounding
\begin{multline*}
\sum_{\chi \pmod{q}}\mathfrak{S}_q(c,\chi)\sum_{|\rho_{\chi}|\leq x}\frac{x^{\rho_{\chi}+1}}{\rho_{\chi}(\rho_{\chi}+1)}\leq \frac{x^{B_q^{*}+1}}{\varphi(q)}\sum_{\chi\pmod{q}}\frac{1}{\varphi(q^{*})}\sum_{|\rho_{\chi}|\leq x}\frac{1}{|\rho_{\chi}(\rho_{\chi}+1)|}\\ \leq \frac{d_8(q)x^{B_q^{*}+1}}{\varphi(q)}\sum_{\chi\pmod{q}}\frac{1}{\varphi(q^{*})}\leq \frac{d_8(q)x^{B_q^{*}+1}d(q)}{\varphi(q)},
\end{multline*}
where $d(q)$ is the number of positive divisors of $q$.

Hence, by Theorem \ref{thm:versiothm4}, we have
\begin{multline*}
\left|\sum_{\substack{n\leq x\\ n\equiv c\pmod{q}}}G(n)-\frac{\mathfrak{S}_q(c)}{2}x^2\right|\leq 2\frac{d_8(q)x^{B_q^{*}+1}d(q)}{\varphi(q)}+\frac{d_{11}x(\log x)(\log q)}{\log 2}\\
+\frac{(4c_{q,3}+3)x(\log x)(\log\log x)}{2(\log x_0)(\log\log x_0)\varphi(q)}+\frac{2x(\log x)(\log\log x)f_6(q,x_0)}{\varphi(q)}+f_1(q,x_0)(\log{x})^5x^{2B_q^*(x)}\\
        +\left(0.35x_0f_1(q,x_0)+\left(1.443+\frac{0.5x_0}{\log{x_0}}\right)\left(f_4(q,T_1,x_0)+\frac{f_5(q,T_1,x_0)}{\varphi(q)}\right)\right)(\log{x})^4x^{2B_q^*(x)}
\end{multline*}

Finally, the proof of Proposition \ref{prop:JG} follows by using Lemma \ref{lem:versionlemma13} and putting everything together.
\end{proof}

\section{Explicit versions of all assumptions except \ref{assumption:psiuchi}}
\label{sec:ExplicitAssumptions}
In this section, we formulate the required explicit estimates to prove
the  main results in Section \ref{sec:resultsExplicit}. Essentially, we 
provide explicit versions of the assumptions described in Appendix 
\ref{appendix:assumptions} except for  \ref{assumption:psiuchi}.
This estimate is given in Section \ref{sec:explicitA5}. In particular in
Sections \ref{subsec:EcplicitNumbers} and \ref{sec:sumsoverzeros} we 
present some of the results in more general settings than is required 
here. The reason is that the more general results follow easily using
similar methods that we apply to prove the cases that are essential for 
our main results.

\subsection{On zero-free regions and exceptional zeros}
In this section, we describe useful known results about zero-free regions and possible Siegel zeros, and derive explicit estimates for the assumptions \ref{eq:ZeroFree}---\ref{eq:assumptionLogDer}.

The following three results tell us that for $q\leq 4\cdot10^5$ and $L$-functions associated to principal characters, the Generalized Riemann Hypothesis holds up to a certain height.
\begin{lemma}
\label{lemma:GRHHolds}
\cite[Theorem 10.1]{P2016}
If $L$ is a Dirichlet $L$-function associated with a primitive character modulo $q\leq 4\cdot10^5$, then GRH holds up to height 
$$
\max\left\{\frac{10^8}{q}, \frac{7.5\cdot10^7}{q}+200\right\}
$$
for even $q$ and up to height
$$
\max\left\{\frac{10^8}{q}, \frac{3.75\cdot10^7}{q}+200\right\}
$$
for odd $q$.
\end{lemma}

\begin{lemma}
\label{lemma:RH}
\cite[Theorem 1]{PlattTrudgian2021}
    The Riemann Hypothesis is true up to height $3\cdot 10^{12}$.
\end{lemma}

Note that if $\chi$ is a non-principal, imprimitive character modulo $q$, then there exists a Dirichlet $L$-function $L(s,\chi^*)$, where $\chi^*$ is a primitive character modulo $q^* \leq q$ that induces $\chi$. The functions $L(s,\chi)$ and $L(s,\chi^*)$ share the same non-trivial zeros. Moreover, we also recall that if $\chi$ is a principal character modulo $q$, then the non-trivial zeros of the function $L(s,\chi)$ are the same as the non-trivial zeros of the Riemann zeta function. Hence, the next corollary follows easily from Lemmas \ref{lemma:GRHHolds} and \ref{lemma:RH}.

\begin{corollary}
\label{corollary:GRH}
If $L(s,\chi)$ is a Dirichlet $L$-function associated with a character modulo $q\leq 4\cdot10^5$, then GRH holds up to height $293.75$. Moreover, if $L(s,\chi)$ is associated with a principal character modulo $q$, then RH holds up to height $3\cdot 10^{12}$.
\end{corollary}

However, the previous results do not say anything about the cases $q>4\cdot10^5$
or when the imaginary parts are large enough and the next lemmas describe what 
we do know in those cases.

\begin{lemma}\cite[Theorem 1.2]{Bellotti2024}
\label{lem:BellottiZeros}
    There are no zeros of $\zeta(\sigma+it)$ for $|t| \geq 3$ and 
    \begin{equation*}
        \sigma \geq 1-\frac{1}{53.989 (\log{|t|})^{2/3}(\log{\log{|t|}})^{1/3}}.
    \end{equation*}
\end{lemma}

\begin{lemma}
\cite[Theorem 1.1]{K2024}
\label{lemma:zeroFreeAll}
Let $q \geq 3$ be a positive integer, and let $\chi \pmod q$ be a Dirichlet character. The Dirichlet $L$-function $L(s,\chi)$ does not vanish in the region
\begin{equation*}
    \Re(s) \geq 1-\frac{1}{10.5\log q+61.5(\log {|\Im(s)|})^{2/3}\left(\log\log{|\Im(s)|}\right)^{1/3}}, \quad |\Im(s)| \geq 10.
\end{equation*}
\end{lemma}

\begin{lemma}
\cite[Theorem 1]{M1984}
\label{lemma:zeroFreeLarge}
Let $q$ be a positive integer. Let $L_q(s)$ be the product of the $\varphi(q)$ Dirichlet $L$-functions formed with characters modulo $q$. The function $L_q(s)$ has at most one single zero in the region
\begin{equation}
\label{eq:setMCCuurleryFree}
    \Re(s) \geq 1-\frac{1}{9.646\log\left(\max\{10,q,q|\Im(s)|\}\right)}.
\end{equation}
The possible zero must be simple and real and it must arise from an $L$-function formed with a real non-principal character modulo $q$.
\end{lemma}

We call the possible zero that is in the set \eqref{eq:setMCCuurleryFree} a \textit{Siegel zero}. According to the previous results, there may be a Siegel zero when $\chi$ is a real non-principal character. The following lemma describes the bounds known for the Siegel zero.

\begin{lemma}
(\cite[Theorem 1.3]{B2019} and \cite[Theorem 1.3]{B2020})
\label{lemma:SiegelBounds}
Assume that $\chi$ is a non-principal real character modulo $q$. If $q>4\cdot10^5$, then
\begin{equation*}
    \beta_1 \leq 
    \begin{cases}
        1-\frac{800}{\sqrt{q}\log^2 q}, &\text{ if } \chi \text{ is odd} \\
        1-\frac{100}{\sqrt{q}\log^2 q}, &\text{ if } \chi \text{ is even}.
    \end{cases}
\end{equation*}
\end{lemma}

Combining the previous results, we are able to provide an estimate for the term $B_q^*(x)$.
\begin{corollary}
\label{corollary:ZeroFree}
Let $q \geq 3$ be a positive integer, let $x\geq e^e$ be a positive real number
and let $L_q(s)$ be the product of the $\varphi(q)$ Dirichlet $L$-functions 
formed with characters modulo $q$. The function $L_q(s)$ has at most one
zero inside the region
\begin{equation}
\label{eq:explicitZeroFree}
     \Re(s) \geq 1-\frac{1}{10.5\log q+61.5(\log {x})^{2/3}\left(\log\log{x}\right)^{1/3}}.
\end{equation}
If the zero exists, then it  must be simple and real and it must arise from an
$L$-function formed with a real non-principal character modulo $q$.

Moreover, if
\begin{equation}
\label{eq:lowerXecplicit}
    x \geq \exp\left(\exp\left(0.5W\left(\frac{q^{3/2}\log^6{q}}{6150^3}\right)\right)\right),
\end{equation}
then the function $L_q(s)$ does not vanish in the set \eqref{eq:explicitZeroFree}.
\end{corollary}

\begin{proof}
    Due to Lemmas \ref{lemma:zeroFreeAll}, \ref{lemma:zeroFreeLarge} and \ref{lemma:SiegelBounds} and Corollary \ref{corollary:GRH}, it is sufficient to show
    \begin{equation*}
        9.646\log\left(\max\{10,q,q|\Im(s)|\}\right)\leq 10.5\log q+61.5(\log {10})^{2/3}\left(\log\log{10}\right)^{1/3}, \quad \text{if } |\Im(s)|<10,
    \end{equation*}
    and
    \begin{equation*}
       \frac{1}{10.5\log q+61.5(\log {x})^{2/3}\left(\log\log{x}\right)^{1/3}} \le \frac{100}{\sqrt{q}\log^2{q}}.
    \end{equation*}
    The first bound clearly holds since $9.646<10.5$ and $9.646\log{10}<61.5(\log {10})^{2/3}\left(\log\log{10}\right)^{1/3}$, and the second one holds because of the assumption \eqref{eq:lowerXecplicit}.
\end{proof}

Lastly, we describe a connection between a possible Siegel zero and the logarithmic derivative of a Dirichlet $L$-function.

\begin{lemma}
\label{lemma:explicitLogDerivativeDistance}
    Let $q \geq 4\cdot 10^5$ be an integer, $\chi$ be a primitive character $\pmod q$ and $\beta_1\neq 1$ a possible exceptional zero inside the set \eqref{eq:explicitZeroFree}. Let $\delta_1(\chi)=1$ if there exists an exceptional non-trivial zero in of set \eqref{eq:explicitZeroFree} for $L(s,\chi)$ and otherwise let $\delta_1=0$. Then
     \begin{equation*}
        \left|\frac{L'}{L}\left(1,\overline{\chi}\right)-\frac{\delta_1(\chi)}{1-\beta_1}\right|\leq 
        \begin{cases}
            2.651(\log{q})^3 &\text{if } \chi \text{ is non-real} \\
            0.025\sqrt{q}(\log{q})^2 &\text{if } \chi \text{ is real, odd and } 4\cdot10^5 \leq q <e^{24}-2 \\
            0.033\sqrt{q}(\log{q})^2 &\text{if } \chi \text{ is real, even and } 4\cdot10^5 \leq q <e^{24}-2 \\
            \frac{\sqrt{q}(\log{q})^2}{800}+0.548\sqrt{q}\log{q}&\text{if } \chi \text{ is real, odd and } e^{24}-2 < q < e^{92\pi}-2 \\
             \frac{\sqrt{q}(\log{q})^2}{100}+0.548\sqrt{q}\log{q}&\text{if } \chi \text{ is real, even and } e^{24}-2 < q < e^{92\pi}-2 \\
            0.004\sqrt{q}(\log{q})^2 &\text{if } \chi \text{ is real, odd and } e^{92\pi}-2<q \\
             0.012\sqrt{q}(\log{q})^2 &\text{if } \chi \text{ is real, even and } e^{92\pi}-2<q.
        \end{cases}.
    \end{equation*}
\end{lemma}

\begin{proof}
    Let us first assume  that $L(s,\chi)$ does not have any zero inside the region \eqref{eq:explicitZeroFree}. Since $\chi$ is a primitive character,  $\overline{\chi}$ is also a primitive character. Hence, by \cite[Lemmas 6.3 and 6.5]{BMOR2018} and \cite[Theorem 1.1]{MST2022}
    \begin{equation*}
        \left|\frac{L'}{L}\left(1,\overline{\chi}\right)\right|\leq 
        \begin{cases}
            2.651(\log{q})^3 &\text{if } \chi \text{ is non-real} \\
            0.023\sqrt{q}(\log{q})^2 &\text{if } \chi \text{ is real and } 4\cdot10^5 \leq q <e^{24}-2 \\
            0.548\sqrt{q}\log{q}&\text{if } \chi \text{ is real and } e^{24}-2 < q < e^{92\pi}-2 \\
            0.002\sqrt{q}(\log{q})^2 &\text{if } \chi \text{ is real and } e^{92\pi}-2<q.
        \end{cases}
    \end{equation*}
    
    Let us first assume that there is a Siegel zero $\beta_1$ of $L(s,\chi)$. By Corollary \ref{corollary:ZeroFree} the character $\chi$ must be real and hence $\chi=\overline{\chi}$. Hence, we must add the term $1/(1-\beta_1)$ to the previous results, and that follows from Lemma \ref{lemma:SiegelBounds}.
\end{proof}

Immediately, from the previous theorem, we get the following corollary.

\begin{corollary}
\label{cor:LDerivative}
    Assume the same hypothesis as in Lemma \ref{lemma:explicitLogDerivativeDistance}. Then
    \begin{equation*}
        \left|\frac{L'}{L}\left(1,\overline{\chi}\right)-\frac{\delta_1(\chi)}{1-\beta_1}\right|\leq 0.055\sqrt{q}(\log{q})^2.
    \end{equation*}
\end{corollary}

\subsection{On the number of zeros and the second Chebyshev function}
\label{subsec:EcplicitNumbers}
In Section \ref{sec:sumsoverzeros}, we are going to apply some estimates of the number of zeros of Dirichlet $L$-functions and those estimates are also needed to derive our main results. In addition, we need a bound for the second Chebyshev function. Hence, in this section we describe explicit estimates for assumptions \ref{eq:defIntervalZeroGeneral}---\ref{eq:assumptionLambda}.

Let us first consider the number of zeros. 

\begin{lemma}
\label{lemma:NumberOfZeros}
(\cite[Theorem 1.1 and Corollary 1.2]{BMOR2021} and \cite[Proposition 2.3]{BMOR2018})
Let $\chi$ be a character with conductor $q>1$ and let $T \geq 5/7$. Then
$$
\left|N(T,\chi)-\frac{T}{\pi}\log{\frac{qT}{2\pi e}}\right| \leq \min\{0.247\log{(qT)+6.894}, 0.298\log{(qT)+4.358}\}.
$$

Even more, let $\chi$ be a primitive character modulo $q$ and let us set $l:=\log{\frac{q(T+2)}{2\pi}}$. If $l \leq 1.567$, then $N(T,\chi)=0$. If $l>1.567$, then
$$
\left|N(T,\chi)-\left(\frac{T}{\pi}\log{\frac{qT}{2\pi e}}-\frac{\chi(-1)}{4}\right)\right| \leq 0.22737l+2\log{(1+l)}-0.5.
$$
If $\chi$ is a principal character and $T >e$, then
\begin{equation}
\label{eq:numberofPrincipalZeros}
    \left|N(T,\chi)-\left(\frac{T}{\pi}\log{\frac{T}{2\pi e}}+\frac{7}{4}\right)\right| \leq 0.34\log{T}+3.996+\frac{25}{24\pi T}
\end{equation}
\end{lemma}

Using the previous estimate, we provide explicit versions of assumptions \ref{eq:defIntervalZeroGeneral} and \ref{eq:defzerosGeneral}.

\begin{corollary}
\label{cor:numberOfZeros}
Let $q \geq 4\cdot 10^5$, $T \geq e^e$ and $\chi$ be a primitive or principal character modulo $q$. Then 
\begin{equation*}
    N(T,\chi) \leq 0.364T\log{\frac{qT}{2\pi e}}.
\end{equation*}
If we have $q=1$ instead, for $T \geq 2\pi e+1$ we have
\begin{equation*}
    N(T,\chi) \leq 6.879T\log{\frac{T}{2\pi e}}.
\end{equation*}
\end{corollary}

\begin{proof}
The claim follows when we apply the second and third results of Lemma \ref{lemma:NumberOfZeros} and keep in mind the lower bounds for $q$ and $T$.
\end{proof}

\begin{lemma}
\label{lemma:ExplicitZeroInterval}
Let $\chi$ be a primitive character with conductor $q \geq 3$ and let $T \geq 9$ be a real number. Then we have 
\begin{equation*}
    \sum_{\substack{{\rho_\chi} \\ T < |\gamma_{\chi}| \leq T+1}} 1<\left(\frac{1}{\pi}+0.45474\right)\log{(q(T+1))}+4\log{\log{(q(T+1))}}-2.357,
\end{equation*}
where the sum runs over non-trivial zeros of the function $L(s,\chi)$. If $\chi$ is a principal character, then
\begin{equation*}
    \sum_{\substack{{\rho_\chi} \\ T < |\gamma_{\chi}| \leq T+1}} 1<\left(\frac{1}{\pi}+0.68\right)\log{(T+1)}+7.407+\frac{25}{12\pi T}.
\end{equation*}
\end{lemma}

\begin{proof}
Let us first consider the case of primitive characters. Similarly as in the proof of Lemma 22 in \cite{EHP2022}, using Lemma \ref{lemma:NumberOfZeros} and bounds
\begin{equation*}
    \frac{T}{\pi}\log\left(1+\frac{1}{T}\right) \leq \frac{1}{\pi},
\end{equation*}
\begin{equation*}
    \log{(T+2)(T+3)} \leq \log \frac{33(T+1)^2}{25}
\end{equation*}
and
\begin{equation*}
    \log\log\left(\frac{qe(T+3)}{2\pi}\right)+\log\log\left(\frac{qe(T+2)}{2\pi}\right)<2\log\log\left(\frac{qe(T+3)}{2\pi}\right)<2\log\log\left(q(T+1)\right),
\end{equation*}
we get
\begin{align*}
\sum_{\substack{{\rho_\chi} \\ T < |\gamma_{\chi}| \leq T+1}} 1 \leq & \frac{T+1}{\pi}\log{\frac{q(T+1)}{2\pi e}}-\frac{\chi(-1)}{4}+0.22737\log{\frac{q(T+3)}{2\pi}} +2\log{\left(1+\log{\frac{q(T+3)}{2\pi}}\right)}-0.5\\
&\quad-\frac{T}{\pi}\log{\frac{qT}{2\pi e}}+\frac{\chi(-1)}{4} +0.22737\log{\frac{q(T+2)}{2\pi}}+2\log{\left(1+\log{\frac{q(T+2)}{2\pi}}\right)}-0.5 \\
&\quad<\left(\frac{1}{\pi}+2\cdot0.22737\right)\log{(q(T+1))}+4\log{\log{(q(T+1))}} \\
&\quad+\frac{1}{\pi}-\frac{\log(2\pi e)}{\pi}+0.22737\log\left(\frac{33}{25\cdot(2\pi)^2}\right)-1 \\
& \quad <\left(\frac{1}{\pi}+0.45474\right)\log{(q(T+1))}+4\log{\log{(q(T+1))}}-2.357.
\end{align*}

The case of the principal character follows similarly when we use the estimate \eqref{eq:numberofPrincipalZeros}.
\end{proof}

Immediately from Lemma \ref{lemma:ExplicitZeroInterval} we obtain the following corollary.

\begin{corollary}
\label{cor:zerosBetween}
    Let $\chi$ be a primitive or a principal character with a conductor $q \geq 4\cdot10^5$, and let $T \geq e^e$. Then 
    \begin{equation*}
    \sum_{\substack{{\rho_\chi} \\ T < |\gamma_{\chi}| \leq T+1}} 1<1.325\log{(q(T+1))}.
\end{equation*}
Moreover, in the case $q=1$ and $T \geq 2 \pi e+1$ we have
$$
\sum_{\substack{{\rho_\zeta} \\ T < |\gamma_{\zeta}| \leq T+1}} 1<3.523\log{(T+1)}.
$$
\end{corollary}

Lastly, we provide the following result for the sum of von Mangoldt functions.

\begin{lemma} \cite[Theorem 12]{RS1962}
\label{lemma:explicitMangoldt}
Let $x > 0$. Then
\begin{equation*}
    \sum_{n \leq x} \Lambda(n) \leq 1.03883x.
\end{equation*}
\end{lemma}

\subsection{Sums over zeros}
\label{sec:sumsoverzeros}

In this section, we provide several explicit results concerning to sums over zeros. They are later used to prove our main explicit results.

The first results describe the sum of $1/{\rho_\chi}$ over non-trivial zeros. 

\begin{lemma}
\label{lemma:principalrhoT}
    \cite[Lemma 8]{BPT2022} If $T \geq 4\pi e$, then
    \begin{equation*}
        \sum_{0<\Im(\rho_\zeta) \leq T} \frac{1}{|{\rho_\zeta}|} \leq \frac{1}{4\pi}\left(\log{\frac{T}{2\pi}}\right)^2,
    \end{equation*}
    where the sum runs over the non-trivial zeros of the Riemann zeta function.
\end{lemma}

\begin{lemma}
\label{lemma:ZeroSums}
Let $q\geq 3$ be an integer, $T \geq 5/7$ be a real number and $\chi$ be a primitive character modulo $q$. 

Then
\begin{align*}
 &\sum_{\substack{ |\gamma_{\chi}| \leq T}} \frac{1}{|{\rho_\chi}|}=\sum_{\substack{{\rho_\chi} \neq 1-\beta_ 1 \\ |\gamma_{\chi}| \leq T}} \frac{1}{|{\rho_\chi}|} \leq \frac{\log^2{T}}{2\pi}+\frac{\log T}{\pi} \cdot\log{\frac{q}{2\pi}}-\frac{0.019\log{q}}{T} \\
 &+\frac{0.22737\log{(T+2)}+2\log\log\frac{eq(2+T)}{2\pi}-7.954}{T}+
 \begin{cases}
0.843 \log q+ 5.153 & \text{if } 3 \leq q \leq 11 \\
 1.116\log{q}+1.2\log\log(1.175q)+4.492 & \text{if } 12 \leq q \leq 4\cdot10^5
 \end{cases},
\end{align*}
if $3 \leq q \leq 4\cdot10^5$. 
Moreover, if $q>4\cdot10^5$, then
\begin{multline}
\label{eq:sum1RhoqLarge}
 \sum_{\substack{{\rho_\chi} \neq 1-\beta_ 1 \\ |\gamma_{\chi}| \leq T}} \frac{1}{|{\rho_\chi}|} \leq  \frac{\log^2{T}}{2\pi}+\frac{\log T}{\pi} \cdot\log{\frac{q}{2\pi}}+ 2.194\log^2{q}+9.646\log{(q)}\cdot\log\log{\left(\frac{19eq}{14\pi}\right)}-5.017 \log{q} \\
 -0.8\log\log{\left(\frac{19eq}{14\pi}\right)}+9.121+\frac{0.22737\log{(T+2)}-0.019\log{q}+2\log{\log{\left(\frac{eq(T+2)}{2\pi}\right)}}-7.953}{T}
\end{multline}
and if we include a Siegel zero, then we add
\begin{equation*}
    \begin{cases}
        0 & \text{if } \chi \text{ is not real} \\
        \frac{\sqrt{q}\log^2{q}}{800}, & \text{if } \chi \text{ is real and odd} \\
        \frac{\sqrt{q}\log^2{q}}{100}, & \text{if } \chi \text{ is real and even}
    \end{cases}
\end{equation*}
The sums run over non-trivial zeros of the function $L(s,\chi)$.
\end{lemma}

\begin{proof}
First,  by Lemma \ref{lemma:zeroFreeLarge} and Corollary \ref{corollary:GRH}, there is no Siegel zero when $q>4\cdot10^5$ and $\chi$ is a non-real or a principal character modulo $q$ or when $3 \leq q \leq 4\cdot10^5$.  In these cases we have
\begin{equation}
\label{eq:SumEquality}
     \sum_{\substack{ |\gamma_{\chi}| \leq T}} \frac{1}{|{\rho_\chi}|}=\sum_{\substack{{\rho_\chi} \neq 1-\beta_ 1 \\ |\gamma_{\chi}| \leq T}} \frac{1}{|{\rho_\chi}|}.
\end{equation}
Further, in general, we have
\begin{equation}
\label{eq:sumsGamma}
    \sum_{\substack{ |\gamma_{\chi}| \leq T}} \frac{1}{|{\rho_\chi}|}=\sum_{\substack{{\rho_\chi} \neq 1-\beta_ 1 \\ |\gamma_{\chi}| \leq T}} \frac{1}{|{\rho_\chi}|}+\frac{1}{1-\beta_1}.
\end{equation}
Now, keeping the previous facts in mind, we consider the sums in two cases: 
depending on whether or not  $|\gamma_{\chi}| \leq T_0$  for some $T_0 \in 
[5/7,293.75]$. In practice, $T_0$ will be chosen to be small.

Let us start with the case $|\gamma_{\chi}| \leq T_0$.
By Corollary \ref{corollary:GRH}, the GRH holds up to height $293.75$ if $q \leq 4\cdot10^5$. Hence, we even have
\begin{equation}
\label{eq:GRHsum}
    \sum_{\substack{ |\gamma_{\chi}| \leq T_0}} \frac{1}{|{\rho_\chi}|}=\sum_{\substack{\beta=1/2 \\ |\gamma_{\chi}| \leq T_0}} \frac{1}{|{\rho_\chi}|}\leq 2N(T_0, \chi), \quad \text{if } q \leq 4\cdot10^5.
\end{equation}
Further, if $q > 4\cdot10^5$ and the Siegel zero exists, then by Lemmas \ref{lemma:zeroFreeLarge} and \ref{lemma:SiegelBounds} we get
\begin{equation}
\label{eq:beta1Est}
    \frac{1}{1-\beta_1} \leq 
    \begin{cases}
        \frac{\sqrt{q}\log^2{q}}{800}, & \text{if } \chi \text{ is real and odd} \\
        \frac{\sqrt{q}\log^2{q}}{100}, & \text{if } \chi \text{ is real and even}.
    \end{cases}
\end{equation}
Note now that since zeros lie symmetrically with  respect to the line $s=1/2$, we have
\begin{equation*}
    N_1\left(9.646\log{q}+2\right)+2N_2\leq 2N_1+N_2+9.646\log{q}\cdot\left(\frac{N_2}{2}+ N_1\right)\leq N(T_0,\chi)+9.646\left(\log{q}\right)\frac{N(T_0,\chi)}{2},
\end{equation*}
where $N_1$ is the number of  zeros with $\Re({\rho_\chi}) <1/2$ and $|\Im({\rho_\chi})| \leq T_0$, and $N_2$ is the number of the zeros with $\Re({\rho_\chi})=1/2$ and  $|\Im({\rho_\chi})| \leq T_0$. 
Moreover, by Lemma \ref{lemma:zeroFreeLarge}, we have
\begin{equation}
\label{eq:minusSum}
    \sum_{\substack{{\rho_\chi} \neq 1-\beta_ 1 \\ |\gamma_{\chi}| \leq T_0}} \frac{1}{|{\rho_\chi}|}\leq \sum_{\substack{{\rho_\chi} \neq \beta_ 1 \\ |\gamma_{\chi}| \leq T_0}} \frac{1}{1-\beta}= \sum_{\substack{{\rho_\chi} \neq \beta_ 1 \\ 1-\beta < 1/2 \\ |\gamma_{\chi}| \leq T_0}}\left(\frac{1}{1-\beta}+\frac{1}{\beta}\right)+\sum_{\substack{{\rho_\chi} \neq \beta_ 1 \\ 1-\beta = 1/2 \\ |\gamma_{\chi}| \leq T_0}} 2\leq 9.646\left(\log{q}\right)\frac{N(T_0,\chi)}{2}+N(T_0,\chi) 
\end{equation}
for $q >4\cdot 10^5$ and $T_0 \leq 1$ (which will be seen to be a good choice later). Hence, we are ready with the case $|\gamma_{\chi}| \leq T_0$. 

Let us now consider the case $|\gamma_{\chi}|>T_0$. Now we have
\begin{equation}
\label{eq:sumT0T}
    \sum_{\substack{ T_0<|\gamma_{\chi}| \leq T}} \frac{1}{|{\rho_\chi}|} < \sum_{T_0<\substack{ |\gamma_{\chi}| \leq T}} \frac{1}{|\gamma_{\chi}|},
\end{equation}
and by the second paragraph of the proof of Lemma 30 in \cite{EHP2022} which applies the first result mentioned in Lemma \ref{lemma:NumberOfZeros}, the right-hand side is
\begin{multline}
\label{eq:estT0T}
    \leq \frac{N(T,\chi)}{T}-\frac{N(T_0, \chi)}{T_0}+\frac{\log^2{T}}{2\pi}+\frac{\log T}{\pi} \cdot\log{\frac{q}{2\pi e}}-\frac{\log^2{T_0}}{2\pi}-\frac{\log T_0}{\pi} \cdot\log{\frac{q}{2\pi e}}\\
    +\int_{T_0}^T \frac{\alpha_1(q,t)\log{(qt)}+\alpha_2(q,t)}{t^2}\, dt,
\end{multline}
where
\begin{equation}
\label{eq:defalpha1alpha2}
\alpha_1(q,t)=
\begin{cases}
0.298 & \text{if } 3\leq q \leq 4\cdot10^5 \text{ and } t\leq 10^{16} \\
0.247 & \text{if } 3\leq q \leq 4\cdot10^5 \text{ and } t> 10^{16} \\
0.247 & \text{if } q>4\cdot10^5
\end{cases},
\quad
\alpha_2(q,t)=
\begin{cases}
4.358 & \text{if } 3\leq q \leq 4\cdot10^5 \text{ and } t\leq 10^{16} \\
6.894 & \text{if } 3\leq q \leq 4\cdot10^5 \text{ and } t> 10^{16} \\
6.894 & \text{if } q>4\cdot10^5
\end{cases}.
\end{equation}

Now we put the previous estimates together. Indeed, in the case $3\leq q \leq 4\cdot10^5$, we apply estimates \eqref{eq:SumEquality}, \eqref{eq:GRHsum}, \eqref{eq:sumT0T} and \eqref{eq:estT0T}, and in the case $q>4\cdot10^5$, we apply estimates \eqref{eq:SumEquality}, \eqref{eq:sumsGamma}, \eqref{eq:beta1Est}, \eqref{eq:minusSum}, \eqref{eq:sumT0T} and \eqref{eq:estT0T}. Moreover, we use Lemma \ref{lemma:NumberOfZeros} to estimate the number of zeros. Further, let us set $T_0=5/7$. Comparing the results for $T \leq 10^{16}$ and $T>10^{16}$, we obtain the desired result.
\end{proof}

\begin{corollary}
\label{cor:RhoNoChi}
Let $T \geq q >4\cdot10^5$ and let $\chi$ be a primitive or a principal character modulo $q$. Then
    \begin{equation*}
        \sum_{\substack{ |\gamma_{\chi}| \leq T \\ \rho_\chi \neq 1-\beta_1}} \frac{1}{|{\rho_\chi}|} \leq \frac{\log^2{T}}{2\pi}+4.434\log{q}\log{T}.
        \end{equation*}
\end{corollary}

\begin{proof}
    Note that the sum of the last four terms in \eqref{eq:sum1RhoqLarge} is negative. Using that estimate, we get an estimate for primitive characters. Since Dirichlet $L$-functions associated with principal characters share the non-trivial zeros with the Riemann zeta function and the Riemann zeta function does not have any non-trivial real zeros, the claim for principal characters follows using Lemma \ref{lemma:principalrhoT} and symmetry of the non-trivial zeros of the Riemann zeta function with respect to the real line.
\end{proof}

The next results are in similar style as the previous ones but we have a different sum over zeros.

\begin{lemma}
\label{lemma:rho2Zeta}
\cite[Lemma 5]{BPT2022}, \cite[Lemma 2.9]{STD2015}
Let $T \geq 2\pi e$. Then
\begin{equation*}
     \sum_{\Im(\rho_{\zeta})>T} \frac{1}{\Im(\rho_\zeta)^2} \leq \frac{\log{T}}{2\pi T} \quad \text{and}\quad \sum_{0<\Im(\rho_\zeta)} \frac{1}{\Im(\rho_\zeta)^2}<0.023105,
\end{equation*}
where the sums run over non-trivial zeros of the Riemann zeta function.
\end{lemma}

\begin{lemma}
\label{lemma:rhosquares}
Assume that $T \geq 5/7$ is a real number, $q \geq 3$ is an integer and $\chi$ is a primitive character modulo $q$. Then, we have
\begin{equation}
\label{eq:SquaresAndProduct}
\sum_{\substack{{\rho_\chi } \\ |\Im({\rho_\chi})|>T}} \frac{1}{|{\rho_\chi}|^2}\leq \frac{1}{\pi T}\log{\frac{qT}{2\pi}}+\frac{0.494\log{(qT)}+13.912}{T^2}
\end{equation}
and 
\begin{equation*}
    \sum_{{\rho_\chi \neq 1-\beta_1}} \frac{1}{|{\rho_\chi}({\rho_\chi}+1)|}
    <
\begin{cases}
0.894 \log{q}+3.107 &\text{if } 3 \leq q \leq 9 \\
1.116 \log{q}+4.637 &\text{if } 10 \leq q \leq 4\cdot10^5 \\
2.194 \log^2{q}+9.646\log{q}\log{\log{q}} &\text{if } q>4\cdot10^5.
\end{cases},
\end{equation*}
where the sums run over non-trivial zeros of the function $L(s,\chi)$. 
If there exists a  Siegel zero and we include the zero $1-\beta_1$, then first two estimates do not change and the last estimate changes if $\chi$ is real. If $\chi$ is real, then we add $\frac{\sqrt{q}(\log{q})^2}{800}$ if $\chi$ is odd and $\frac{\sqrt{q}(\log{q})^2}{100}$ if $\chi$ is even, to the last line.
\end{lemma}

\begin{proof}
Let us start with the first sum in \eqref{eq:SquaresAndProduct}. By partial summation and Lemma \ref{lemma:NumberOfZeros} we have
\begin{multline*}
\sum_{\substack{{\rho_\chi} \\ |\Im({\rho_\chi})|>T}} \frac{1}{|{\rho_\chi}|^2}\leq \sum_{\substack{\gamma_{\chi} \\ |\gamma_{\chi}|>T}} \frac{1}{|\gamma_{\chi}|^2} \leq -\frac{N(T,\chi)}{T^2}+\int_{T}^\infty \frac{2N(t,\chi)}{t^3} \, dt \leq -\frac{N(T,\chi)}{T^2} \\
+\int_{T}^\infty \frac{2}{t^3}\left(\frac{t}{\pi}\log{\frac{qt}{2\pi e}}+\alpha_1(q,t)\log{(qt)}+\alpha_2(q,t)\right) \, dt = \frac{2}{\pi T}\log{\frac{qT}{2\pi}}-\frac{N(T,\chi)}{T^2}+\int_{T}^\infty \frac{2}{t^3}\left(\alpha_1(q,t)\log{(qt)}+\alpha_2(q,t)\right) \, dt,
\end{multline*}
where $\alpha_1(q,t)$ and $\alpha_2(q,t)$ are defined as in \eqref{eq:defalpha1alpha2}.

Let us now move on to estimate the second sum in \eqref{eq:SquaresAndProduct}. Let us first notice that when $|\Im({\rho_\chi})|\leq 293.75$, then 
\begin{equation}
\label{eq:ZerosSmall}
\frac{1}{|{\rho_\chi}(1+{\rho_\chi})|} \leq 
\begin{cases}
\frac{1}{\frac{3}{4}+\gamma_{\chi}^2} &\text{if } 3 \leq q \leq 4\cdot 10^5 \\
\frac{1}{(9.646\log{(\max\{q, q|\Im({\rho_\chi})|\}))^{-1}}+\gamma_{\chi}^2} &\text{if } q>4\cdot10^5 \text{ and } \chi \text{ is non-real}\\
\frac{\sqrt{q}(\log{q})^2}{800} &\text{if } q>4\cdot10^5 \text{ and } \chi \text{ is real and odd} \\
\frac{\sqrt{q}(\log{q})^2}{100} &\text{if } q>4\cdot10^5 \text{ and } \chi \text{ is real and even} 
\end{cases}.
\end{equation}
by Corollary \ref{corollary:GRH} and Lemmas \ref{lemma:zeroFreeLarge} and \ref{lemma:SiegelBounds}.

Hence, using the first paragraph, Lemma \ref{lemma:zeroFreeLarge} and recalling that the zeros lie symmetrically with  respect to the line $\Re(s)=1/2$, we have
\begin{multline}
\label{eq:ZerosExp}
\sum_{\substack{{\rho_\chi}}} \frac{1}{|{\rho_\chi}(1+{\rho_\chi})|}=\left(\sum_{\substack{{\rho_\chi} \\ |\Im({\rho_\chi})| \leq T}}+\sum_{\substack{{\rho_\chi} \\ |\Im({\rho_\chi})| >T}}\right) \frac{1}{|{\rho_\chi}(1+{\rho_\chi})|} \\
\leq \frac{2}{\pi T}\log{\frac{qT}{2\pi}}-\frac{N(T,\chi)}{T^2}+\int_{T}^\infty \frac{2}{t^3}\left(\alpha_1(qt)\log{(qt)}+\alpha_2(qt)\right) \, dt\\+
\begin{cases}
\frac{4N(T,\chi)}{3} &\text{if } 3 \leq q \leq 4\cdot10^5 \\
\frac{9.646}{2}\log{(\max\{q, qT\})}N(T,\chi)+\frac{4N(T,\chi)}{6} &\text{if } q>4\cdot10^5 \text{ and } \chi \text{ is non-real} \\
\frac{\sqrt{q}(\log{q})^2}{800}+9.646\log{(\max\{q, qT\})}(\frac{N(T,\chi)}{2}-1)+\frac{4N(T,\chi)}{6}&\text{if } q>4\cdot10^5 \text{ and } \chi \text{ is real and odd} \\
\frac{\sqrt{q}(\log{q})^2}{100}+9.646\log{(\max\{q, qT\})}(\frac{N(T,\chi)}{2}-1)+\frac{4N(T,\chi)}{6}&\text{if } q>4\cdot10^5 \text{ and } \chi \text{ is real and even}
\end{cases}
.
\end{multline}
The second estimate follows when we use the second result in Lemma \ref{lemma:NumberOfZeros}, and set
$T=1.035149976$
if $3\leq q \leq 4\cdot10^5$ and $T=5/7$ if $q>4\cdot10^5$. In addition we recognize that if we do not include the zero $1-\beta_1$ to the sum, then the second estimate holds for all $q>4\cdot10^5$.

In the case $q\leq 4\cdot 10^5$, we get 
\[
0.2782 \log{q} + 4.2158
\]
from the integral. Then we still have to bound
\[
\frac{2}{\pi T}\log{\frac{qT}{2\pi}}+\left(\frac{4}{3}-\frac{1}{T^2}\right)N(T,\chi)
\]
If $3\leq q\leq 9$, then $l=\log\frac{q(T+2)}{2\pi}<1.5$, so $N(T,\chi)=0$, and hence, this reduces to the term $\frac{2}{\pi T}\log{\frac{qT}{2\pi}}<0.61501 \log{q} - 1.109$. This means, that in total, for $3\leq q\leq 9$, we get
\[
0.27811 \log{q} + 4.2158+0.61501 \log(q) - 1.109<0.894 \log{q} + 3.107.
\]
In the case of $10\leq q\leq 4\cdot 10^{5}$, we proceed similarly except in the end when we need to estimate the term $\frac{2}{\pi T}\log{\frac{qT}{2\pi}}+\left(\frac{4}{3}-\frac{t}{T^2}\right)N(T,\chi)$. We have
\begin{multline*}
\frac{2}{\pi T}\log{\frac{qT}{2\pi}}+\left(\frac{4}{3}-\frac{1}{T^2}\right)N(T,\chi)\\
<\frac{2}{\pi T}\log{\frac{qT}{2\pi}}+\left(\frac{4}{3}-\frac{1}{T^2}\right)\left(\frac{T}{\pi}\log \frac{qT}{2\pi e}+\frac{1}{4}+0.22737 \log\frac{q(T+2)}{2\pi}+2\log(1+\log\frac{q(T+2)}{2\pi})-0.5\right)\\
<0.837802 \log{q} + 0.800186 \log(\log{q} + 0.272384) - 1.64483
\end{multline*}
Now we have the final estimate
\begin{multline*}
0.837802 \log{q} + 0.800186 \log(\log{q} + 0.272384) - 1.64483+0.27811 \log{q} + 4.2158\\
<1.116 \log{q} + 0.801 \log(\log{q} + 0.273) + 2.571< 1.116 \log{q}+4.637.
\end{multline*}

Let us now move to the case with $q> 4\cdot 10^5$. Let us start with evaluating the integral. We get
\begin{equation}
\label{eq:FirstIntegral}
    0.48412 \log{q} + 13.5914.
\end{equation}
Now $\max\{q,qT\}=q$. Hence we need to bound the expression
\begin{multline*}
\frac{2}{\pi T}\log{\frac{qT}{2\pi}}-\frac{N(T,\chi)}{T^2}+9.646\log{(\max\{q, qT\})}\frac{N(T,\chi)}{2}+\frac{4N(T,\chi)}{6}+0.48412 \log{q} + 13.5914\\
=\frac{2}{\pi T}\log{\frac{qT}{2\pi}}+N(T,\chi)\left(-\frac{1}{T^2}+\frac{9.646}{2}\log{q}+\frac{2}{3}\right)+0.48412 \log{q} + 13.5914\\
<\frac{2}{\pi T}\log{\frac{qT}{2\pi}}+0.48412 \log{q} + 13.5914\\
+\left(-\frac{1}{T^2}+\frac{9.646}{2}\log{q}+\frac{2}{3}\right)\left(\frac{T}{\pi}\log{\frac{qT}{2\pi e}}+0.22737\log{\frac{q(T+2)}{2\pi}}+2\log{\left(\log{\frac{q e(T+2)}{2\pi}}\right)}-0.25\right)\\ <2.194 \log^2{q}+9.646\log{q}\log{\log{q}}.
\end{multline*}
\end{proof}

\begin{remark}
    Even though we have slightly different formulas for $q>4\cdot10^5$ in \eqref{eq:ZerosExp} for real and non-real characters, we state the same results for them when $\rho \neq 1-\beta_1$. The reason is that both of the estimates give the same simplified upper bound.
\end{remark}

Immediately from the previous lemma we get the following corollary.
\begin{corollary}
\label{cor:sumRhoRhoGen}
    Assume that $\chi$ is a primitive or principal character modulo $q\geq 4\cdot10^5$ and $T \geq e^e$. Then
    \begin{equation}
\label{eq:SquaresAndProduct}
\sum_{\substack{{\rho_\chi } \\ |\Im({\rho_\chi})|>T}} \frac{1}{|{\rho_\chi}|^2}\leq \frac{0.430}{T}\log{\frac{qT}{2\pi}}.
\end{equation}
\end{corollary}

\section{Explicit estimates for the function $\psi(u,\chi)$}
\label{sec:explicitA5}

 In this section, we provide explicit version for the assumption \ref{assumption:psiuchi}. We would like to note to the reader that these generalize the explicit estimates in \cite{CH2023} to cover cases $q \neq 1$, $x$ is non-half odd integer and $x \geq T$.

First we provide a useful lemma that allows us to estimate the function $\psi(u,\chi)$ with the function $I(u,T,\chi)$ defined as
\begin{equation}
\label{eq:defJ}
I(u,T,\chi):=\frac{1}{2\pi i}\int_{c-iT}^{c+iT} \left(-\frac{L'(s,\chi)}{L(s,\chi)}\right)\frac{u^s}{s} ds,
\end{equation}
with $c=1+1/\log u$.

\begin{lemma}
\label{lemma:principal1}
    Let $q \geq 7$ or $q=1$, $u \geq 2$ , $T\geq 4\cdot10^5$ and $\chi$ be a primitive (possibly principal) character modulo $q$. Then
    \begin{multline*}
          \left|\psi(u,\chi)-I(u,T,\chi)\right|< \log u(\log q+1)+\frac{\log u+1}{T}+\frac{eu\log{u}}{T\log{(5/4)}}+\left(2e+\frac{\log{3}}{\log{2}}\right)\log{u}\\+\left(\left(\frac{5}{4}\right)^{1+\frac{1}{\log{2}}}+1\right)\frac{2u}{T}\log{\left(u\right)}\log\log{\left(u\right)}+\frac{19.498u}{T}\log{\left(u\right)}\\
          +\frac{0.645\sqrt{u}\log^2u}{T}+\frac{1.156\sqrt{u}\log u}{T},
    \end{multline*}
    where $I(u,T,\chi)$ is as in \eqref{eq:defJ}, with $\zeta$ in the place of the $L$-function in the case of the principal character and where  
    the term $\log q \log u$ can be omitted in the case of a non-principal character.
\end{lemma}

First we describe the result for the principal character.
   
\begin{proposition}
\label{lemma:principal2}
    Let $q \geq 7$ or $q=1$, $u \geq 2$, $T\geq 4\cdot10^5$ and $\chi_0$ be the principal character modulo $q$. Then 
    \begin{align*}
          &\left|\psi(u,\chi_0)-u+\log{(2\pi)}+\sum_{|\gamma_\zeta|<T}\frac{u^{\rho_\zeta}}{\rho_\zeta}\right|<\left(\left(\frac{5}{4}\right)^{1+\frac{1}{\log{2}}}+1\right)\frac{2u}{T}\log{\left(u\right)}\log\log{\left(u\right)}+\log u(\log q+1) \\
          &\quad+\frac{\log u+1}{T}+\frac{eu\log{u}}{T\log{(5/4)}}+\left(2e+\frac{\log{3}}{\log{2}}\right)\log{u}+\frac{19.498u}{T}\log{\left(u\right)}+\frac{0.645\sqrt{u}\log^2u}{T}\\
          &\quad+\frac{1.156\sqrt{u}\log u}{T}+\frac{4u\log{T}}{T-1}+\frac{2eu\log{u}}{\pi(T-1)}
        +\frac{e u}{\pi(T-1) \log{u}}\left(\log^2{(T+1)}+20\log{(T+1)}\right)\\
        &\quad+\frac{18+2\log{\sqrt{(u+1)^2+T^2}}}{2\pi u^2}+\frac{18+2\log{\sqrt{(u+1)^2+(T+1)^2}}}{2\pi u^2(T-1)}+\frac{9+\frac{1}{2}\log{\left((u+1)^2+(T+1)^2\right)}}{\pi u (T-1)} +\frac{1}{6}.
    \end{align*}
    Here the sum runs over non-trivial zeros of the Riemann zeta function.
\end{proposition}

The following corollary provides a simplified version of the result in the case of principal character.

\begin{corollary}
\label{cor:psiPrincipal}
    Let $q\geq 7$ or $q=1$, $T \geq 4\cdot10^5$, $2 \leq u \leq 3T$ and $\chi_0$ be the principal character modulo $q$. Then
    \begin{multline*}
        \left|\psi(u,\chi_0)-u+\log{(2\pi)}+\sum_{|\gamma_\zeta|<T}\frac{u^{\rho_\zeta}}{\rho_\zeta}\right|<2\left(\left(\frac{5}{4}\right)^{1+\frac{1}{\log{2}}}+1\right)\frac{u\log{u}\log\log{u}}{T} \\
        +\frac{34.544 u\log u}{T}+ 18.249\log(T)+(\log{q}+8.022)\log{u}.
    \end{multline*}
\end{corollary}

Lastly, we provide the result for non-principal, primitive characters. In the result, we use the definition
\begin{equation}
\label{def:Bchi}
    B(\chi)=-\sum_{\rho} \frac{1}{\rho}.
\end{equation}
\begin{proposition}
\label{lemma:mainLemmaPsiuchi}
Assume $q\geq 7$, $\chi$ is a non-principal, primitive character modulo $q$, $u\geq 2$ and $T \geq 4\cdot10^5$. Then, we have
\begin{align*}
    &\left|\psi(u,\chi)-C(\chi)+\sum_{|\Im\rho|<T} \frac{u^\rho}{\rho}\right|<+\left(\left(\frac{5}{4}\right)^{1+\frac{1}{\log{2}}}+1\right)\frac{2u}{T}\log{\left(u\right)}\log\log{\left(u\right)}+ \log u+\frac{\log u+1}{T}\\
   &\quad+\frac{eu\log{u}}{T\log{(5/4)}}+\left(2e+\frac{\log{3}}{\log{2}}\right)\log{u}+\frac{19.498u\log(u)}{T}+\frac{0.645\sqrt{u}\log^2u}{T}\\
   &\quad +\frac{1.156\sqrt{u}\log u}{T}+\left(\log{u}+C_0+\frac{0.478}{\log{u}}\right)\frac{2ue}{\pi(T-1)}+\frac{u(1.092\log{(qT)}+4\log{\log{(qT)}}-3.005)}{T-1}\\
   &\quad +\frac{ue}{\pi (T-1)\log{u}}\left(\vphantom{\frac{15}{2}\Re B(\chi)}3.578 \log^2(q T) + 48 \log^2(\log(q T)) + 26.208 \log(\log(q T)) \log(q T) 
    \right.\\ 
   &\quad \left.+  3\left(\frac{1}{(T-1)^2}+\frac{\pi}{4(T-1)}\right)-16.412\log(q T) -60.12 \log(\log(q T)) + 22.921+\frac{15}{8}\log\left(\frac{T^2}{4}+\frac{T}{2}+\frac{5}{2}\right)  \right.\\ 
   &\quad\left.+\frac{15}{4}\log \frac{q}{\pi}-\frac{15}{2}\Re B(\chi)\right)+\frac{3\log{\left(\frac{T+1}{2}\right)}+\log{(q\pi)}+3.570+C_0+\frac{8}{T-1}}{\pi (T-1)u\log{u}}+\frac{1}{6},
\end{align*}
where the term $\Re B(\chi)$ is given in \eqref{def:Bchi} and bounded as in Lemma \ref{lem:Bchi}. 
\end{proposition}

\begin{corollary}
\label{cor:psiPrimitive}
    Assume $q\geq 7$, $\chi$ is a non-principal, primitive character modulo $q>4\cdot10^5$, $T \geq q$ and $2\leq u \leq 3T$. Then, we have 
    \begin{multline*}
        \left|\psi(u,\chi)-C(\chi)+\sum_{|\Im\rho|<T} \frac{u^\rho}{\rho}\right|<2\left(\left(\frac{5}{4}\right)^{1+\frac{1}{\log{2}}}+1\right)\frac{u\log{u}\log\log{u}}{T} \\
        +\frac{34.544u\log u}{T}+\left(0.195\sqrt{q}+147.735\right)\log(T)+8.022\log{u}.
    \end{multline*}
\end{corollary}

By observing that $|\psi(u,\chi)-\psi(u,\chi*)|\leq \log q\log u$ for a non-primitive character and the character inducing that character, we obtain
\begin{corollary}
\label{cor:psiNotPrimitive}
    Assume $q\geq 7$, $\chi$ is a non-principal, non-primitive character modulo $q>4\cdot10^5$, $T \geq q$ and $2\leq u \leq 3T$. Then, we have 
    \begin{multline*}
        \left|\psi(u,\chi)-C(\chi^*)+\sum_{|\Im\rho|<T} \frac{u^\rho}{\rho}\right|<2\left(\left(\frac{5}{4}\right)^{1+\frac{1}{\log{2}}}+1\right)\frac{u\log{u}\log\log{u}}{T} \\
        +\frac{ 34.544u\log u}{T}+\left(0.195\sqrt{q}+148.430\right)\log(T)+ (\log{q}+8.022)\log{u}.
    \end{multline*}
\end{corollary}

\subsection{Proofs of Lemma \ref{lemma:principal1}, Proposition \ref{lemma:principal2} and Corollary \ref{cor:psiPrincipal}}

\begin{proof}[Proof of Lemma \ref{lemma:principal1}]
    By \cite[Lemma 12]{EHP2022} and since $q \geq 7$, in the case of a principal character, we have
    \begin{equation*}
        \left|\psi(u,\chi_0)-\psi(u)\right|=\sum_{\substack{n \leq u \\ (n,q)>1}} \Lambda(n) \leq \log{q}\log{u},
    \end{equation*}
   where $\psi(u)$ is the Chebyshev $\psi$-function. Note that this term can be omitted if $q=1$. Let us now define 
    \begin{equation*}
        \psi_0(u):=\frac{\psi(u^+)+\psi(u^-)}{2},
    \end{equation*}
    and similarly for $\psi(u,\chi)$ in the case of a non-principal character.
    Replacing $\psi(u)$ (or $\psi(u,\chi)$) with $\psi_0(u)$ (or $\psi_0(u,\chi)$) gives an error that is at most $\log{u}$. Let us now estimate the function $\psi_0(u)$ and  $\psi_0(u,\chi)$. 
    
    We first note that the first paragraph in the proof of \cite[Lemma 28]
    {EHP2022} still holds when we replace $\psi_0(u,\chi)$ with $\psi_0(u)$, characters $\chi$ with $1$ and functions $L(s,\chi)$ with $\zeta(s)$. Hence, we consider the sum
    \begin{equation*}
       \left( \sum_{n=1}^{\floor{\frac{4u}{5}}}+\sum_{n=\floor{\frac{4u}{5}}+1}^{\floor{u}-1}+\sum_{\substack{n=\floor{u}\\ n \neq u}}^{\floor{u}+1}+\sum_{n=\floor{u}+2}^{\floor{\frac{5u}{4}}}+\sum_{n=\floor{\frac{5u}{4}}+1}^\infty\right) \left(\Lambda(n)\left(\frac{u}{n}\right)^{c}\min\left\{T^{-1}\left|\log{\frac{u}{n}}\right|^{-1},1\right\}\right)+\frac{c\Lambda(u)}{T}.
    \end{equation*}
    The last term can be estimated trivially as $\leq (\log{u}+1)/T$.   
    Let us denote the first five sums by $S_1,\ldots, S_5$. We will estimate each of the sums.
    
    By the main theorem of \cite{D1987}
    \begin{equation*}
        S_1+S_5 \leq \frac{eu\log{u}}{T\log{(5/4)}}.
    \end{equation*}
    Moreover, because of \cite[p. 1346]{EHP2022} and due to the fact $u \geq 2$, we can estimate
    \begin{equation*}
        S_3 \leq \log{u}\left(\frac{u}{u-1}\right)^c+\log{(u+1)} \leq 2e\log{u}+\log{(u+1)}\leq \left(2e+\frac{\log{3}}{\log{2}}\right)\log{u}.
    \end{equation*}
       
    The sums $S_2$ and $S_4$ can be estimated similarly as the sums $S_2$ and $S_4$ \cite[Lemma 1]{CH2023}. Note first that $S_2 =0$ if $u<5$. Let us not point out the differences. We proceed similarly as in the beginning of the proof. Denote $A(u):= \floor{\frac{4u}{5}}+1$ and $B(u):=\floor{u}-1$. The expression corresponding to one on p. 107 in the proof of \cite[Lemma 1]{CH2023}, will be for $u \geq 6$
    
    \begin{multline}
    \label{eq:S2k1}
    \sum_{A(u) \leq p \leq B(u)} \frac{1}{u-p}  \leq \sum_{A(u) \leq p \leq B(u)} \frac{1}{\lfloor u\rfloor -p} \leq \sum_{n=1}^{\lfloor u\rfloor -A(u)}\frac{1}{n}(P(B(u),n+1)-P(B(u),n))\\ \leq\sum_{n=2}^{\lfloor u\rfloor -A(u)} \frac{1}{(n-1)n}P(B(u),n)+\frac{1}{\lfloor u\rfloor -A(u)}P(B(u), \lfloor u\rfloor -A(u)+1),
    \end{multline}
     where $P(x,y)$ denotes the number of primes in the interval $(x-y,x]$. (If $u/5-1<2$, we consider the first sum to be zero.) By \cite[Theorem 2]{MontgomeryVaughan} $P(x,y) \leq 2y/\log{y}$ if $1<y <x$.  Hence, applying the Euler–Maclaurin formula with the error term $R_2<0$, for $u\geq 21$ the right-hand side of \eqref{eq:S2k1} is
\begin{multline*}
=\sum_{2\leq n\leq \lfloor u\rfloor-A(u)}\frac{2}{(n-1)\log n}+2\left(1+\frac{1}{\lfloor u\rfloor -A(u)}\right)\frac{1}{\log(\lfloor u\rfloor-A(u)+1)}\\ 
< \frac{2}{\log 2}+\frac{2}{\log(\lfloor u\rfloor-A(u)+1)}+\sum_{2\leq n\leq \lfloor u\rfloor-A(u)+1}\frac{2}{n\log n}\leq \frac{2}{\log 2}+\frac{2}{\log(\frac{u}{5}-1)}+\sum_{2\leq n\leq \frac{u}{5}}\frac{2}{n\log n}\\ 
< \frac{2}{\log 2}+\frac{2}{\log(\frac{u}{5}-1)}+2\log\log\left(\frac{u}{5}\right)-2\log\log 2+\frac{1}{2\log 2}+\frac{5}{u\log(u/5)}+\frac{\log 2+1}{6\cdot 4\log^22}\\<2\log\log\left(\frac{u}{5}\right)+6.372.
\end{multline*}
Please observe that at the point of the first strict inequality, the sum has been artificially lengthened to make the summing interval easier to handle.

    Note that $2\log{\log{\left(\frac{u}{5}\right)}}+6.372$ gives an upper bound for the left-hand side of \eqref{eq:S2k1} also in the case $6 \leq u < 21$. 
    For $k\geq 2$ and $u\geq 6$, we estimate
    \begin{equation*}
       \sum_{ 4u/5\leq p^k\leq  u-1}\frac{1}{u-p^k} \leq  \frac{2 \tanh^{-1}\left(\sqrt{\frac{u - 1}{u}}\right) + \log(9 - 4 \sqrt{5})}{k\sqrt{u}}. 
    \end{equation*}
    When we sum over $k \in [2,\log{u}/\log{2}]$, by \cite{Young} the right-hand side is 
        \begin{equation*}
        \leq \frac{2 \tanh^{-1}\left(\sqrt{\frac{u - 1}{u}}\right) + \log(9 - 4 \sqrt{5})}{\sqrt{u}}\left(\log{\frac{\log{u}}{\log{2}}}+\frac{\log{2}}{2\log{u}}+C_0 -1\right) < 0.4824, \quad u \geq 6.
    \end{equation*} 
    We note that to find maximum it is sufficient that $$\frac{\log\left(\frac{\sqrt{x-1}+\sqrt{x}}{\sqrt{x}-\sqrt{x-1}}\right)+\log(9-4\sqrt{5})}{\sqrt{x}}\left(\log\frac{\log{x}}{\log{2}}+C_0-1)\right)$$ is decreasing after the maximum point since all of the other parts are also decreasing then. Since both $\frac{\log\frac{1+\sqrt{x-1}}{\sqrt{x}-\sqrt{x-1}}+\log(9-4\sqrt{5})}{x^{0.35}}$ and $\frac{\log\frac{\log{x}}{\log{2}}+C_0-1}{x^{0.15}}$ are decreasing for $x \geq 158$, it is sufficient to find the maximum in $6 \leq x \leq 158$. The maximum of the whole formula is $0.4824$ around the point $x\approx 75.7$. Thus, for $u \geq 6$, we have
    \begin{equation*}
        S_2 \leq \left(\frac{5}{4}\right)^{1+\frac{1}{\log{2}}}\frac{u}{T}\log{\left(u\right)}\left(2\log\log{\left(\frac{u}{5}\right)}+6.372+0.4824\right) \leq \left(\frac{5}{4}\right)^{1+\frac{1}{\log{2}}}\frac{u}{T}\log{\left(u\right)}\left(2\log\log{\left(u\right)}+ 6.8544\right).
    \end{equation*}

    Let us next consider $S_4$. Now we set $A(u):= \floor{u}+2$ and $B(u):=\floor{\frac{5u}{4}}$. Similarly as in the case of $S_2$, following the lines of the proof of \cite[Lemma 1]{CH2023}, we can deduce that for $u \geq 12$
We have 
\begin{multline*}
\sum_{A(u) \leq p \leq B(u)} \frac{1}{p-u} \leq \sum_{A(u) \leq p \leq B(u)}\frac{1}{p-\lfloor u\rfloor -1}=\sum_{1\leq n\leq  B(u)-A(u)+1}\frac{1}{n}(P(\lfloor u\rfloor+n+1,n+1)-P(\lfloor u\rfloor+n,n))\\ \leq \sum_{2\leq n\leq B(u)-A(u)+1}\frac{P(\lfloor u\rfloor +n,n)}{n(n-1)}+\frac{P(\lfloor u\rfloor +B(u)-A(u)+2,B(u)-A(u)+2)}{B(u)-A(u)+1}\\ \leq \frac{2}{\log 2}+\frac{2}{\log(B(u)-A(u)+2)}+\sum_{n=2}^{B(u)-A(u)+1}\frac{2}{n\log n}\leq \frac{2}{\log 2}+\frac{2}{\log(u/4-1)}+\sum_{n=2}^{u/4}\frac{2}{n\log n}
\\ \leq \frac{2}{\log 2}+\frac{2}{\log(u/4-1)}+2\log\log (u/4)-2\log\log 2+\frac{2}{4\log 2}+\frac{4}{u\log(u/4)}+2\cdot\frac{\log 2+1}{12\cdot 4\log^22}\\ 
\leq 2\log\log (u/4)+7.6755
\end{multline*} for $u\geq 12$.
    The right-hand side of the estimate works in the case $5\leq u <11$. Following the ideas of \cite[p. 109]{CH2023}, for $u \geq 5$ we get an estimate
    \begin{multline*}
        S_4 \leq \frac{\log{u}}{T}\left(2u\log\log{\left(\frac{u}{4}\right)}+7.6755u+2\left(\frac{\pi^2}{6}-1\right)\sqrt{u}\log{\left(2\sqrt{u+1}\right)}\right) \\
        \leq \frac{\log{u}}{T}\left(2u\log\log{\left(u\right)}+7.6755u+2\left(\frac{\pi^2}{6}-1\right)\sqrt{u}\log{\left(2\sqrt{u+1}\right)}\right).
    \end{multline*}
 The last estimate holds in the case $2 \leq u <5$, too. Putting everything together, we have
 \begin{multline*}
 S_2+S_4\leq \left(\frac{5}{4}\right)^{1+\frac{1}{\log{2}}}\frac{u}{T}\log{\left(u\right)}\left(2\log\log{\left(u\right)}+ 6.8544\right)\\ 
 +\frac{\log{u}}{T}\left(2u\log\log{\left(u\right)}+7.6755u+2\left(\frac{\pi^2}{6}-1\right)\sqrt{u}\log{\left(2\sqrt{u+1}\right)}\right)\\ 
 \leq \left(\left(\frac{5}{4}\right)^{1+\frac{1}{\log{2}}}+1\right)\frac{2u}{T}\log{\left(u\right)}\log\log{\left(u\right)}+\frac{19.498u}{T}\log{\left(u\right)}\\
 +\frac{0.645\sqrt{u}\log^2u}{T}+\frac{1.156\sqrt{u}\log u}{T},
 \end{multline*}
 where the identity $\log{(2\sqrt{u+1})}=\frac{1}{2}\log{u}+\log{(2\sqrt{1+1/u})}$ has been used.
 \end{proof}
\begin{proof}[Proof of Proposition \ref{lemma:principal2}]
In Lemma \ref{lemma:principal1}, we have bounded the term
    \begin{equation*}
        \left|\psi_0(u)-\frac{1}{2\pi i}\int_{c-iT}^{c+iT} \left(-\frac{\zeta'(s)}{\zeta(s)}\right)\frac{u^s}{s} \, ds\right|.
    \end{equation*}
  
   The integral in the left-hand side can be estimated by Dudek \cite{D2016}. By \cite{D2016} page 184 and Section 2.1, and \cite[Lemma 18]{EHP2022}, we have
    \begin{multline*}
        \left|\frac{1}{2\pi i}\int_{c-iT}^{c+iT} \left(-\frac{\zeta'(s)}{\zeta(s)}\right)\frac{u^s}{s} \, ds-u+\sum_{|\gamma_\zeta|<T}\frac{u^{\rho_\zeta}}{\rho_\zeta}+\frac{\zeta'(0)}{\zeta(0)}\right|<\frac{4u\log{T}}{T-1}+\frac{2eu\log{u}}{\pi(T-1)}
        \\+\frac{e u}{\pi(T-1)  \log{u}}\left(\log^2{(T+1)}+20\log{(T+1)}\right)+\frac{18+2\log{\sqrt{U^2+T^2}}}{2\pi u^U}+\frac{18+2\log{\sqrt{U^2+(T+1)^2}}}{2\pi u^U(T-1)}\\
        +\frac{9+\frac{1}{2}\log{\left(U^2+(T+1)^2\right)}}{\pi u (T-1)} +\frac{1}{6},
    \end{multline*}
where we have bounded Dudek's term $|I_5|\leq \frac{9+\log{\sqrt{U^2+(T+1)^2}}}{2\pi u^U(T-1)}$, and as also noted in \cite{CH2023}, we have 
$$
\left|I_7\right| \leq eu\frac{\log^2{(T+1)}+20\log{(T+1)}}{2\pi(T-1)\log{u}}.
$$ 
and where the sum runs over non-trivial zeros of the Riemann zeta function and we choose $U$ to be an even integer that is closest to number $u$.  Hence $2\leq U \leq u+1$ and the result follows when we remember that $\zeta'(0)/\zeta(0)=\log(2\pi)$.
\end{proof}

\begin{proof}[Proof of Corollary \ref{cor:psiPrincipal}]
Let us now look at the terms contributing towards the $u\log u$ term: 
\begin{equation}
\label{eq:estuloguT}
\frac{eu\log{u}}{T\log{(5/4)}}+\frac{19.498u\log{u}}{T} +\frac{0.645\sqrt{u}\log^2{u}}{T}+\frac{1.156\sqrt{u}\log{u}}{T}+\frac{2eu\log{u}}{\pi(T-1)}<34.544\frac{u\log u}{T}.
\end{equation}
Here we used the fact that the maximum of $0.645\log{u}/\sqrt{u}+1.156/\sqrt{u}$ for $u\geq 2$ is obtained at $u=2$.
Let us now move towards the term $\log(T)$.
\begin{multline*}
 \frac{9+\frac{1}{2}\log{\left((u+1)^2+(T+1)^2\right)}}{\pi u (T-1)}+\frac{1}{6} \\
=\frac{1}{T-1}\left(\frac{1}{u\pi}\log T+\frac{1}{u\pi}\log \left(1+\frac{1}{T}\right)+\frac{\log(u+1)}{\pi u}+\frac{18+\log\left(\frac{1}{(T+1)^2}+\frac{1}{(u+1)^2}\right)}{2\pi u}\right)+\frac{1}{6} \\
<0.01293\log{T}. 
\end{multline*}

Let us now look at the following terms:
\begin{multline*}
\frac{18+2\log{\sqrt{(u+1)^2+T^2}}}{2\pi u^2}+\frac{1}{T-1}\cdot \frac{9+\frac{1}{2}\log{\left((u+1)^2+(T+1)^2\right)}}{\pi u^2}+\frac{1}{T}+\frac{4u\log T}{T-1}\\ 
<\log (T)\cdot \frac{1}{T-1}\left(\frac{1}{u^2\pi}+\frac{1}{u^2\pi\log T}\log \left(1+\frac{1}{T}\right)+\frac{\log(u+1)}{\pi u^2\log T}+\frac{9+\frac{1}{2}\log\left(\frac{1}{(T+1)^2}+\frac{1}{(u+1)^2}\right)}{\pi u^2\log T} +\frac{1}{\log T}\right)\\
+\log (T)\cdot\left(\frac{1}{u^2\pi}+\frac{\log(u+1)}{\pi u^2\log T}+\frac{9+\frac{1}{2}\log\left(\frac{1}{T^2}+\frac{1}{(u+1)^2}\right)}{\pi u^2\log T}+ \frac{12T}{T-1}\right)
<12.13514\log (T).
\end{multline*}
Also, we have
\begin{multline*}
\frac{e u\left(\log^2{(T+1)}+ 20\log{(T+1)}\right)}{\pi(T-1) \log{u}}=\log(T)\frac{eu\left(\log{(T+1)}+20\right)}{\pi(T-1)\log{u}}+\log\left(1+\frac{1}{T}\right)\frac{eu\left(\log{(T+1)}+20\right)}{\pi(T-1) \log{u}}\\  \leq\log(T)\left(\frac{3Te\left(\log{(T+1)}+20\right)}{\pi(T-1) \log{(3T)}}+\log\left(1+\frac{1}{T}\right)\frac{3Te\left(\log{(T+1)}+20\right)}{\pi(T-1)(\log{(3T)})\log(T)}\right)<6.10088\log(T).
\end{multline*}
We may now put these together to obtain
\[
0.01293\log(T)+12.13514\log (T)+ 6.10088\log(T)<18.249\log(T).
\]
Finally, we have
\begin{equation}
\label{eq:estlogu}
\log{u}+\left(2e+\frac{\log{3}}{\log{2}}\right)\log{u}+\frac{\log{u}}{T}<8.022\log{u}.
\end{equation}
\end{proof}

\subsection{Proofs of Proposition \ref{lemma:mainLemmaPsiuchi} and Corollary \ref{cor:psiPrimitive}}

This is similar to the proof of Theorem 29 in \cite{EHP2022} following the lines of Chapter 19 in \cite{D2000}.

Next we wish to replicate the proof of Theorem 29 \cite{EHP2022} to estimate $I(u,T,\chi)$. There are some differences because in the current setting, we do not assume the GRH. We pick the route of integration in such a way that it avoids zeros of the $L$-function. Lemma 22 in \cite{EHP2022} has to be replaced with the following lemma:

\begin{lemma}
\label{lemmaDifference}
Assume that $\chi$ is a primitive non-principal character modulo $q \geq 3$ and $T \geq 4\cdot10^5$. Then there are numbers $T_1\in (T-1,T+1]$ and $T_2 \in [-T-1,-T+1)$  such that for $T_j$,  $j \in \{1,2\}$, it holds that
\begin{equation*}
\left|\Im\rho-T_j\right|>\frac{1}{1.092\log{(qT)}+4\log{\log{(qT)}}-2.005}
\end{equation*}
for all nontrivial zeros of the function $L(s,\chi)$.
\end{lemma}
\begin{proof}
For these values of $T$, we use the bound 
\begin{equation*}
\frac{T}{\pi}\log{\left(1+\frac{2}{T-1}\right)}\leq 0.6367,
\end{equation*} 
(which is true, because the second derivative is positive, and the limit of the first derivative at infinity is $0$, meaning that the function is decreasing)
and
\begin{equation*}
\log{(T+1)(T+3)}\leq \log{(1.00002T^2)},
\end{equation*}
 The original estimate is thus replaced with the following one: 
\begin{equation}
\label{estimateNZeros}
\begin{aligned}
& \frac{T+1}{\pi}\log{\frac{q(T+1)}{2\pi e}}-\frac{\chi(-1)}{4}+0.22737\log{\frac{q(T+3)}{2\pi}} +2\log{\left(1+\log{\frac{q(T+3)}{2\pi}}\right)}-0.5\\ & \quad -\frac{T-1}{\pi}\log{\frac{q(T-1)}{2\pi e}}+\frac{\chi(-1)}{4} +0.22737\log{\frac{q(T+1)}{2\pi}}+2\log{\left(1+\log{\frac{q(T+1)}{2\pi}}\right)}-0.5 \\
&<2\left(0.22737+\frac{1}{\pi}\right)\log{(qT)}+0.6367+0.22737\log{(1.00002)}+
\left(-\frac{2}{\pi}-0.45474\right)\log{(2\pi)}\\ 
& \quad+4\log{\left(\log{\frac{(4\cdot 10^5+3)qeT}{2\cdot4\cdot 10^5\pi}}\right)}-\frac{2}{\pi}-1 \\
&<1.092\log{(qT)}+4\log{\log{(qT)}}-3.005,
\end{aligned}
\end{equation}
where we applied \cite[Corollary 1.2]{BMOR2021}.
\end{proof}

We use this lemma to pick suitable $T_1$ and $T_2$. The following lemma is a version of Lemma 23 \cite{EHP2022}:
\begin{lemma}\label{smalldifference}
Assume $T\geq 4\cdot 10^5$, $q\geq 3$ and $\chi$ is a primitive non-principal character modulo $q$. Moreover, let $\Im s$ be either $T_1$ or $T_2$ in Lemma \ref{lemmaDifference}. We have
\begin{multline*}
\sum_{\substack{\rho \\ |\Im\rho- \Im s|< 1}}\frac{3}{|\sigma+ i\Im s-\rho||2+ i\Im s-\rho|}\\<3.578 \log^2(q T) + 48 \log^2(\log(q T)) + 26.208 \log(\log(q T)) \log(q T) - 16.412 \log(q T) -60.12 \log(\log(q T)) + 18.076.
\end{multline*}
where the sum runs over the nontrivial zeros of the function $L(s,\chi)$ with $|\Im\rho- \Im s|< 1$.
\end{lemma}

\begin{proof} 
By \eqref{estimateNZeros}, there are at most $1.092\log{(qT)}+4\log{\log{(qT)}}-3.005$ nontrivial zeros with $|\Im \rho-\Im s|<1$ and by Lemma \ref{lemmaDifference}, we have
\[
\left|\Im\rho-T_j\right|>\frac{1}{1.092\log{(qT)}+4\log{\log{(qT)}}-2.005}.
\]
Hence,
\begin{align*}
&\sum_{\substack{\rho \\ |\Im\rho- \Im s|< 1}}\frac{3}{|\sigma+ i\Im s-\rho||2+ i\Im s-\rho|}\\
&\quad<3\left(1.092\log{(qT)}+4\log{\log{(qT)}}-3.005\right)\left(1.092\log{(qT)}+4\log{\log{(qT)}}-2.005\right)\\ 
&\quad<3.578 \log^2(q T) + 48 \log^2(\log(q T)) + 26.208 \log(\log(q T)) \log(q T) - 16.412 \log(q T) -60.12 \log(\log(q T)) + 18.076.
\end{align*}
\end{proof}

The following lemma is a version of Lemma 24 \cite{EHP2022}:
\begin{lemma}\label{largedifference}
Assume $T\geq 2$, $q\geq 3$, $\chi$ is a primitive non-principal character modulo $q$ and $\Im s$ be either $T_1$ or $T_2$ in Lemma \ref{lemmaDifference}. We have
\[
\sum_{\substack{\rho \\ |\Im\rho- \Im s|\geq 1}}\frac{3}{|\Im s-\Im\rho|^2}<\frac{15}{2}\cdot 0.570+\frac{15}{8}\log\left(\frac{T^2}{4}+\frac{T}{2}+\frac{5}{2}\right)+\frac{15}{4}\log \frac{q}{\pi}-\frac{15}{2}\Re B(\chi),
\]
where $B(\chi)$ is as in \eqref{def:Bchi} and the sums run over the nontrivial zeros of the function $L(s,\chi)$. 
\end{lemma}
\begin{proof}
Let us start by bounding $\frac{1}{|\Im s-\Im\rho|^2}$ with $\Re \frac{c}{2+i\Im s-\rho}$ for some suitable $c$. We have
\[
\Re \frac{1}{2+i\Im s-\rho}=\frac{2-\Re \overline{\rho}}{(2-\Re \overline \rho)^2+(\Im s-\Im \rho)^2}.
\]
For the sake of clarity, write $t=2-\Re \overline{\rho}$ and $y=|\Im s-\Im \rho|$. We need to find a constant $c_1=\frac{1}{c}$ such that
\[
\frac{t}{t^2+y^2}\geq \frac{c_1}{y^2}.
\]
Denote $f(t)=\frac{t}{t^2+y^2}$ and differentiate:
\[
f'(t)=\frac{1}{t^2+y^2}-\frac{2t^2}{(t^2+y^2)^2}=\frac{y^2-t^2}{(t^2+y^2)^2}.
\]
Since $1\leq t\leq 2$, when $y\geq 2$, the derivative is always non-negative. For a fixed value of $y$, the function obtains its smallest value at $t=1$. We must have $\frac{1}{1+y^2}\geq \frac{c_1}{y^2}$, so $c_1\leq \frac{y^2}{1+y^2}$. Since this must hold for all $y \geq 2$, we choose $c_1=\frac{4}{5}$.

Let us now look at the case with $1\leq y<2$. Now the derivative is non-negative for $1\leq t\leq y$ and negative for $t>y$. The smallest value is thus obtained either at $t=1$ or at $t=2$. We get the conditions
\[
\frac{1}{1+y^2}\geq \frac{c_1}{y^2}\quad \textrm{and}\quad \frac{2}{4+y^2}\geq \frac{c_1}{y^2}.
\]
Since $c_1$ has to satisfy both of these conditions, we have $c_1=\frac{2}{5}$. Now
\[
\frac{3}{|\Im s-\Im\rho|^2}\leq \frac{3\cdot 5}{2}\Re \frac{1}{2+i\Im s-\rho}=\frac{15}{2}\Re \frac{1}{2+i\Im s-\rho}.
\]
Again, we have $\Re \frac{1}{\rho}\geq 0$ for all $\rho$. Further, we also have $\Re \frac{1}{2+i\Im s-\rho}>0$, so we may estimate
\[
\sum_{\substack{\rho \\ |\Im\rho- \Im s|\geq 1}}\frac{3}{|\Im s-\Im\rho|^2}\leq \sum_{\substack{\rho \\ |\Im\rho- \Im s|\geq 1}}\frac{15}{2}\Re \frac{1}{2+i\Im s-\rho}\leq \sum_{\rho}\frac{15}{2}\Re \left(\frac{1}{2+i\Im s-\rho}+\frac{1}{\rho}\right)
\]
Similarly as in the proof of Lemma 24 \cite{EHP2022}, we use the functional equation to obtain
\begin{multline*}
\sum_{\rho}\frac{15}{2}\Re \left(\frac{1}{2+i\Im s-\rho}+\frac{1}{\rho}\right)=\frac{15}{2}\Re \left(\frac{L'(2+i\Im s,\chi)}{L(2+i\Im s,\chi)}\right)+\frac{15}{2}\cdot\frac{1}{2}\log \frac{q}{\pi} \\
+\frac{15}{4}\Re \left(\frac{\Gamma'(1+i\Im s/2+\mathfrak{a}/2)}{\Gamma(1+i\Im s/2+\mathfrak{a}/2)}\right)-\frac{15}{2}\Re B(\chi),
\end{multline*}
where the value of $\mathfrak{a}$ depends on the character. The first three terms can be bounded similarly as in \cite{EHP2022}. We have
\[
\frac{15}{2}\Re \left(\frac{L'(2+i\Im s,\chi)}{L(2+i\Im s,\chi)}\right)<\frac{15}{2}\cdot 0.570
\]
and
\[
\frac{15}{4}\Re \left(\frac{\Gamma'(1+i\Im s/2+\mathfrak{a}/2)}{\Gamma(1+i\Im s/2+\mathfrak{a}/2)}\right)\leq \frac{15}{8}\log\left(\frac{T^2}{4}+\frac{T}{2}+\frac{5}{2}\right).
\]
Putting everything together finishes the proof.
\end{proof}

Let us now prove the original lemma \ref{lemma:mainLemmaPsiuchi}. 

\begin{proof}[Proof of Proposition \ref{lemma:mainLemmaPsiuchi}]
First, we remember that because of Lemma \ref{lemma:principal1} we only have to estimate the function $J$ given in \eqref{eq:defJ}. We have
\begin{equation}
\label{LxIntegral}
\begin{aligned}
&I(u,T+1,\chi) \\
&\quad= \frac{1}{2\pi i}\int_{c-i(T+1)}^{c+i(T+1)} \left(-\frac{L'(s,\chi)}{L(s,\chi)}\right)\frac{u^s}{s} ds\\
&\quad=\frac{1}{2\pi i}\int_{c-i(T+1)}^{c+iT_2} \left(-\frac{L'(s,\chi)}{L(s,\chi)}\right)\frac{u^s}{s} ds+\frac{1}{2\pi i}\int_{c+iT_1}^{c+i(T+1)} \left(-\frac{L'(s,\chi)}{L(s,\chi)}\right)\frac{u^s}{s} ds  \\
& \quad\quad+\frac{1}{2\pi i}\int_\mathcal{R} \left(-\frac{L'(s,\chi)}{L(s,\chi)}\right)\frac{u^s}{s} ds-\frac{1}{2\pi i}\int_{\mathcal{R}_1} \left(-\frac{L'(s,\chi)}{L(s,\chi)}\right)\frac{u^s}{s} ds \\
& \quad\quad-\frac{1}{2\pi i}\int_{c+iT_1}^{-U+iT_1} \left(-\frac{L'(s,\chi)}{L(s,\chi)}\right)\frac{u^s}{s} ds-\frac{1}{2\pi i}\int_{-U+iT_2}^{c+iT_2} \left(-\frac{L'(s,\chi)}{L(s,\chi)}\right)\frac{u^s}{s} ds.
\end{aligned}
\end{equation}
Here $c=1+1/\log{u}$, $U>1$, $T_1$ and $T_2$ are as in Lemma \ref{lemmaDifference}, if there does not exist a zero $\rho=-U$, then $\mathcal{R}$ is a rectangle with vertices
\begin{equation*}
    c+iT_2, \quad c+iT_1, \quad -U+iT_1 \quad\text{and}\quad -U+iT_2
\end{equation*}
and otherwise we avoid the point $U$ with a small radius. The left horizontal part of $\mathcal{R}$ is denoted by $\mathcal{R}_1$.

We may start the same way as in the original paper. By Lemma 16 \cite{EHP2022}, we have
\[
\left|\frac{1}{2\pi i}\int_{c-i(T+1)}^{c+iT_2} \left(-\frac{L'(s,\chi)}{L(s,\chi)}\right)\frac{u^s}{s} ds+\frac{1}{2\pi i}\int_{c+iT_1}^{c+i(T+1)} \left(-\frac{L'(s,\chi)}{L(s,\chi)}\right)\frac{u^s}{s} ds\right|<\left(\log{u}+C_0+\frac{0.478}{\log{u}}\right)\frac{2ue}{\pi(T-1)}.
\]
We also have
\begin{multline}
\label{eq:RR1}
\frac{1}{2\pi i}\int_\mathcal{R} \left(-\frac{L'(s,\chi)}{L(s,\chi)}\right)\frac{u^s}{s} ds-\frac{1}{2\pi i}\int_{\mathcal{R}_1} \left(-\frac{L'(s,\chi)}{L(s,\chi)}\right)\frac{u^s}{s} ds\\=-\sum_{\substack{\rho \\ T_2<\Im\rho<T_1}} \frac{u^\rho}{\rho}-\mathfrak{a}\frac{L'(0,\chi)}{L(0,\chi)}-(1-\mathfrak{a})\left(\log{u}+b(\chi)\right)+\sum_{m=1}^\infty\frac{u^{\mathfrak{a}-2m}}{2m-\mathfrak{a}},
\end{multline}
as $U \to \infty$, where $\mathfrak{a}=1$ if $\chi(-1)=-1$, $\mathfrak{a}=0$ if $\chi(-1)=1$ and $b(\chi)$ comes from the Laurent series expansion of $L'(s,\chi)/L(s,\chi)-1/s=b(\chi)+...$ at $s=0$ if $\chi(-1)=1$. Note that
$$
-\mathfrak{a}\frac{L'(0,\chi)}{L(0,\chi)}-(1-\mathfrak{a})\left(\log{u}+b(\chi)\right)=C(\chi).
$$ By Lemma 18 \cite{EHP2022}, we have
\[
\sum_{m=1}^\infty\frac{u^{\mathfrak{a}-2m}}{2m-\mathfrak{a}}\leq \begin{cases} 1 & \textrm{if }\chi(-1)=-1 \\ \frac{1}{6} & \textrm{if }\chi(-1)=1.\end{cases}
\]

Let us now change the range of summation of the first term in the right-hand side of \eqref{eq:RR1}
into $-T<\Im \rho <T$. By the proof of Lemma \ref{lemmaDifference}, there are at most $1.092\log{(qT)}+4\log{\log{(qT)}}-3.005$ zeros with $|\Im(\rho)| \in (T-1, T+1]$. Since $|T-T_1|<1$ and $|T+T_2|<1$, the error from changing the range of summation, is at most
\[
\frac{u(1.092\log{(qT)}+4\log{\log{(qT)}}-3.005)}{T-1}.
\]

Now, the last integrals are the most complicated ones. Here we cannot do exactly the same bounds as earlier because bounding these in \cite{EHP2022}, GRH was assumed. We divide the integrals so that we treat separately the interval $\Re s \geq -1$ and $\Re s<-1$. 

Let us start with $\Re s<-1$. Just like in \cite{EHP2022}, we get the bound
\[
\left|-\frac{1}{2\pi i}\int_{-1+iT_1}^{-U+iT_1}\left(\frac{L'(s,\chi)}{L(s,\chi)}\right)\frac{u^s}{s}ds-\frac{1}{2\pi i}\int_{-U+iT_2}^{-1+iT_2}\left(\frac{L'(s,\chi)}{L(s,\chi)}\right)\frac{u^s}{s}ds\right|<\frac{3\log{\left(\frac{T+1}{2}\right)}+\log{(q\pi)}+3.570+C_0+\frac{8}{T-1}}{\pi (T-1)u\log{u}}.
\]

Let us now turn to the case $\Re s \geq -1$. Similarly, as in \cite{EHP2022}, we get the bound
\[
\left|-\frac{1}{2\pi i}\int_{c+iT_1}^{-1+iT_1}\left(\frac{L'(s,\chi)}{L(s,\chi)}\right)\frac{u^s}{s}ds-\frac{1}{2\pi i}\int_{-1+iT_2}^{c+iT_2}\left(\frac{L'(s,\chi)}{L(s,\chi)}\right)\frac{u^s}{s}ds\right|<\frac{ue}{\pi (T-1)\log{u}}R(T,\chi),
\]
but the function $R(T,\chi)$, which is an estimate for the logarithmic derivative of the $L(s,\chi)$, is different.

We can derive the expression for $R(T,\chi)$ using absolute values of the terms in the integral.

We bound the logarithmic derivative of the $L$-function. We follow the structure of \cite{EHP2022} Lemma 25. We have
\[
\frac{L'(s, \chi)}{L(s,\chi)} =\frac{L'(2+ i\Im s, \chi)}{L(2+i\Im s, \chi)}-\frac{1}{2}\frac{\Gamma'\left(\frac{1}{2}s+\frac{1}{2}\mathfrak{a}\right)}{\Gamma\left(\frac{1}{2}s+\frac{1}{2}\mathfrak{a}\right)}+\frac{1}{2}\frac{\Gamma'\left(1+\frac{i\Im s}{2}+\frac{1}{2}\mathfrak{a}\right)}{\Gamma\left(1+\frac{i\Im s}{2}+\frac{1}{2}\mathfrak{a}\right)}+\sum_\rho\left(\frac{1}{s-\rho}-\frac{1}{2+ i\Im s-\rho}\right).
\]
Now using Lemma 16 in \cite{EHP2022}, we have $\frac{L'(2+ i\Im s, \chi)}{L(2+i\Im s, \chi)}<0.570$. Furthermore,
\[
-\frac{1}{2}\frac{\Gamma'\left(\frac{1}{2}s+\frac{1}{2}\mathfrak{a}\right)}{\Gamma\left(\frac{1}{2}s+\frac{1}{2}\mathfrak{a}\right)}+\frac{1}{2}\frac{\Gamma'\left(1+\frac{i\Im s}{2}+\frac{1}{2}\mathfrak{a}\right)}{\Gamma\left(1+\frac{i\Im s}{2}+\frac{1}{2}\mathfrak{a}\right)}<3\left(\frac{1}{(T-1)^2}+\frac{\pi}{4(T-1)}\right).
\]
We only have one term left here. We have
\[
\left|\sum_\rho\left(\frac{1}{s-\rho}-\frac{1}{2+ i\Im s-\rho}\right)\right|=\left|\sum_\rho\frac{2-\Re s}{(s-\rho)(2+i\Im s-\rho)}\right|\leq \left|\sum_\rho\frac{3}{(s-\rho)(2+i\Im s-\rho)}\right|.
\]
We need to divide this sum into to parts:
\[
\sum_\rho\frac{3}{(s-\rho)(2+i\Im s-\rho)}=\sum_{\substack{\rho \\ |\Im\rho- \Im s|< 1}}\frac{3}{(s-\rho)(2+i\Im s-\rho)}+\sum_{\substack{\rho \\ |\Im\rho- \Im s|\geq 1}}\frac{3}{(s-\rho)(2+i\Im s-\rho)}
\]

Let us now consider the case $|\Im\rho- \Im s|< 1$. By Lemma \ref{smalldifference}, we have
\begin{multline*}
\sum_{\substack{\rho \\ |\Im\rho- \Im s|< 1}}\frac{3}{|\sigma+ i\Im s-\rho||2+ i\Im s-\rho|}\\<3.578 \log^2(q T) + 48 \log^2(\log(q T)) + 26.208 \log(\log(q T)) \log(q T) 
 - 16.412 \log(q T) -60.12 \log(\log(q T)) + 18.076
\end{multline*}

By Lemma \ref{largedifference}, we have
\[
\sum_{\substack{\rho \\ |\Im\rho- \Im s|\geq 1}}\frac{3}{|\Im s-\Im\rho|^2}<\frac{15}{2}\cdot 0.570+\frac{15}{8}\log\left(\frac{T^2}{4}+\frac{T}{2}+\frac{5}{2}\right)+\frac{15}{4}\log \frac{q}{\pi}-\frac{15}{2}\Re B(\chi).
\]
Summing these together yields the expression for $R(T,\chi)$, which finishes the proof.
\end{proof}

We still need to bound the expression
\[
-\Re B(\chi)=\Re \sum_{\rho} \frac{1}{\rho}.
\]
This is described in the following lemma.

\begin{lemma}\label{lem:Bchi}
Assume that $q \geq 3$ and $\chi$ is a character with conductor $q$. Then
    \[
\left|\Re B(\chi)\right| \leq 
\begin{cases}
    1.275\log{q}+6.961 &\text{if } 3 \leq q \leq 4\cdot10^5 \\
    2.288\log^2{q}+30.264\log{q}+5.809 &\text{if } 4\cdot10^5< q \text{ and } \chi \text{ is non-real} \\
    \frac{\sqrt{q}\log^2{q}}{800}+2.288\log^2{q}+20.618\log{q}+5.809 &\text{if } 4\cdot10^5< q \text{ and } \chi \text{ is real and odd} \\
    \frac{\sqrt{q}\log^2{q}}{100}+2.288\log^2{q}+20.618\log{q}+5.809 &\text{if } 4\cdot10^5< q \text{ and } \chi \text{ is real and even}. 
\end{cases}
\]
\end{lemma}

\begin{proof}
If $q\leq 4\cdot 10^5$, then there are no exceptional zeros on the real line. Furthermore, the GRH holds up to height $293.75$. In this case, we wish to work similarly as in \cite[Lemma 17]{EHP2022}. For zeros near real axis, we bound the term $\left|\Re \frac{1}{\rho}\right|\leq 2$ to obtain $2N(5/7,\chi)$. For the sake of simplicity, the constant $5/7$ is chosen similarly as in the proof of \cite[Lemma 17]{EHP2022}. For zeros above this, we simply use the bound $\left|\Re \frac{1}{\rho}\right|\leq \frac{1}{2t^2}$ up to $293.75$ and $\left|\Re \frac{1}{\rho}\right|\leq \frac{1}{t^2}$ for values above it. We use partial summation to obtain
\begin{multline*}
\sum_{\substack{\rho \\ |\Im(\rho)| \geq 5/7}} \Re\frac{1}{\rho}\leq \int_{5/7}^{\infty}\frac{2}{t^3}(N(t,\chi)-N(5/7,\chi))dt-\int_{5/7}^{293.75}\frac{1}{t^3}(N(t,\chi)-N(5/7,\chi))dt\\
\leq -0.98N (5/7, \chi)+\int_{5/7}^{293.75} \frac{1}{t^3}\left(\frac{t}{\pi}\log{\frac{qt}{2\pi e}}+0.298\log{(qt)}+4.358\right) \, dt \\
+\int_{293.75}^\infty \frac{2}{t^3}\left(\frac{t}{\pi}\log{\frac{qt}{2\pi e}}+0.298\log{(qt)}+4.358\right) \, dt\\
\leq 0.73876 \log(q) + 3.35384 -0.98N (5/7, \chi).
\end{multline*}
The total bound is therefore
\begin{equation*}
   0.73876 \log{q}+3.35384+1.02 N (5/7, \chi) \leq 1.275\log{q}+6.961
\end{equation*}

Assume now that $q>4\cdot 10^5$. Then there may be exceptional zeros with absolute values of the real parts less than $293.75$.

We bound the sum for $T>5/7$ similarly as in \cite[p. 1333]{EHP2022}, and get
\begin{equation}
\label{eq:qLargetLarge}
\frac{-3.920N (5/7, \chi) + 2.751 \log q + 23.308}{2}<-1.960N (5/7, \chi)+1.351\log q+11.654
\end{equation}
for the sum on this interval. Now we need to bound the contribution coming from zeros near but off the real axis. 
Since again at least half of the non-trivial zeros lie on the line $\Re(s)=1/2$, we can bound
\begin{equation*}
    \sum_{\substack{\rho \\ |\rho| \leq 5/7}} \Re\left(\frac{1}{\rho}\right) \leq
    \begin{cases}
        \left(\frac{9.646\log{q}}{2}+1\right)N (5/7, \chi) &\text{if } \chi \text{ is non-real} \\
        \frac{\sqrt{q}\log^2{q}}{800}+9.646\log{q}\left(\frac{N (5/7, \chi)}{2}-1\right)+N (5/7, \chi) &\text{if } \chi \text{ is real and odd} \\
        \frac{\sqrt{q}\log^2{q}}{100}+9.646\log{q}\left(\frac{N (5/7, \chi)}{2}-1\right)+N (5/7, \chi) &\text{if } \chi \text{ is real and even}. 
    \end{cases}
\end{equation*}
Combining the previous estimate with the estimate \eqref{eq:qLargetLarge} and using the bound $N (5/7, \chi)\leq \frac{5}{7\pi}\log{\frac{5q}{14\pi e}}+0.247\log{\frac{5q}{7}}+6.894$, the claim follows.
\end{proof}

Lastly, we prove Corollary \ref{cor:psiPrimitive} that gives a simplified error term for $\psi(u,\chi)$.

\begin{proof}[Proof of Corollary \ref{cor:psiPrimitive}]
    First we note that
    \begin{equation*}
        \left(C_0+\frac{0.478}{\log{u}}\right)\cdot\frac{2e}{\pi}-3.005<0
    \end{equation*}
for all $u \geq 2$. Hence, we can estimate the $u\log{u}/T$ terms as in \eqref{eq:estuloguT} and the $\log{u}$ terms as in \eqref{eq:estlogu}. Thus, we have only $\log{T}$ terms left.

Since $T\geq \max\{q, 4\cdot10^5\}$ and $u \leq 3T$, we have
\begin{multline*}
    \frac{u(1.092\log{(qT)}+4\log{\log{(qT)}})}{T-1}+\frac{3\log{\left(\frac{T+1}{2}\right)}+\log{(q\pi)}+3.570+C_0+\frac{8}{T-1}}{\pi (T-1)u\log{u}}+\frac{1}{6}+\frac{1}{T} \\
    <\frac{3T}{T-1}\log{T}\left(2\cdot1.092+\frac{4\log{(2\log{T})}}{\log{T}}\right)+\frac{1}{6}+\frac{1}{T}\\
    +\frac{\log{T}}{\pi (T-1)u\log{u}}\left(4+\frac{3\log{(1+\frac{1}{T})}-3\log{2}+\log{\pi}+3.570+C_0+\frac{8}{T-1}}{\log{T}}\right) \\
    <9.58869\log{T}.
\end{multline*}

Next we recognize that $u/\log{u}$ is decreasing until $u=e$ and after that it is increasing. By Lemma \ref{lem:Bchi} and since $2/\log{2}<3T/\log{(3T)}$, we obtain
\begin{multline*}
    \frac{ue}{\pi (T-1)\log{u}}\left(\vphantom{\frac{15}{2}\Re B(\chi)}3.578 \log^2(q T) + 48 \log^2(\log(q T)) + 26.208 \log(\log(q T)) \log(q T) 
    \right.\\ 
    \left.+  3\left(\frac{1}{(T-1)^2}+\frac{\pi}{4(T-1)}\right)-16.412\log(q T) -60.12 \log(\log(q T)) + 22.921+\frac{15}{8}\log\left(\frac{T^2}{4}+\frac{T}{2}+\frac{5}{2}\right)  \right.\\
    \left.+\frac{15}{4}\log \frac{q}{\pi}-\frac{15}{2}\Re B(\chi)\right) \\
    < \frac{3Te}{\pi (T-1)\log{(3T)}}\left(\vphantom{\frac{15}{2}\Re B(\chi)}3.578\cdot4 \log^2(T) + 48 \log^2(2\log{T}) + 2\cdot26.208 \log(2\log{T}) \log{T} 
    \right.\\ 
    \left.+\frac{15}{4}\log{T}+\frac{3\sqrt{q}\log^2{q}}{40}+\frac{15\cdot2.288}{2}\log^2{T}+125.561\log{T} \right) \\
    <\left(0.195\sqrt{q}+138.146\right)\log{T}.
\end{multline*}
Here we used the fact that if we divide the terms that do not contain $q$ in the fourth and the fifth line by $\log{T}$, then we get a decreasing function for all $T>4\cdot10^5$. Combining the estimates, we get the results.
\end{proof}

\section{Proofs of Theorems \ref{thm:main1}---\ref{thm:main4}}
\label{sec:ProofsMain}

\begin{proof}[Proof of Theorems \ref{thm:main1} and \ref{thm:main2}] These theorems are numerical versions of Propositions \ref{prop:firstMainS(x)} and \ref{prop:SxSecond}. We use the following values for constants in these propositions:
\begin{center}
\begin{tabular}{ |c|c|c|c|c| } 
 \hline
 $c_1$ & $c_2$ & $T_0$ & $T_1$ & $x_0$ \\ \hline
 $0$ & $53.989$ & $4\cdot10^5$ & $2\pi e+1$ & $2\pi e+1$\\ \hline
\end{tabular}
\end{center}
\begin{center}
\begin{tabular}{|c|c|c|c|c|c|c|c| } 
 \hline
 $d_1(1)$ & $d_2(1)$ & $d_3(1)$ & $d_4(1)$ & $d_5(1)$ & $d_6(1)$ & $d_7(1)$ & $d_8(1)$ \\ \hline
 $2\left(\left(\frac{5}{4}\right)^{1+\frac{1}{\log{2}}}+1\right)$ & $34.544$ & $8.022$ & $18.249$ & $\frac{1}{2\pi}$ & $0$ & $2.058$ & $0.04621$ \\ \hline
\end{tabular}
\end{center}
\begin{center}
\begin{tabular}{|c|c|c| } 
 \hline
$d_9$ & $d_{10}(1)$ & $d_{11}$ \\ \hline
 $3.523$ & $6.879$ & $1.03883$ \\ \hline
\end{tabular}
\end{center}
These values are available in Lemmas  \ref{lem:BellottiZeros}, \ref{lemma:explicitMangoldt}, \ref{lemma:principalrhoT}, \ref{lemma:rho2Zeta}
and Corollaries \ref{cor:numberOfZeros}, \ref{cor:zerosBetween}, \ref{cor:psiPrincipal}
\end{proof}

\begin{proof}[Proofs of Theorems \ref{thm:main3} and \ref{thm:main4}] We use Proposition \ref{prop:firstMain} and Remark \ref{rmk:primitiveS}. 
 Using Lemmas \ref{lemma:zeroFreeAll}, \ref{lemma:explicitMangoldt}, \ref{lemma:rho2Zeta} and \ref{lemma:rhosquares}, and Corollaries \ref{corollary:ZeroFree}, \ref{cor:numberOfZeros}, \ref{cor:zerosBetween}, \ref{cor:RhoNoChi}, \ref{cor:sumRhoRhoGen},  \ref{cor:psiPrincipal} and \ref{cor:psiPrimitive} we have the following values
\begin{center}
\begin{tabular}{ |c|c|c|c|c|c|c| } 
 \hline
 $c_1$ & $c_2$ & $c_{q,3}$ & $c_{q,4}$ & $T_0$ & $T_1$ & $x_0$ \\ \hline
 $10.5$ & $61.5$ & $\frac{100}{\sqrt{q}\log^2{q}}$ & $0.055\sqrt{q}\log{q}$ & $4\cdot10^5$ &$e^e$ & $e^e$ \\ \hline
\end{tabular}
\end{center}
\begin{center}
\begin{tabular}{|c|c|c|c|c|c|c| } 
 \hline
 $d_1(q)$ & $d_2(q)$ & $d_3(q)$ & $d_4(q)$ & $d_5(q)$ & $d_6(q)$ & $d_7(q)$ \\ \hline
 $2\left(\left(\frac{5}{4}\right)^{1+\frac{1}{\log{2}}}+1\right)$ &$34.544$ & $\log{q}+8.022$ & $0.195\sqrt{q}+147.735$ & $\frac{1}{2\pi}$ & $4.434\log{q}$ & $0.430$ \\ \hline
\end{tabular}
\end{center}
\begin{center}
\begin{tabular}{|c|c|c|c| } 
 \hline
 $d_8(q)$ & $d_9$ & $d_{10}(q)$ & $d_{11}$ \\ \hline
  $\frac{\sqrt{q}\log^2{q}}{100}+2.194 \log^2{q}+9.646\log{q}\log{\log{q}}$ & $1.325$ & $0.364$ & $1.03883$ \\ \hline
\end{tabular}
\end{center}
Note that in the term $2d_8(q)/\varphi_1^*(q)$ we can choose 
$$
d_8(q)=\frac{2\sqrt{q}\log^2{q}}{100\varphi_1^*(q)}+2.194 \log^2{q}+9.646\log{q}\log{\log{q}},
$$ 
since there are at most two primitive characters that are real. We use \eqref{eq:logqLowerNum} for the lower bound for $\log{x}$ and Lemma \ref{lemma:psi1Estimate} for the lower bound of $\varphi_1^*(q)$. Notice that the coefficient of $x^{B_q^*(x)+1}$ obtains its minimum at $q=4\cdot10^5+1$ and hence we get a constant coefficient. Similarly, we find the largest value in the second case, too.

In addition we note that since $q \geq 4\cdot10^5+1$, we can improve the constant $82.366$ to $25.090$. In both theorems, we can also improve the constant $2.1$ to $2.001$. We also note that 
$$
25.096x^{1+B_q^*(x)}\left(\frac{(\log\log{x})^2}{\log{x}}\right)^2+2.001x^{\frac{1+B_q^*(x)}{2}}\log{x}<x^{1+B_q^*(x)}
$$ 
for all $x$ satisfying the condition \eqref{eq:logqLowerNum}. Hence, according to the proof of Proposition \ref{prop:firstMain}, we can remove that term.
\end{proof}

\section*{Acknowledgements}

We  thank Daniel Johnston, Nicol Leong and Valeriia Starichkova for helpful 
discussions.
The work of Neea Paloj\"arvi was mostly carried out at the University of 
Helsinki and supported by the Emil Aaltonen Foundation.  

Le premiere auteur remercie le CDP C2EMPI pour son soutien, ainsi que
l’État Français dans le cadre du programme France-2030, l'Université de 
Lille, l'Initiative d'Excellence de l'Université de Lille, la Métropole
Européenne de Lille pour leur financement et leur appui au projet 
R-CDP-24-004-C2EMPI.

\bibliographystyle{abbrv}
\bibliography{references}

\newpage
\renewcommand\thesection{\Alph{section}}
\setcounter{section}{0}
\section{Appendix: Assumptions}
\label{appendix:assumptions}
Here $T_i, c_i, c_{q,i}, d_i(q), d_i $ and $q\geq 1$ are positive real numbers. The characters $\chi, \chi_1, \chi_2$ denote Dirichlet characters modulo $q$ and $\chi^*, \chi_1^*, \chi_2^*$ denotes the primitive character that induce $\chi, \chi_1, \chi_2$, respectively. For the definitions of $\delta_1(\chi)$, $\psi(u,\chi)$ and $C(\chi^*)$ we refer to the beginning of Section \ref{sec:preliminarylemmas}.
\begin{enumerate}[label=\textbf{A.\arabic*}]
    \item \label{eq:ZeroFree}
    For $c_1,c_2$ and $|\Im(s)| \geq T_1$ the function $L(s,\chi)$ does not have any non-trivial zeros in the set 
    \begin{equation}
    \label{set:zerofree}
        \left\{s: \frac{1}{2}\leq \Re(s)\leq 1,\Re(s)\geq 1- \frac{1}{c_1\log{q}+c_2(\log{|\Im(s)|})^{2/3}(\log{|\Im(s)|})^{1/3}}\right\}.
    \end{equation} 
    \item \label{eq:ZeroFree3}
    For $c_1,c_2$ the function $L(s,\chi)$ has at most one non-trivial zero ${\rho_\chi}=\beta_1\in \mathbb{R}$ in the set \eqref{set:zerofree}.
    \item \label{assumptionProductExceptional} Let $L_q(s)=\prod_{\chi \pmod q} L(s,\chi)$. Assume that $L_q$ has at most one zero ${\rho_\chi}=\beta_1 \in \mathbb{R}$ inside the set \eqref{set:zerofree}. Assume also that this zero does not arise for principal character modulo $q$. Suppose also that $\beta_1 \leq 1-1/c_{q,3}$ for some positive constant $c_{q,3}$ that may depend on $q$. If there
    is no exceptional zero in the set \eqref{set:zerofree} if $|s|<x$, we choose $c_{q,3}=c_1\log{q}+c_2(\log{x})^{2/3}(\log{\log{x}})^{1/3}$. 
    \item \label{eq:assumptionLogDer} Let $\chi\neq \chi_0$ be a character $\pmod q$ and $\beta_1\neq 1$ a possible exceptional zero inside the set \eqref{set:zerofree}. Assume
     \begin{equation*}
        \left|\frac{L'}{L}\left(1,\overline{\chi^*}\right)-\frac{\delta_1(\chi^*)}{1-\beta_1}\right|\leq c_{q,4}\log{q}.
    \end{equation*}
    \item    \label{assumption:psiuchi}
    For all $u \geq 2$, $T\geq \max\{q,T_0\}$ and $u \leq 3T$, we have
    \begin{equation*}
    \left|\psi(u,\chi)-\delta_0(\chi) u-C(\chi^*)+\sum_{|\Im({\rho_\chi})|<T} \frac{u^{\rho_\chi}}{{\rho_\chi}} \right |
    \leq d_1(q)\frac{u(\log{u})(\log\log{u})}{T}+d_2(q)\frac{u\log{u}}{T}+d_3(q)\log{u}+d_4(q)\log{T}. 
    \end{equation*}
\item \label{eq:assumption1OverRhoSum} 
For all $T \geq T_0$, we have
\begin{equation*}
      \sum_{\substack{\rho_\chi \neq 1-\beta_1 \\ |\gamma_\chi| \leq T}} \frac{1}{\left|\rho_\chi\right|} \leq d_{5}(q)\log^2{T}+d_{6}(q)\log{T},
    \end{equation*}
    where the sum runs over non-trivial zeros of $L(s,\chi)$ with additional assumptions.
\item \label{eq:rhoSecondUpper}
For all $T \geq T_1$, we have
\begin{equation*}
    \sum_{\substack{\rho_\chi \\ |\gamma_\chi|>T}} \frac{1}{\left|\rho_\chi\right|^2}\leq \frac{d_{7}(q)\log{\left(\frac{qT}{2\pi}\right)}}{T},
\end{equation*}   
where the sum runs over non-trivial zeros of $L(s,\chi)$ with additional assumptions. 
    where the sum runs over non-trivial zeros of $L(s,\chi)$ with additional assumptions.
    \item \label{eq:assumptionrhorho1}
    For $T \geq T_0$, we have
    \begin{equation*}
    \sum_{\substack{|\rho_\chi }|<T} \frac{1}{\left|\rho_\chi(1+\rho_\chi)\right|}\leq d_{8}(q)
\end{equation*}
 where the sum runs over non-trivial zeros of $L(s,\chi)$. 
 \item \label{eq:defIntervalZeroGeneral}
Assume that for all $T \geq T_1$, we have
\begin{equation*}
        \sum_{\substack{\rho_\chi \\ T \leq |\gamma_\chi|\leq T+1}} 1 \leq d_{9}\log{\left(q(T+1)\right)},
    \end{equation*} 
  \item \label{eq:defzerosGeneral}
    For all $T \geq T_1$ we have
    \begin{equation*}
    N(T,\chi) \leq d_{10}(q)T\log{\frac{qT}{2\pi e}}.
    \end{equation*}
    \item \label{eq:assumptionLambda}
    For $T \geq \max\{q, T_0\}$, we have 
     $\sum_{n \leq T} \Lambda(n) \leq d_{11} T$, where $\Lambda(n)$ denotes the von Mangoldt function.
\end{enumerate}

\section{Appendix: Functions $f_j$}
\label{appendix:f}
Here we give notation for functions $f_j$ that are used in the results. For the terms $T_i, c_i, c_{q,i}, d_i(q), d_i$ we refer to Appendix \ref{appendix:assumptions}. The function $\mathfrak{S}_q(c)$ is given as in \eqref{def:Gqc}. 

\allowdisplaybreaks
\begin{align*}
    &f_1(q,x_0):=\frac{\pi^2 (d_{5}(q))^2\left(2x_0+1\right)}{2\log{2}}, \\
    & f_2(q,x_0):=\pi\left(1.2\cdot\left(4+\frac{1}{x_0}\right)\left(1+\frac{\log{1.2}}{\log\log{x_0}}\right)d_1(q)+\frac{1.2d_2(q)}{\log\log{x_0}}\left(4+\frac{1}{x_0}\right) \right. \\
    &\quad\quad\left.+\frac{d_3(q)}{\log\log{x_0}}\left(2+\frac{\log{(2(2+x_0^{-1}))}}{\log{x_0}}\right)+\frac{2d_4(q)}{\log\log{x_0}}\right)^2, \\
    & f_3(q,T_1,x_0):=\sqrt{2}\left(\left(2+\frac{1}{x_0}\right)^{3}-1\right)\left(2d_{7}(q)+d_{10}(q)\right)\left(\left(4+\frac{11.9}{\log{x_0}}+\frac{6.3}{(\log{x_0})^2}+\frac{\log{x_0}+\log{2}}{x_0(\log{x_0})^2}\right)d_{9} \right. \\
    &\quad\quad\left.+\frac{2}{(\log{x_0})^2}d_{10}(q)T_1\log{\frac{qT_1}{2\pi e}}\right), \\
    & f_4(q,T_1,x_0):=\left(2x_0+1\right)\left(\frac{d_1(q) \log\log{x_0}}{\sqrt{x_0}\log{x_0}}+\frac{d_2(q)+d_3(q)+d_4(q)+c_{q,4}+1}{\sqrt{x_0}\log{x_0}}+2d_{5}(q)+\frac{d_{6}(q)+1}{\log{x_0}}\right)  \\
        & \quad\quad \cdot\frac{\pi^2}{2}\left(\frac{d_1(q)\log\log{x_0}}{\sqrt{x_0}}+\frac{d_2(q)+d_3(q)+d_4(q)+c_{q,4}+1}{\sqrt{x_0}}+d_{6}(q)+1\right) \\
        &\quad\quad+\frac{6 \pi\log\log{x_0}}{\log{x_0}}f_2(q,x_0)+3\pi^2 f_3(q,T_1,x_0), \\
    &  f_5(q,T_1,x_0):=2\pi^2\left(\frac{\log\log{x_0}}{(\log{x_0})^2 \sqrt{\pi}}\sqrt{3f_2(q, x_0)}+\sqrt{\frac{3f_3(q,T_1,x_0)}{(\log{x_0})^3}}\right)+\frac{\pi^2}{2(\log{x_0})^3} \\
        & \quad\quad+\frac{2\left(2x_0+1\right)\pi^2}{\sqrt{x_0}\log{x_0}}\left(\frac{d_1(q)\log\log{x_0}+d_2(q)+d_3(q)+d_4(q)+c_{q,4}+1}{\sqrt{x_0}\log{x_0}}+d_{5}(q)+\frac{d_{6}(q)}{\log{x_0}}+\frac{1}{\sqrt{x_0}(\log{x_0})^2}\right),  \\ 
    & f_6(q,x_0):=d_1(q)+\frac{d_2(q)}{2\log\log{x_0}}+\frac{d_3(q)+d_4(q)+c_{q,4}+1}{\log\log{x_0}}+\frac{4d_{8}(q)}{x_0(\log{x_0})(\log\log{x_0})}+\frac{d_{11}}{(\log{x_0})(\log\log{x_0})} \\ 
    &  f_7(q, T_1,x_0):=0.35x_0f_1(q,x_0)+\left(1.443+\frac{0.5x_0}{\log{x_0}}\right)\left(f_4(q,T_1,x_0)+\frac{2f_5(q,T_1,x_0)}{\varphi(q)}\right) \\
     &\quad\quad    +\frac{\log\log{x_0}}{\varphi(q)(\log{x_0})^3}\left(\frac{2c_{q,3}+1.5}{\varphi(q)(\log{x_0})(\log\log{x_0})}+2f_6(q, x_0)\right) \\
     & \text{ and }\\
       & f_8(q,T_1,x_0,c):=\frac{2(\log\log{x_0})f_6(q,x_0)/\varphi(q)+d_{11}\log{q}/\log{2}+17.314+1.624/\log{x_0}}{(\log{x_0})^4}+\frac{4c_{q,3}+3}{2(\log x_0)^5\varphi(q)} \\
        &\quad\quad +\frac{1}{\log{x_0}}\left(0.35x_0f_1(q,x_0)+\left(1.443+\frac{0.5x_0}{\log{x_0}}\right)\left(f_4(q,T_1,x_0)+\frac{f_5(q,T_1,x_0)}{\varphi(q)}\right)\right)+f_1(q,x_0).
\end{align*}

In the case of $S(x)$ we have a slightly different versions for the some of the functions above:
\begin{align*}
& f_{4,\zeta}(q,T_1,x_0):=\left(2x_0+1\right)\left(\frac{d_1(q) \log\log{x_0}}{\sqrt{x_0}\log{x_0}}+\frac{d_2(q)+d_3(q)+d_4(q)}{\sqrt{x_0}\log{x_0}}+2d_{5}(q)+\frac{d_{6}(q)}{\log{x_0}}+\frac{\log{(2\pi)}}{\sqrt{x_0}(\log{x_0})^2}\right)  \\
        & \quad\quad \cdot\frac{\pi^2}{2}\left(\frac{d_1(q)\log\log{x_0}}{\sqrt{x_0}}+\frac{d_2(q)+d_3(q)+d_4(q)}{\sqrt{x_0}}+d_{6}(q)+\frac{\log{(2\pi)}}{\sqrt{x_0}\log{x_0}}\right) \\
        &\quad\quad+\frac{6 \pi\log\log{x_0}}{\log{x_0}}f_2(q,x_0)+3\pi^2 f_3(q,T_1,x_0) \\
    & f_{5,\zeta}(q,T_1,x_0):= {2\pi^2}\left(\frac{\log\log{x_0}}{(\log{x_0})^2 \sqrt{\pi}}\sqrt{3f_2(q, x_0)}+\sqrt{\frac{3f_3(q,T_1,x_0)}{(\log{x_0})^3}}\right)+\frac{\pi^2}{2(\log{x_0})^3} \\
        & \quad\quad+\frac{2\left(2x_0+1\right)\pi^2}{\sqrt{x_0}\log{x_0}}\left(d_{5}(q)+\frac{d_{6}(q)}{\log{x_0}}+\frac{d_1(q)\log\log{x_0}+d_2(q)+d_3(q)+d_4(q)}{\sqrt{x_0}\log{x_0}}+\frac{\log{(2\pi)}+1}{\sqrt{x_0}(\log{x})^2}\right), \\
& f_{6, \zeta}(q,x_0):=2d_1(q)+\frac{d_2(q)+2(d_3(q)+d_{4}(q))}{\log\log{x_0}}+\frac{8d_8(q)}{x_0(\log{x_0})(\log\log{x_0})}+\frac{2d_{11}+2\log{(2\pi)}+1.5}{(\log{x_0})(\log\log{x_0})} \\
&\text{and} \\
& f_{7,\zeta}(q, T_1,x_0):=0.35x_0f_1(q,x_0)+\left(1.443+\frac{0.5x_0}{\log{x_0}}\right)\left(f_{4,\zeta}(q,T_1,x_0)+f_{5,\zeta}(q,T_1,x_0)\right) \\
        &\quad\quad+\frac{\log\log{x_0}}{(\log{x_0})^3}\left(\frac{1.5}{(\log{x_0})(\log\log{x_0})}+2f_{6,\zeta}(q, x_0)\right).
\end{align*}
\end{document}